\documentclass{amsart}

\usepackage{amsmath}
\usepackage{amsfonts}
\usepackage{amssymb}
\usepackage{amsthm}
\usepackage{mathrsfs}
\usepackage{amsaddr}
\usepackage{yfonts}
\usepackage{enumitem}
\usepackage{tikz}
\usepackage{subcaption}
\usepackage{comment}
\usepackage{tikz-cd}
   
\usetikzlibrary{positioning, fit, calc, shapes, arrows, backgrounds}

\newtheorem{theorem}{Theorem}[section]
\newtheorem{claim}[theorem]{Claim}
\newtheorem{lemma}[theorem]{Lemma}
\newtheorem{proposition}[theorem]{Proposition}
\newtheorem{corollary}[theorem]{Corollary}

\newtheorem*{remark}{Remark}

\newtheorem{question}[theorem]{Question}
\newtheorem{definition}[theorem]{Definition}
\newtheorem*{acknowledgements}{Acknowledgements}
\newcommand{\foralmostall}{\forall^\infty}
\newcommand{\existsinfty}{\exists^\infty}

\newcommand{\bbb}{\mathfrak{b}}
\newcommand{\ddd}{\mathfrak{d}}
\newcommand{\sss}{\mathfrak{s}}
\newcommand{\rrr}{\mathfrak{r}}

\newcommand{\continuum}{\mathfrak{c}}
\newcommand{\Borel}{\mathrm{Borel}}

\newcommand{\Meager}{\mathcal{M}}
\newcommand{\Null}{\mathcal{N}}
\newcommand{\E}{\mathcal{E}}
\newcommand{\J}{\mathcal{J}}
\newcommand{\I}{\mathcal{I}}
\newcommand{\HH}{\mathcal{H}}

\newcommand{\Baire}{\omega^{\omega}}
\newcommand{\Cantor}{2^{\omega}}
\newcommand{\baire}{\omega^{<\omega}}
\newcommand{\cantor}{2^{<\omega}}
\newcommand{\PIF}{\Pi F}

\newcommand{\MMM}{\mathbb{M}}

\newcommand{\B}{\mathbb{B}}

\newcommand{\NA}{\mathcal{NA}}
\newcommand{\MA}{\mathcal{MA}}
\newcommand{\EA}{\mathcal{EA}}
\newcommand{\EM}{\langle\mathcal{E},\mathcal{M}\rangle}
\newcommand{\EN}{\langle\mathcal{E},\mathcal{N}\rangle}
\newcommand{\Nstar}{\mathcal{N}^{*}}
\newcommand{\Mstar}{\mathcal{M}^{*}}
\newcommand{\Estar}{\mathcal{E}^{*}}
\newcommand{\SMZ}{\mathcal{SMZ}}

\newcommand{\BC}{\mathrm{BC}}

\newcommand{\forces}{\Vdash}

\newcommand{\cov}{\mathrm{cov}}
\newcommand{\non}{\mathrm{non}}
\newcommand{\add}{\mathrm{add}}
\newcommand{\cof}{\mathrm{cof}}

\title{Combinatorics of translations of meager and closed measure zero sets}
\author{Aleksander Cieślak}
\address{Faculty of Pure and Applied Mathematics, Wrocław University of Science and Technology, Wybrzeże Stanisława Wyspiańskiego 27, 50-370 Wrocław, Poland}
\email{aleksander.cieslak@pwr.edu.pl}
\date{}

\makeindex

\begin{document}
    \begin{abstract}
        We study combinatorial properties of translations of meager and closed sets of measure zero. We discuss constellations of Borel conjectures for related classes of small sets and show that there are no uncountable null-additive sets in the Miller model. We also study cardinal invariants of those classes. In particular, we investigate the cardinal invariants of $\sigma$-ideals $\mathcal{H}_{F}$ related to meager-additive sets. We obtain a new characterization of additivity of meager ideal and answer some questions from \cite{CardonaMA}. We also show that $\non(\mathcal{E}^{*})$, that is equal to the translation version of the covering number of the ideal $\E$, is close to the cardinal invariant related to the Laver property. This strengthens a result of Bartoszyński and Judah from \cite{BartJudahBorelImages} and a result of Elekes and Stepr\=ans from \cite{ElekesSteprans}. Finally, we show that every $\mathcal{E}$-Luzin set is in $\mathcal{E}^{\star}$.
    \end{abstract}
    \maketitle


    \section{Introduction}
    Let $\Meager$ and $\Null$ denote the $\sigma$-ideals on $\Cantor$ of meager set and sets of Lebesgue measure zero, respectively. Also, let $\E$ denote the $\sigma$-ideal generated by closed measure zero subsets of $\Cantor$. It is well-known that $\E\subseteq\Meager\cap \Null$ and that the inclusion is proper. The Cantor space $\Cantor$ equipped with addition ( modulo 2 on coordinates) forms a Polish group and all three ideals are translation invariant. For two sets of reals $X,Y\subseteq\Cantor$ and $z\in\Cantor$ we let $X+z=\{x+z:x\in X\}$ and $X+Y=\{x+y:x\in X$ and $y\in Y\}$. The following definition is central to this article.
    \begin{definition}
        Let $\I$ be a $\sigma$-ideal on $\Cantor$. Define then
        \begin{itemize}
            \item[--] $\mathcal{IA}=\{X\subseteq\Cantor:\forall F\in\I$ $X+F\in\I\}$,
            \item[--] $\I^{*}=\{X\subseteq\Cantor:\forall F\in\I$ $X+F\neq\Cantor\}$
        \end{itemize}
    \end{definition}
    Members of $\mathcal{IA}$ are referred to as $\I$-additive sets. Also, members of $\MA$ or $\NA$ are known as \emph{meager-additive} and \emph{null-additive} sets. Members of $\I^{*}$ are sometimes referred to as $\I$-shiftable (see \cite{WohoPHD} or \cite{BreWoh}). It is clear from the definitions that $\mathcal{IA}\subseteq\I^{*}$ for any $\sigma$-ideal $\I$. Many researchers investigated the null-additive and meager-additive sets (see \cite{BartoszynskiBD}, \cite{Calderon}, \cite{CardonaMA}, \cite{coveringstrongmeasurecanbeaboveeverythingelse}, \cite{MejiaDirectedSums}, \cite{MejiaCardonaMA}, \cite{NowikWeiss}, \cite{KysiakNowikWeiss}, \cite{WeissSzewczak}, \cite{Weiss2013}, \cite{Weiss2018}, \cite{WeissNew}, \cite{ZindulkaMAEA}, \cite{ZindulkaMeagerAdditiveinTopGroups} among many others). Pawlikowski asked, if every null-additive set is meager-additive. This question was answered in positive by Shelah, who gave the following combinatorial characterizations of these two classes.
    \begin{theorem}(Shelah; \cite{ShelahNAMA})\label{ShelahNAMA}
        A set $X\subseteq\Cantor$ is null-additive if and only if for any interval partition $(I_{n})_{n\in\omega}$ there is a slalom $S=(S_{n})_{n\in\omega}$ with $S_{n}\in[2^{I_{n}}]^{n}$ for $n\in\omega$ such that 
        \begin{center}
            $X\subseteq\{x\in\Cantor:\foralmostall n\in\omega$ $x|_{I_{n}}\in S_{n}\}$
        \end{center}
        A set $X\subseteq\Cantor$ is meager-additive if and only if for any interval partition $(I_{n})_{n\in\omega}$ there is a real $y\in\Cantor$ and an interval partition $(J_{n})_{n\in\omega}$ that dominates $(I_{n})_{n\in\omega}$ such that
        \begin{center}
            $X\subseteq\{x\in\Cantor:\foralmostall n\in\omega$ $\exists k\in\omega$ that $I_{k}\subseteq J_{n}$ and $x|_{I_{k}}= y|_{I_{k}}\}$
        \end{center}
        In particular, $\NA\subseteq\MA$ holds.
    \end{theorem}
    With regards to the $\sigma$-ideal $\E$, in \cite{NowikWeiss} Nowik and Weiss asked for connection between $\MA$ and $\EA$.
    In \cite{ZindulkaMAEA} and \cite{ZindulkaMeagerAdditiveinTopGroups} Zindulka investigated meager-additive sets in Polish groups and showed that the classes $\MA$ and $\EA$ coincide. 
    \begin{theorem}(Zindulka; \cite{ZindulkaMAEA})
        $\MA=\EA$
    \end{theorem}
    For meager sets, the family $\Mstar$ was historically the first class investigated many researchers under the following  definition.
    \begin{definition}
        A set $X\subseteq\Cantor$ is a \emph{strongly measure zero} set if for every increasing sequence $(f_{n})_{n\in\omega}\in\Baire$ there is a sequence $(\sigma_{n})_{n\in\omega}\subseteq \cantor$ with $|\sigma_{n}|=f_{n}$ for all $n\in\omega$ and such that 
        \begin{center}
            $X\subseteq\{x\in\Cantor:\exists^{\infty}n\in\omega$ $\sigma_{n}\subseteq x\}$
        \end{center}
        The collection of all strongly measure zero sets will be denoted by $\mathcal{SMZ}$.
    \end{definition}
    It is not difficult to see that strong measure zero sets form a $\sigma$-ideal. The first connection of strong measure zero sets to additive classes were given by the following theorem of Galvin, Mycielski, Solovay.
    \begin{theorem}(Galvin, Mycielski, Solovay; \cite{GalvMycielSolov})
        $\mathcal{SMZ}=\Mstar$
    \end{theorem}
    This theorem raised questions about the dual collection - $\Nstar$ - that is known as strongly meager sets and will be denoted by $\mathcal{SM}$. As in the case of strong measure zero sets, the existence of uncountable meager additive set is undecidable in ZFC. Also, in \cite{BartShNstarIsNotIdeal} Bartoszyński and Shelah showed that under assumption of Continuum Hypothesis, $\mathcal{SM}$ does not form an ideal. We will introduce one more additive collection which is a generalization of notion of $\I$-additive sets.
    \begin{definition}
        Let $\I\subseteq\J$ will be two arbitrary $\sigma$-ideals on $\Cantor$. Define
        \begin{center}
            $\langle\I,\J\rangle=\{X\subseteq\Cantor:\forall F\in\I$ $X+F\in\J\}$,
        \end{center}
    \end{definition}
    Clearly, $\mathcal{IA}=\langle\I,\I\rangle$ for any $\sigma$-ideal $\I$. Such classes appeared in work of Pawlikowski in which he characterized the strongly measure zero sets using translations of closed measure zero sets.
    \begin{theorem}(Pawlikowski; \cite{PawlikowskiEN})
        $\EN=\mathcal{SMZ}$
    \end{theorem}
    This characterization turned out to be useful in establishing consistency results about stongly measure sets (see \cite{BorelWithDualBorel}). With regards to the dual class $\EM$, originally Recław and later Bartoszyński and Judah, showed that in the dual of Pawlikowski's result only one implication holds.
    \begin{theorem}(Recław; see also Bartoszyński, Judah; \cite{BJ})
        $\mathcal{SM}\subseteq\EM$
    \end{theorem}
    We will later see that in the theorem above, consistently, the two collections are not equal (this holds for example in Cohen or Hechler model - see Theorem \ref{CutInCohenHechler}; one may also see \cite{Weiss2013}). It is also not difficult to verify the following general observation.
    \begin{proposition}
        The following holds:
        \begin{itemize}
            \item[--] $\langle\I,\J\rangle\subseteq\J$,
            \item[--] $\J_{0}\subseteq\J_{1}$ implies $\langle\I,\J_{0}\rangle\subseteq \langle\I,\J_{1}\rangle$,
            \item[--] $\I_{0}\subseteq\I_{1}$ implies $\langle\I_{1},\J\rangle\subseteq \langle\I_{0},\J\rangle$,
        \end{itemize}
    \end{proposition} The inclusions between additive classes involving $\Null$, $\Meager$ and $\E$ can be summarized by the following diagram (a similar diagram appears in \cite{Weiss2018}). The arrow points towards the larger collection.
    \begin{center}
\tikzset{
  net node/.style = {draw, circle, minimum size=10mm},
  net edge/.style = {->},
  net cut/.style = {shorten >=-10mm, shorten <=-10mm, rounded corners=10mm, color=red},
  net cross/.style = {sloped, allow upside down, pos=.3},
}
\begin{tikzpicture}
        \newcommand{\edge}[5][]{\draw[net edge, #1] (#3) -- coordinate[net cross, name=#2] node[pos=.7, auto]{#5} (#4);}

        \node[name=na] at ( 0, 0) {$\NA$};
        \node[name=sm] at ( 0, 3) {$\Nstar=\mathcal{SM}$};
        \node[name=ma]  at ( 4, 0) {$\MA=\EA$};
        \node[name=em] at ( 4, 3) {$\EM$};
        \node[align=center, name=smz] at ( 8, 0) {$\mathcal{SMZ}=\Mstar=\EN$};
        \node[name=e]  at ( 8, 3) {$\Estar$};

        \edge[swap] {e1}  {na} {ma} {}
        \edge       {e2}  {na} {sm} {}
        \edge[swap] {e3} {ma} {em} {}
        \edge       {e4} {ma}  {smz} {}
        \edge       {e5} {em}  {e} {}
        \edge       {e6} {smz}  {e} {}
        \edge       {e7} {sm}  {em} {}

\end{tikzpicture}
    \end{center}
    The paper is structured as follows:
    \begin{itemize}
        \item[--] In the first section we investigate different variants of Borel conjectures. First, we gather known results and then we show that adding a single Miller's real $\NA$-kills uncountable ground model sets of reals. As a result, we obtain that in Miller's model, null-additive sets are countable while meager-additive are not.
        \item[--] We start the second section with discussion on cardinal invariants (slalom numbers) related to this work. We then investigate cardinal invariants of the $\sigma$-ideals $\mathcal{H}_{F}$. We show that $\add(\Meager)=\min\{\add(\mathcal{H}_{F}):F\in\Baire\}$ and we discuss independence of $\cov(\mathcal{H}_{F})$ and $\non(\mathcal{H}_{F})$ from other classical invariants like splitting and reaping number or bounding and dominating number. We finish second section with discussion on the idealized forcing notion $\mathbb{P}_{\HH_{F}}=\mathcal{B}or/\mathcal{H}_{F}$.
        \item[--] In the third section we investigate the cardinal invariants of $\EM$ and $\Estar$. We estimate $\non(\Estar)$ with the slalom number related to the Laver property and we prove some bounds of the uniformity of $\EM$. We finish with estimations of the transitive invariants $\add^{*}_{t}(\E)$ and $\cof^{*}_{t}(\E)$.
        \item[--] In the last section, we show that any $\E$-Luzin set is in $\E^{\star}$ and has Menger property.
    \end{itemize}

    \subsection{Cardinal invariants of the continuum}
    In this article we will be mostly interested in the Polish spaces $\Cantor$ and $\Baire$. For increasing $F\in\Baire$ by $\PIF$ we will denote the space of those functions in $\Baire$ that are below $F$, i.e. $\PIF=\Pi_{n\in\omega}F(n)$.\\
     If $\I$ is a $\sigma$-ideal on a Polish space $X$, then the standard cardinal characteristics of $\I$ are
\begin{itemize}
    \item[--] $\mathrm{add}(\I)=\min\{|\mathcal{F}|:\mathcal{F}\subseteq\I$ and $\bigcup\mathcal{F}\notin\I\}$,
    \item[--] $\mathrm{cov}(\I)=\min\{|\mathcal{F}|:\mathcal{F}\subseteq\I$ and $\bigcup\mathcal{F}=X\}$,
    \item[--] $\mathrm{non}(\I)=\min\{|Y|:Y\subseteq X$ and $X\notin\I\}$,
    \item[--] $\mathrm{cof}(\I)=\min\{|\mathcal{F}|:\mathcal{F}\subseteq\I$ and $\forall Y\in\I$ $\exists Y'\in\mathcal{F}$ $Y\subseteq Y'\}$.
\end{itemize}
These are known as additivity, covering, uniformity and cofinality of the ideal $\I$, respectively. In the case when $X$ is a Polish group, we will also be interested in transitive versions of the above coefficients.
\begin{itemize}
    \item[--] $\mathrm{add}_{t}(\I)=\min\{|X|:X\subseteq\Cantor$ and $X+I\notin\I$ for some $I\in\I\}$,
    \item[--] $\mathrm{add}_{t}(\I,\J)=\min\{|X|:X\subseteq\Cantor$ and $X+I\notin\J$ for some $I\in\I\}$,
    \item[--] $\mathrm{cov}_{t}(\I)=\min\{|\mathcal{F}|:X\subseteq\Cantor$ and $X+I=\Cantor$ for some $I\in\I\}$,
\end{itemize}
Clearly, $\add(\I)\leq\add_{t}(\I)$, $\cov(\I)\leq\cov_{t}(\I)$. It is also easy to see the following connection between the transitive coefficients and uniformities of additive classes.
\begin{proposition}
The following holds:
    \begin{itemize}
        \item[--] $\non(\mathcal{IA})=\add_{t}(\I)$,
        \item[--] $\non(\I^{*})=\cov_{t}(\I)$,
        \item[--] $\non(\langle\I,\J\rangle)=\add_{t}(\I,\J)$.
    \end{itemize}
\end{proposition}
For the $\sigma$-ideals $\Null$ and $\Meager$, these invariants were investigated by Pawlikowski in \cite{PawlikowskiTransitiveInvariants} and will be studied in the section \ref{SectionCardInvsHf} of this article.\\

For $f,g\in\omega^{\omega}$ we write $f\leq^{*} g$ if $g$ dominates $f$ i.e. $f(n)\leq g(n)$ for almost all $n\in\omega$. We will call $\mathcal{F}\subseteq\omega^{\omega}$ a dominating family if for every $g\in\omega^{\omega}$ there is $f\in\mathcal{F}$ which dominates $g$. We will call $\mathcal{F}\subseteq\omega^{\omega}$ an unbounded family if there is no $g\in\omega^{\omega}$ which dominates every $f\in\mathcal{F}$. The bounding and dominating numbers are defined as follows.
\begin{itemize}
    \item[--] $\bbb$ is defined as a minimal cardinality of an unbounded family,
    \item[--] $\ddd$ as a minimal cardinality of an dominating family.
\end{itemize}
For two sets $A,B\in[\omega]^{\omega}$, we will say that $A$ splits $B$ if both $A\cap B$ and $B\setminus A$ are infinite. Family $\mathcal{A}\subseteq[\omega]^{\omega}$ is a splitting family if for every $B\in[\omega]^{\omega}$ there is $A\in\mathcal{A}$ that splits $B$. Family $\mathcal{A}\subseteq[\omega]^{\omega}$ is a reaping (unsplit) family if there is no single $B\in[\omega]^{\omega}$ that splits all elements of $\mathcal{A}$. The splitting and reaping numbers are defined as follows. 
\begin{itemize}
    \item[--] $\sss$ is defined as a minimal cardinality of a splitting family,
    \item[--] $\rrr$ as a minimal cardinality of a reaping family.
\end{itemize}
It is well-known, that the inequalities $\sss\leq\ddd,\non(\Null),\non(\Meager)$ and $\cov(\Null),\cov(\Meager),\bbb\leq\rrr$ hold and $\sss$ together with $\rrr$ are mutually incomparable.\\

The cardinal invariants of $\E$ were investigated by Bartoszyński and Shelah.
\begin{theorem}(Bartoszyński, Shelah; \cite{BartShClosedMeasureZeroSets})\label{CardInvsofE}
    The following (in)equalities hold:
    \begin{itemize}
        \item[--] $\add(\E)=\add(\Meager)$ and $\cof(\E)=\cof(\Meager)$,
        \item[--] $\cov(\Meager),\cov(\Null)\leq\cov(\E)\leq\max\{\ddd,\cov(\Null)\}$
        \item[--] $\min\{\bbb,\non(\Null)\}\leq\non(\E)\leq\non(\Meager),\non(\Null)$
    \end{itemize}
\end{theorem}
It is also know that $\cov(\E)\leq\rrr$ and $\sss\leq\non(\E)$. The ZFC-provable inequalities between these invariants are summarized in the following diagram (the reader may consult \cite{BJ}).
\begin{figure}[h!]
\centering
\begin{tikzpicture}[]
  \matrix[matrix of math nodes,column sep={30pt,between origins},row
    sep={30pt,between origins},nodes={asymmetrical rectangle}] (s)
  {
    &&&&&&&|[name=rrr]| \rrr\\
    &|[name=covn]| \mathrm{cov}(\Null) &&&&|[name=nonm]| \mathrm{non}(\Meager) &&&|[name=cofm]| \mathrm{cof}(\Meager) &&&&|[name=cofn]| \mathrm{cof}(\Null)\\
    &&&&&&&|[name=covE]| \cov(\E)\\
    &&&&&|[name=b]| \bbb &&&|[name=d]| \ddd \\
    &&&&&&|[name=nonE]| \non(\E) \\
    &|[name=addn]| \mathrm{add}(\Null) &&&&|[name=addm]| \mathrm{add}(\Meager) &&&|[name=covm]| \mathrm{cov}(\Meager) &&&&|[name=nonn]| \mathrm{non}(\Null) \\
    &&&&&&|[name=sss]| \sss \\
  };
    \draw[->] 
            (addn) edge (addm)
            (addm) edge (covm)
            (covm) edge (nonn)
            (addm) edge (b)
            (b) edge (d)
            (covm) edge (d)
            (covn) edge (nonm)
            (addn) edge (covn)
            (b) edge (nonm)
            (d) edge (cofm)
            (nonm) edge (cofm)
            (cofm) edge (cofn)
            (nonn) edge (cofn)

            (covn) edge [bend right=10] (covE)
            (covm) edge [bend left=10] (covE)
            (covE) edge (cofm)

            (nonE) edge [bend right=10] (nonm)
            (nonE) edge [bend left=10] (nonn)
            (addm) edge (nonE)

            (covE) edge (rrr)
            (sss) edge (nonE)
            (b) edge (rrr)
            (sss) edge (d)
            ;
\end{tikzpicture}\caption{Cicho\'n's diagram with covering and uniformity of the ideal $\E$ and the splitting and reaping number. The arrows point towards the ZFC-provable larger cardinal invariant.}
\end{figure}
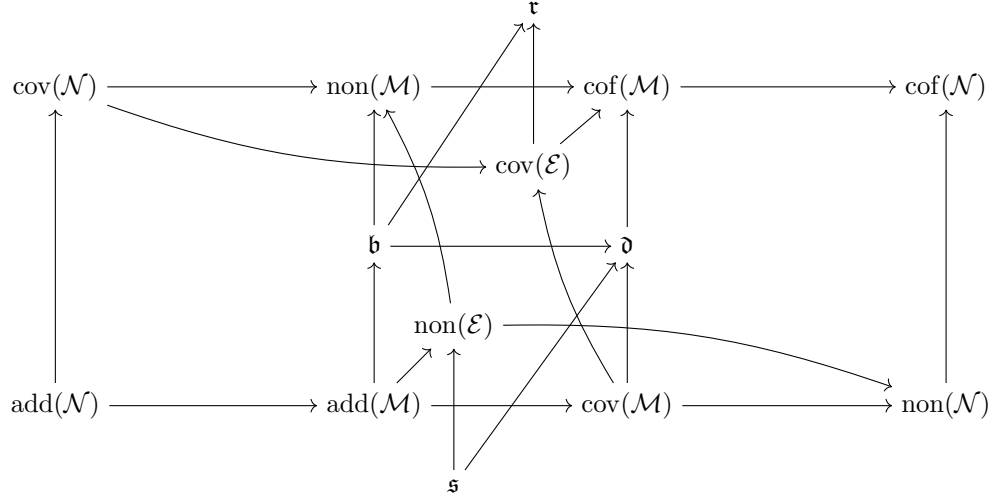

In this article, we will frequently use the following two combinatorial characterizations of the $\sigma$-ideals $\Meager$ and $\E$ (in \cite{BJ} see Theorem 2.2.4 and Theorem 2.3.6). The first one is often referred to as chopped real.
    \begin{proposition}(Talagrand \cite{Tala} or Theorem 2.2.4. in \cite{BJ})\label{CharacterizationM}
    For an interval partition $I=(I_{n})_{n\in\omega}$ and a real $y\in\Cantor$ define the set
    \begin{center}
        $M(y,I)=:\{x\in\Cantor:\foralmostall n\in\omega$ $x|_{I_{n}}\neq y|_{I_{n}}\}$
    \end{center}
    Every set $M(y,I)$ is meager and for every meager set $M\subseteq\Cantor$ there are an interval partition $I=(I_{n})_{n\in\omega}$ and a real $y\in\Cantor$ such that $M\subseteq M(y,I)$.
    \end{proposition}
The second one is provides a basis of the ideal $\E$ using interval partitons and slaloms.
\begin{proposition}(Bartoszyński and Shelah; \cite{BartShClosedMeasureZeroSets})\label{CharacterizationE}
    For an interval partition $I=(I_{n})_{n\in\omega}$ and a slalom $S=(S_{n})_{n\in\omega}$ where $S_{n}\subseteq 2^{I_{n}}$ and $\frac{|S_{n}|}{2^{|I_{n}|}}\leq\frac{1}{2^{n}}$ define the set 
    \begin{center}
        $E(S,I)=:\{x\in\Cantor:\foralmostall n\in\omega$ $x|_{I_{n}}\in S_{n}\}$
    \end{center}
    Every set $E(S,I)$ is an $F_{\sigma}$-set of measure zero and for every $E\in\E$ there are an interval partition $I=(I_{n})_{n\in\omega}$ and a slalom $S=(S_{n})_{n\in\omega}$ such that $E\subseteq E(S,I)$.
\end{proposition}

    \section{Distinguishing different Borel conjectures}
    This section is devoted to distinguishing 'Borel conjectures' for our additive classes i.e. showing that consistently some classes contains only countable sets while others do not.
\begin{definition}
    Let $\mathcal{X}$ be a collection of subsets of reals. We will write $\BC(\mathcal{X})$, if the class $\mathcal{X}$ consists only of countable sets.
\end{definition}
Thus $\BC(\mathcal{SMZ})$ is the classical Borel conjecture and $\BC(\mathcal{SM})$ is the dual Borel conjecture. We will be interested in finding different constellations of the diagram of our additive classes. The cuts in the following diagrams mean that the classes of the south-west part consist only of countable sets, while the classes of the north-east part contain uncountable sets.\\

The following is celebrated result of Laver, saying that the Borel conjecture is consistent with ZFC.
\begin{theorem}(Laver; \cite{LaverBorelConj})
    It is consistent that $\BC(\mathcal{SMZ})$ holds.
\end{theorem}
Galvin asked for the dual Borel conjecture i.e. is it consistent that $\mathcal{SM}$ is made only of countable sets. This question was answered in positive by Carlson, who proved that the dual Borel conjecture holds in Cohen model.
\begin{theorem}(Carlson; \cite{Carlson})\label{CarlsonDualBC}
        It is consistent that $\BC(\mathcal{SM})$.
\end{theorem}
    Later on, Pawlikowski strengthened this result by showing the following.
    \begin{theorem}(Pawlikowski; \cite{PawlikowskiDualBorel})\label{PawlikowskiDualBC}
If $\mathbb{P}$ is a forcing notion with precaliber equal $\omega_{1}$, then the finite support iteration of length $\omega_{2}$ of $\mathbb{P}$ forces the dual Borel conjecture.
\end{theorem}
It follows that the dual Borel conjecture holds not only in Cohen model but also in Hechler model. For some time it was an open problem whether Borel conjecture and dual Borel conjecture can hold simultaneously. It was settled with the following result.
\begin{theorem}(Goldstern, Kellner, Shelah, Wohofsky; \cite{BorelWithDualBorel})\label{BorelWithDualBorel}
    It is consistent that $\BC(\mathcal{SMZ})$ and $\BC(\mathcal{SM})$ hold simultaneously.
\end{theorem}
In case of the existence of uncountable sets in our additive classes, several results should be mentioned. Clearly, if $\add_{t}(\J)>\omega_{1}$ (in particular if $\add(\J)>\omega_{1}$) then every set $X\subseteq\Cantor$ of size $\omega_{1}$ is in $\mathcal{JA}$. Similarly, if $\cov_{t}(\J)>\omega_{1}$ (in particular if  $\cov(\J)>\omega_{1}$) then any $X\subseteq\Cantor$ of size $\omega_{1}$ is in $\J^{*}$. It follows that $\BC(\mathcal{JA})$ implies $\add(\J)=\add_{t}(\J)=\omega_{1}$ and $\BC(\mathcal{J}^{\star})$ implies $\cov(\J)=\cov_{t}(\J)=\omega_{1}$.
Shelah also showed that under the Continuum Hypothesis, there exist an uncountable null-additive set. Bartoszyński strengthened this with the following result saying that $\ddd=\omega_{1}$ is enough.
    \begin{theorem}(Bartoszyński; \cite{BartoszynskiBD})\label{BartoszynskiMANA}
        \begin{itemize}
            \item[--] $\bbb=\omega_{1}$ implies existence of an uncountable meager-additive set,
            \item[--] $\ddd=\omega_{1}$ implies existence of an uncountable null-additive set.
        \end{itemize}
    \end{theorem}    
As $\MA\subseteq\SMZ$, this result strengthens the previously known result of Rothberger saying that $\bbb=\omega_{1}$ implies the existence of an uncountable strong measure zero set. Bartoszyński's result gives us the following, the simplest, constellation of Borel Conjectures in which all classes contain uncountable sets.
\begin{theorem}
        It is consistent that $\neg \BC(\NA)$. In particular, the following constellation of Borel Conjectures holds:
        
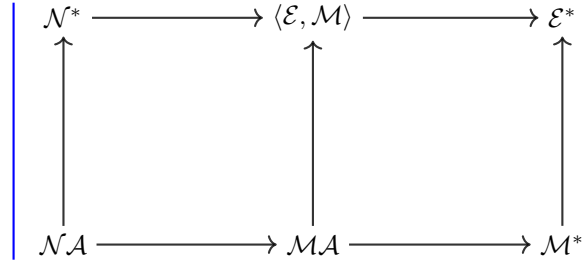
\begin{center}
		\begin{tikzpicture}[]
			\tikzset{
				textnode/.style={text=black}, 
			}
			\tikzset{
				edge/.style={color=black, thick, opacity=0.8}, 
			}
			\newcommand{\w}{3.3}
			\newcommand{\h}{2.0}
			
			\node[textnode] (NA) at (0,  0) {$\NA$};
            \node[textnode] (N) at (0,  1.5*\h) {$\Nstar$};
			\node[textnode] (MA) at (\w,  0) {$\MA$};
			\node[textnode] (EM) at (\w,  1.5*\h) {$\EM$};
			\node[textnode] (M) at (\w*2,  0) {$\Mstar$};
			\node[textnode] (E) at (\w*2,  1.5*\h) {$\Estar$};
			
			\draw[->, edge] (NA) to (MA);
			\draw[->, edge] (MA) to (EM);
			\draw[->, edge] (MA) to (M);
			\draw[->, edge] (EM) to (E);
            \draw[->, edge] (M) to (E);
            \draw[->, edge] (NA) to (N);
            \draw[->, edge] (N) to (EM);

			\draw[blue,thick] (-0.2*\w,-0.1*\h)--(-0.2*\w,1.6*\h);

		\end{tikzpicture}
\end{center}

    \end{theorem}
    \begin{proof}
        By Theorem \ref{BartoszynskiMANA}, any model satisfying $\ddd=\omega_{1}$ gives the constellation.
    \end{proof}
The following constellation, which is a consequence of results of Carlson, Pawlikowski and Bartoszyński mentioned before, distinguishes the versions of Borel conjectures for the classes $\Nstar$ and $\MA$. We have that.
    \begin{theorem}\label{CutInCohenHechler}
        It is consistent that $\BC(\Nstar)$, but $\neg \BC(\MA)$. In particular, the following constellation of Borel Conjectures holds:
        
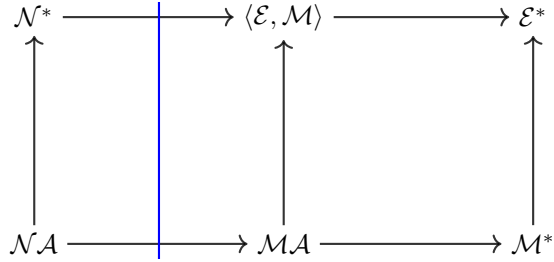
\begin{center}
		\begin{tikzpicture}[]
			\tikzset{
				textnode/.style={text=black}, 
			}
			\tikzset{
				edge/.style={color=black, thick, opacity=0.8}, 
			}
			\newcommand{\w}{3.3}
			\newcommand{\h}{2.0}
			
			\node[textnode] (NA) at (0,  0) {$\NA$};
            \node[textnode] (N) at (0,  1.5*\h) {$\Nstar$};
			\node[textnode] (MA) at (\w,  0) {$\MA$};
			\node[textnode] (EM) at (\w,  1.5*\h) {$\EM$};
			\node[textnode] (M) at (\w*2,  0) {$\Mstar$};
			\node[textnode] (E) at (\w*2,  1.5*\h) {$\Estar$};
			
			\draw[->, edge] (NA) to (MA);
			\draw[->, edge] (MA) to (EM);
			\draw[->, edge] (MA) to (M);
			\draw[->, edge] (EM) to (E);
            \draw[->, edge] (M) to (E);
            \draw[->, edge] (NA) to (N);
            \draw[->, edge] (N) to (EM);

			\draw[blue,thick] (0.5*\w,-0.1*\h)--(0.5*\w,1.6*\h);

		\end{tikzpicture}
\end{center}
    \end{theorem}
    \begin{proof}
        In Cohen and Hechler models, the $\Nstar$-sets are countable. For Cohen model this follows from Theorem \ref{CarlsonDualBC}, and for Hechler from Theorem \ref{PawlikowskiDualBC}. In Cohen model $\bbb=\omega_{1}$, thus by the Theorem \ref{BartoszynskiMANA} there exists uncountable meager-additive sets in that model. In Hechler model $\add(\Meager)=\omega_{2}$, so any set of size $\omega_{1}$ is meager-additive.
    \end{proof}
    Another constellation of Borel conjecture - which is a consequence of result of Judah, Shelah and Woodin - is the following.
    \begin{theorem}
        It is consistent that $\BC(\mathcal{SMZ})$, but $\neg \BC(\mathcal{SM})$. In particular, the following constellation of Borel Conjectures holds:
        
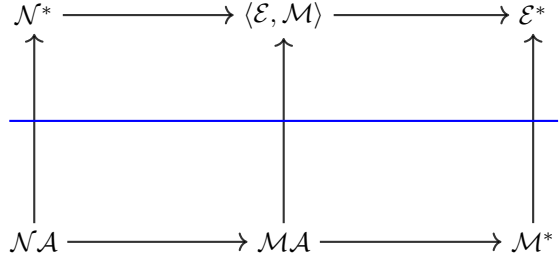
\begin{center}
		\begin{tikzpicture}[]
			\tikzset{
				textnode/.style={text=black}, 
			}
			\tikzset{
				edge/.style={color=black, thick, opacity=0.8}, 
			}
			\newcommand{\w}{3.3}
			\newcommand{\h}{2.0}
			
			\node[textnode] (NA) at (0,  0) {$\NA$};
            \node[textnode] (N) at (0,  1.5*\h) {$\Nstar$};
			\node[textnode] (MA) at (\w,  0) {$\MA$};
			\node[textnode] (EM) at (\w,  1.5*\h) {$\EM$};
			\node[textnode] (M) at (\w*2,  0) {$\Mstar$};
			\node[textnode] (E) at (\w*2,  1.5*\h) {$\Estar$};
			
			\draw[->, edge] (NA) to (MA);
			\draw[->, edge] (MA) to (EM);
			\draw[->, edge] (MA) to (M);
			\draw[->, edge] (EM) to (E);
            \draw[->, edge] (M) to (E);
            \draw[->, edge] (NA) to (N);
            \draw[->, edge] (N) to (EM);

			\draw[blue,thick] (-0.1*\w,0.8*\h)--(2.1*\w,0.8*\h);
            
	\end{tikzpicture}
\end{center}
    \end{theorem}
    \begin{proof}
        Judah, Shelah and Woodin showed (\cite{JudShWoodRandomVsBorelConj}) that adding any number of random reals by measure algebra to the model of Borel conjecture preserves Borel conjecture, so let $\mathbb{V}'$ be a model obtained by adding $\omega_{2}$-many random reals to the Laver model of Borel Conjecture. On the one hand, by the Judah, Shelah and Woodin result, $\Mstar=\mathcal{C}tbl$, on the other, $\cov(\Null)=\omega_{2}$ holds in $\mathbb{V}'$, so any set of size $\omega_{1}$ is strongly meager.
    \end{proof}
    In the rest of this section we will be interested in how adding Miller reals affects our additive classes. Recall the definition of Miller (super-perfect) forcing 
    \begin{definition}
    Let $\MMM$ denote the Miller forcing, i.e. collection of all trees $T\subseteq\baire$ such that every $\sigma\in T$ can be extended to a $\tau\in T$ such that the set of immediate successors $succ_{T}(\tau)=\{i\in\omega:\tau^{\frown}i\in T\}$ is infinite. The order is inclusion.
\end{definition}
It is well-known that Miller forcing is proper and adds a real in $\Baire$ which is not bounded by any real from the ground model. On the other hand, Miller forcing has the Laver property and preserves both outer measure and Baire category. It follows that the countable support iteration of Miller forcing over a model of the Continuum Hypothesis produces a model of $\non(\Null)=\non(\Meager)<\ddd$. For more applications of Miller forcing to the small sets of reals, the reader may also consult \cite{LyubomyrSzewczak}, \cite{ConcentratedGammaSetsInMillersModel}, \cite{ScalesProductsScheepersDiagram}, \cite{RepovsZdomskyy}, \cite{SelectionPrinciplesLaverMillerSacksModels}, \cite{ProductsOfMengerInMillerModel} or \cite{PreservationGammaSpaces}\\

The following lemma, which states that Miller forcing $\NA$-kills all uncountable ground model sets will be crucial for our distinction of Borel conjectures.
    \begin{lemma}
            For any uncountable $X\subseteq\Cantor$ we have $\MMM\forces"X\notin\NA"$.
        \end{lemma}
        \begin{proof}
        According to Shelah characterization of null-additive sets (see Theorem \ref{ShelahNAMA}), we have to show that after adding single Miller real there is an interval partition $I=(I_{n})_{n\in\omega}$ such that for every slalom $(S_{n})_{n\in\omega}$ with $S_{n}\in[2^{I_{n}}]^{n}$ for all $n\in\omega$ there is $x\in X$ such that $x|_{I_{n}}\notin S_{n}$ for infinitely many $n\in\omega$. The desired partition will be the one added by the Miller generic real: if $\dot{r}_{gen}\in\Baire$ is the generic real, then $\dot{I}=(\dot{I}_{n})_{n\in\omega}$ is such that $|\dot{I}_{n}|=\dot{r}_{gen}(n)+1$. Now, let $T_{0}\in\MMM$ be such that 
        \begin{center}
            $T_{0}\forces"\dot{S}_{n}=\{\dot{\sigma}^{n}_{1},\dot{\sigma}^{n}_{2},...,\dot{\sigma}^{n}_{n}\}\subseteq 2^{\dot{I}_{n}}$ for every $n\in\omega"$
        \end{center}
        Using standard fusion argument, it is possible to construct a Miller tree $T_{1}\leq T_{0}$ such that for every $\sigma\in split(T_{1})$ and $i\in succ_{T_{1}}(\sigma)$ there is a collection $\{\tau^{(\sigma,i)}_{1},\tau^{(\sigma,i)}_{2},...,\tau^{(\sigma,i)}_{|\sigma|}\}$ such that we have 
        \begin{center}
            $T_{1}|_{\sigma^{\frown}i}\forces"\forall k\leq |\sigma|$ $\dot{\sigma}^{|\sigma|}_{k}=\tau^{(\sigma,i)}_{k}"$
        \end{center}
        We need the following claim
        \begin{claim}
            There is $T_{2}\leq T_{1}$ such that for each $\sigma\in split(T_{2})$ there is $\{y^{\sigma}_{k}:k\leq|\sigma|\}\subseteq\Cantor$ such that the following condition holds
        \begin{center}
            $\forall N\in\omega$ $\foralmostall i\in succ_{T_{2}}(\sigma)$ we have $\forall k\leq|\sigma|$ $y^{\sigma}_{k}|_{N}\subseteq \tau^{(\sigma,i)}_{k}$
        \end{center}
        \end{claim}
        \begin{proof}
            We will show how to ensure the required condition on the stem of $T_{1}$. Propagating this construction over all splitnodes by the standard fusion argument produces the requierd tree $T_{2}$.\\
            Let $\sigma^{*}$ be the stem of $T_{1}$. By induction on $k\leq|\sigma^{*}|$ we construct $\{y^{\sigma^{*}}_{k}:k\leq |\sigma^{*}|\}\subseteq\Cantor$ and decreasing sequence of trees $T_{1}^{k}\leq T_{1}$, $k\leq|\sigma^{*}|$ as follows: assume that $\{y^{\sigma^{*}}_{l}:l< k\}$ and $T_{1}^{l}$ with $stem(T^{l}_{1})=\sigma^{*}$ have been constructed. We construct $T_{1}^{k}\leq T_{1}^{l}$ with the same stem and $y^{\sigma^{*}}_{k}$ such that for all $N\in\omega$ for almost all $i\in succ_{T^{k}_{1}}(\sigma^{*})$ we have  that $y^{\sigma^{*}}_{k}|_{N}\subseteq \tau^{(\sigma^{*},i)}_{k}$. To do this, find $A_{k}\in [succ_{T^{l}_{1}}(\sigma^{*})]^{\omega}$ such that for some $y\in\Cantor$ we have
            \begin{center}
            $\forall N\in\omega$ $\foralmostall i\in succ_{T^{k}_{1}}(\sigma^{*})$ $y|_{N}\subseteq \tau^{(\sigma^{*},i)}_{k}$
            \end{center}
            Let then $y^{\sigma^{*}}_{k}$ be such $y$ and let $T^{k}_{1}\leq T^{l}_{1}$ be such that $stem(T^{k}_{1})=\sigma^{*}$ and $succ_{T^{k}_{1}}(\sigma^{*})=A_{k}$. This finishes the induction step. When the induction is over, the tree $T^{|\sigma^{*}|}_{1}$ is the required one.
        \end{proof}
        Once the claim is proven, let $x\in X\setminus\{y^{\sigma}_{k}: \sigma\in split(T_{2})$ 
        and $k\leq|\sigma|\}$. We will now construct a tree $T_{3}\leq T_{2}$ such that $T_{3}\forces"\exists^{\infty}_{n\in\omega}$ $x|_{\dot{I}_{n}}\notin \dot{S}_{n}"$. For every $\sigma\in split(T_{2})$ let $N_{\sigma}\in\omega$ be such that $y^{\sigma}_{k}|_{N_{\sigma}}\neq x|_{N_{\sigma}}$ for every $k\leq|\sigma|$. Let $T_{3}\leq T_{2}$ be a Miller tree obtained from $T_{2}$ by removing finitely many successors at each splitnode $\sigma$ of $T_{2}$ according to the $N=N_{\sigma}$ in the condition above. This means, that the tree $T_{3}$ has the property that for every $\sigma\in split(T_{3})$ we have
    \begin{center}
        $\forall i\in succ_{T_{3}}(\sigma)$ $T_{3}\forces"$ for every $ k\leq |\sigma|$ we have $y^{\sigma}_{k}|_{N_{\sigma}}\subseteq \tau^{(\sigma,i)}_{k}$ and $y^{\sigma}_{k}|_{N_{\sigma}}\neq x|_{N_{\sigma}}"$
    \end{center}
    It follows that
    \begin{center}
        $T_{3}\forces"\forall n\in\{m\in\omega:\dot{r}_{gen}|_m\in split(T_{3})\}$ we have $x|_{\dot{I}_{n}}\notin \dot{S}_{n}"$
    \end{center}
    This finishes the proof of the lemma.
    \end{proof}
As a corollary from preceding lemma, we obtain that the following distinction of Borel conjectures holds for the classes $\NA$, $\MA$ and $\Nstar$.
    \begin{theorem}
        It is consistent that $\BC(\NA)$, while $\neg \BC(\MA)$ and $\neg \BC(\Nstar)$. In particular, the following constellation of Borel Conjectures holds:
        
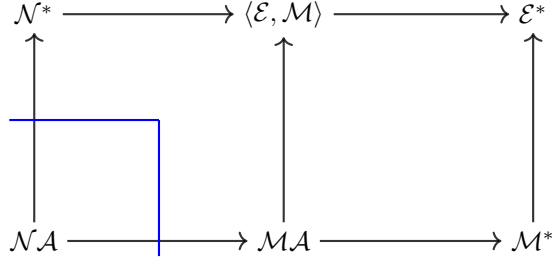
\begin{center}
		\begin{tikzpicture}[]
			\tikzset{
				textnode/.style={text=black}, 
			}
			\tikzset{
				edge/.style={color=black, thick, opacity=0.8}, 
			}
			\newcommand{\w}{3.3}
			\newcommand{\h}{2.0}
			
			\node[textnode] (NA) at (0,  0) {$\NA$};
            \node[textnode] (N) at (0,  1.5*\h) {$\Nstar$};
			\node[textnode] (MA) at (\w,  0) {$\MA$};
			\node[textnode] (EM) at (\w,  1.5*\h) {$\EM$};
			\node[textnode] (M) at (\w*2,  0) {$\Mstar$};
			\node[textnode] (E) at (\w*2,  1.5*\h) {$\Estar$};
			
			\draw[->, edge] (NA) to (MA);
			\draw[->, edge] (MA) to (EM);
			\draw[->, edge] (MA) to (M);
			\draw[->, edge] (EM) to (E);
            \draw[->, edge] (M) to (E);
            \draw[->, edge] (NA) to (N);
            \draw[->, edge] (N) to (EM);

			\draw[blue,thick] (-0.1*\w,0.8*\h)--(0.5*\w,0.8*\h);

			\draw[blue,thick] (0.5*\w,-0.10*\h)--(0.5*\w,0.8*\h);

		\end{tikzpicture}
\end{center}
    \end{theorem}
    \begin{proof}
        Consider a c.s.i. of length $\omega_{2}$ over a model of CH such that random forcing $\B$ and Miller forcing $\MMM$ alternate as iterands. We claim that in the resulting model $\mathbb{V}^{\omega_{2}}$ the thesis holds. In the final model, on the one hand $\cov(\Null)=\omega_{2}$, which means that any set of size $\omega_{1}$ is in $\Nstar$, on the other, $\bbb=\omega_{1}$ so by Bartoszyński Theorem \ref{BartoszynskiMANA} there are uncountable meager-additive sets. We will finish the proof by showing that any null-additive set is countable. Suppose that $X\subseteq\Cantor$ is an uncountable set in the final model $\mathbb{V}^{\omega_{2}}$. Without loss of generality, we assume that $X$ is of size $\omega_{1}$. Then there is $\alpha\in\omega_{2}$ such that $X\in\mathbb{V}^{\alpha}$. Then, after adding consecutive Miller real, $X$ is not null-additive. Let $N\in\Null\cap \mathbb{V}^{\alpha+1}$ be such that $X+N\notin\Null$ holds in $\mathbb{V}^{\alpha+1}$. It follows that the set $X+N$ has positive outer measure. Recall then that both Miller and random forcing (and theirs countable support iterations) preserve outer measure (in \cite{BJ} see Theorem 7.3.47 and Theorem 6.3.13). However, this means that $X+N\notin\Null$ holds in $\mathbb{V}^{\omega_{2}}$. This completes the proof.
    \end{proof}
    Applying the same argument, we get the following result for Miller model.
    \begin{theorem}\label{NoNAinMillers}
        In Millers model $\NA$-sets are countable while $\MA$-sets are not.
    \end{theorem}
    The status of dual Borel conjecture in Miller model is still open. A natural example of a strongly meager set is a Sierpiński set (see discussion in section \ref{SpecialSmallSetsSection}). However, in \cite{KillingLuzinSierp} Judah and Shelah showed that there are no Sierpiński sets in Miller or Laver model. However, the existence of uncountable strongly meager sets in Miller (or Laver model) is an old open problem (see question at the end of \cite{KillingLuzinSierp}).\\

    The above Theorem has one more consequence on the existence of small sets of reals. In \cite{GalvinMiller}  Galvin and Miller introduced the notion of \emph{strong} $\gamma$-\emph{sets}, a certain strengthening of a notion of $\gamma$-\emph{sets}. They showed that every $\gamma$-set is meager-additive and every strong $\gamma$-set is null-additive. Recently, in \cite{LyubomyrSzewczak} Halber, Szewczak and Zdomskyy showed that in Miller model there are no strongly measure zero (and thus no meager-additive) sets of size $\omega_{2}$. As a corollary from our Theorem \ref{NoNAinMillers} we get the following.
    \begin{corollary}
        There are no uncountable strong $\gamma$-sets in Miller model.
    \end{corollary}
    We will finish this section by showing that certain constellations of Borel Conjectures are not possible. The first instance of such phenomenon is the following result of Calder\'on, which shows that it is not always possible to separate Borel Conjectures for different additive classes.
    \begin{theorem}(Calder\'on; \cite{Calderon})\label{CalderonSMZandMA}
        $\BC(\SMZ)$ and $\BC(\MA)$ are equivalent.
    \end{theorem}
    It turns out that similar situation works between the classes $\BC(\Estar)$ and $\BC(\EM)$.
    \begin{theorem}\label{BCofEstarAndEM}
        $\BC(\Estar)$ and $\BC(\EM)$ are equivalent.
    \end{theorem}
    \begin{proof}
        We will only show that $\BC(\Estar)$ implies $\BC(\EM)$. Suppose that $\BC(\Estar)$ is not true i.e. that there is an uncountable $X\subseteq\Cantor$, $X\in\Estar$. Without loss of generality, we may assume that $|X|=\omega_{1}$. If $\bbb=\omega_{1}$ then by Theorem \ref{BartoszynskiMANA} there is an uncountable meager-additive set, which by  discussion from the introduction is in $\EM$. Assume that $\omega_{1}<\bbb$. Note that any Borel image of $X$ into $\Baire$ is bounded. By the argument similar to the proof of Proposition \ref{NonEMvsNonEstar} (see also Proposition 11 of \cite{Weiss2018}), $X$ is in $\EM$, which finishes the proof.
    \end{proof}
    We finish this section with several questions regarding different Borel Conjectures.\\
    We already know that in Miller's model $\BC(\NA)$ holds while $\neg \BC(\MA)$. It is an old open problem what happens with strongly meager sets is Miller's or Laver's model.
    \begin{question}
        Does $\BC(\mathcal{SM})$, $\BC(\EM)$ or $\BC(\Estar)$ hold in Laver model? Is dual Borel conjecture $\BC(\Nstar)$ true in Miller model?
    \end{question}
    In \cite{BorelWithDualBorel}, Goldstern, Kellner, Shelah and Wohofsky showed that it is consistent that both Borel conjecture and dual Borel conjecture holds. As $\mathcal{SMZ}\cup\mathcal{SM}\subseteq\Estar$, it is natural to ask the following.
    \begin{question}
        Is $\BC(\Estar)$ (or equivalently $\BC(\EM)$) consistent?
    \end{question}
    We already see that several constellations of Borel conjectures are consistent with ZFC while others (by Theorem \ref{CalderonSMZandMA} and Theorem \ref{BCofEstarAndEM}) are impossible. The only remaining constellations for which I do not know the status of consistency are the following two. Note that at least one of these two constellations must hold in the model of Borel Conjecture and dual Borel Conjecture (see Theorem \ref{BorelWithDualBorel}).
    \begin{question}
        Is it possible to consistently obtain the following two constellations of Borel Conjectures presented below?
    
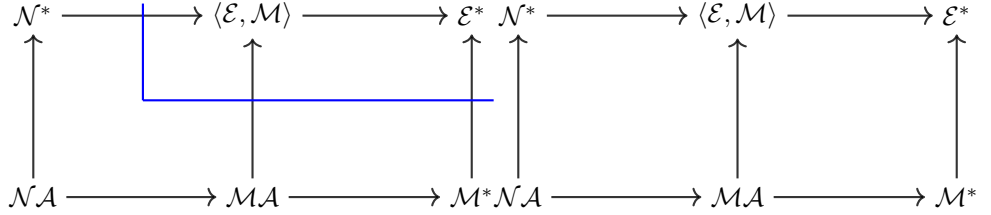
\begin{figure}[h!]
    \centering
    \begin{subfigure}[t]{0.5\textwidth}
        \centering
        \begin{center}
		\begin{tikzpicture}[]
			\tikzset{
				textnode/.style={text=black}, 
			}
			\tikzset{
				edge/.style={color=black, thick, opacity=0.8}, 
			}
			\newcommand{\w}{2.9}
			\newcommand{\h}{1.6}
			
			\node[textnode] (NA) at (0,  0) {$\NA$};
            \node[textnode] (N) at (0,  1.5*\h) {$\Nstar$};
			\node[textnode] (MA) at (\w,  0) {$\MA$};
			\node[textnode] (EM) at (\w,  1.5*\h) {$\EM$};
			\node[textnode] (M) at (\w*2,  0) {$\Mstar$};
			\node[textnode] (E) at (\w*2,  1.5*\h) {$\Estar$};
			
			\draw[->, edge] (NA) to (MA);
			\draw[->, edge] (MA) to (EM);
			\draw[->, edge] (MA) to (M);
			\draw[->, edge] (EM) to (E);
            \draw[->, edge] (M) to (E);
            \draw[->, edge] (NA) to (N);
            \draw[->, edge] (N) to (EM);

			\draw[blue,thick] (0.5*\w,1.6*\h)--(0.5*\w,0.8*\h);

			\draw[blue,thick] (0.5*\w,0.8*\h)--(2.1*\w,0.8*\h);

		\end{tikzpicture}
\end{center}
        \label{Fig:SubLeft}
    \end{subfigure}%
    ~ 
    \begin{subfigure}[t]{0.5\textwidth}
        \centering
        
\begin{center}
		\begin{tikzpicture}[]
			\tikzset{
				textnode/.style={text=black}, 
			}
			\tikzset{
				edge/.style={color=black, thick, opacity=0.8}, 
			}
			\newcommand{\w}{2.9}
			\newcommand{\h}{1.6}
			
			\node[textnode] (NA) at (0,  0) {$\NA$};
            \node[textnode] (N) at (0,  1.5*\h) {$\Nstar$};
			\node[textnode] (MA) at (\w,  0) {$\MA$};
			\node[textnode] (EM) at (\w,  1.5*\h) {$\EM$};
			\node[textnode] (M) at (\w*2,  0) {$\Mstar$};
			\node[textnode] (E) at (\w*2,  1.5*\h) {$\Estar$};
			
			\draw[->, edge] (NA) to (MA);
			\draw[->, edge] (MA) to (EM);
			\draw[->, edge] (MA) to (M);
			\draw[->, edge] (EM) to (E);
            \draw[->, edge] (M) to (E);
            \draw[->, edge] (NA) to (N);
            \draw[->, edge] (N) to (EM);

			\draw[blue,thick] (2.2*\w,-0.15*\h)--(2.2*\w,1.6*\h);

		\end{tikzpicture}
\end{center}
        \label{Fig:SubRight}
    \end{subfigure}
    \label{Fig:Main}
\end{figure}
    \end{question}

    \section{Cardinal invariants of {$\sigma$}-ideals {$\mathcal{H}_{F}$}}\label{SectionCardInvsHf}
    We start this section with an introduction of some more specific cardinal invariants (often referred to as slalom numbers) that we are going to use throughout the section. The framework that we are going to use is the one of relational systems (\cite{Blass}).

\begin{definition}
    A relational system is a triple $\mathcal{R}=\langle\mathcal{X},\preccurlyeq,\mathcal{Y}   \rangle$ where $\mathcal{X}$ and $\mathcal{Y}$ are non-empty sets and $\preccurlyeq$ is a relation between $\mathcal{X}$ and $\mathcal{Y}$ i.e. $\preccurlyeq\subseteq\mathcal{X}\times\mathcal{Y}$. The bounding and dominating number of a relational system $\mathcal{R}$ is defined as:
    \begin{itemize}
        \item[--] $\bbb(\mathcal{R})=\min\{|\mathcal{Z}|:\mathcal{Z}\subseteq\mathcal{X}$ $\neg\exists y\in\mathcal{Y}$ $\forall x\in\mathcal{Z}$ $x\preccurlyeq y\}$
        \item[--] $\ddd(\mathcal{R})=\min\{|\mathcal{Z}|:\mathcal{Z}\subseteq\mathcal{Y}$ $\forall x\in\mathcal{X}$ $\exists y\in\mathcal{Z}$ $x\preccurlyeq y\}$
    \end{itemize}
\end{definition}
Many cardinal invariants can be presented as bounding and dominating numbers of appropriate relational systems. For example, for a $\sigma$-ideal $\I$ on $\Cantor$ we have $\add(\I)=\bbb(\I,\subseteq,\I)$, $\cof(\I)=\ddd(\I,\subseteq,\I)$ and $\non(\I)=\bbb(\Cantor,\in,\I)$, $\cov(\I)=\ddd(\Cantor,\in,\I)$. Also $\bbb=\bbb(\Baire,\leq^{*},\Baire)$ and $\ddd=\ddd(\Baire,\leq^{*},\Baire)$.\\

An important tool for comparing bounding and dominating numbers of different relational systems is the notion of Tukey reduction
\begin{definition}
    Let $\mathcal{R}_{0}=\langle\mathcal{X}_{0},\preccurlyeq_{0},\mathcal{Y}_{0}\rangle$ and $\mathcal{R}_{1}=\langle\mathcal{X}_{1},\preccurlyeq_{1},\mathcal{Y}_{1}\rangle$ be two relational systems.
    We will say that a pair of functions $\phi:\mathcal{X}_{0}\rightarrow\mathcal{X}_{1}$ and $\psi:\mathcal{Y}_{1}\rightarrow\mathcal{Y}_{0}$ forms a Tukey reduction if for any $\mathcal{Z}\subseteq\mathcal{X}_{0}$ and $y\in\mathcal{Y}_{1}$, if $y$ $\preccurlyeq_{1}$-dominates all members of $\phi[\mathcal{Z}]$, then $\psi(y)$ $\preccurlyeq_{0}$-dominates all members of $\mathcal{Z}$. If such pair exists we will say that $\mathcal{R}_{0}$ is Tukey reducible to the $\mathcal{R}_{1}$. If $\mathcal{R}_{0}$ and $\mathcal{R}_{1}$ are Tukey reducible to each other we will say that they are Tukey equivalent.
\end{definition}
The existence of such reduction implies inequalities between bounding and dominating numbers of corresponding systems.
\begin{proposition}
    If $\mathcal{R}_{0}$ is Tukey reducible to $\mathcal{R}_{1}$, then 
    \begin{itemize}
        \item[--] $\bbb(\mathcal{R}_{1})\leq\bbb(\mathcal{R}_{0})$,
        \item[--] $\ddd(\mathcal{R}_{0})\leq\ddd(\mathcal{R}_{1})$.
    \end{itemize}
\end{proposition}
We will be interested only in the cardinal invariants themselves but not in the Tukey reduction between the corresponding relational systems. However, most of the proofs are formulated in a way that the underlying Tukey reduction is easy to see.\\

Let $\mathbb{IP}$ denote the collection of all partitions of $\omega$ into finite intervals. For two such partitions $I=(I_{n})_{n\in\omega}$ and $J=(J_{n})_{n\in\omega}$ we will define relation $I\sqsubseteq^{*}J$ if for almost all $n\in\omega$ there is $k\in\omega$ such that $I_{k}\subseteq J_{n}$. Then the relational systems $\langle\mathbb{IP},\sqsubseteq^{*},\mathbb{IP}\rangle$ and $\langle\Baire,\leq^{*},\Baire\rangle$ are Tukey equivalent. In particular $\bbb(\mathbb{IP},\sqsubseteq^{*},\mathbb{IP})=\bbb$ and $\ddd(\mathbb{IP},\sqsubseteq^{*},\mathbb{IP})=\ddd$ (see Blass \cite{Blass}).\\

We will now recall the slalom numbers that we are going to use in this article. The reader may consult extensive surveys \cite{SlalomNumbersSurvey} and \cite{SlalomNumbersOnReals} for more information of slalom numbers and Tukey reductions between related relational systems. 

\subsection{Invariants related to additivity of measure}
Let $b\in\Baire$ be increasing function. By $\mathcal{S}lm(b)$ we denote the set of $b$-slaloms i.e. functions $S:\omega\rightarrow[\omega]^{<\omega}$ such that $|S(n)|\leq b(n)$ for every $n\in\omega$. If $F\in\Baire$ is such that $b(n)\leq F(n)$ for every $n$, then by $\mathcal{S}lm(F,b)$ we mean slaloms that lives below $F$ i.e. $S(n)\in [F(n)]^{b(n)}$ for every $n\in\omega$.\\
For $x\in\Baire$ (or $\PIF$) and a slalom $S\in\mathcal{S}lm(b)$ (or $\mathcal{S}lm(F,b)$) let $x\in^{*}S$ if $x(n)\in S(n)$ for almost all $n\in\omega$. The following is a well-known characterization of additivity and cofinality of measure
    \begin{theorem}(Bartoszyński, Fremlin; see \cite{BJ})
    For any increasing $b\in\Baire$ we have:
    \begin{itemize}
        \item[--] $\add(\Null)=\bbb(\Baire,\in^{*},\mathcal{S}lm(b))$,
        \item[--] $\cof(\Null)=\ddd(\Baire,\in^{*},\mathcal{S}lm(b))$.
    \end{itemize}
    \end{theorem}
    In case of the slalom number on $\PIF$ it is easy to see that we have $\bbb(\Baire,\in^{*},\mathcal{S}lm(b))\leq\bbb(\PIF,\in^{*},\mathcal{S}lm(F,b))$ holds for any increasing $F\in\Baire$ and $b\leq F$. Also, it is easy to see that if $F_{0}\leq F_{1}$, then $\bbb(\Pi F_{1},\in^{*},\mathcal{S}lm(F_{1},b))\leq \bbb(\Pi F_{0},\in^{*},\mathcal{S}lm(F_{0},b))$. Using the characterization stated above it is also clear that $\add(\Null)\leq\bbb(\PIF,\in^{*},\mathcal{S}lm(F,b))$ and also $\cof(\Null)\geq\ddd(\PIF,\in^{*},\mathcal{S}lm(F,b))$ hold for every $F\in\Baire$. Pawlikowski used the bounded version of the above to characterize the transitive additivity of measure zero sets and thus the uniformity number of null-additive sets.
    \begin{theorem}(Pawlikowski; \cite{PawlikowskiTransitiveInvariants})
        For any increasing $b\in\Baire$ we have:
        \begin{itemize}
            \item[-] $\add_{t}(\Null)=\min\{\bbb(\PIF,\in^{*},\mathcal{S}lm(F,b)):F\in\Baire\}$
        \end{itemize}
        In particular $\add(\Null)=\min\{\bbb,\add_{t}(\Null)\}$.
    \end{theorem}
    It follows that $\non(\NA)=\min\{\bbb(\PIF,\in^{*},\mathcal{S}lm(F,b)):F\in\Baire\}$. Recently, Cardona, Mejía and Rivera-Madrid calculated the additivity of null-additive sets to be the same invariant.
    \begin{theorem}(Cardona, Mejía and Rivera-Madrid; \cite{MejiaCardonaMA})
    \begin{itemize}
        \item[--] $\add(\NA)=\non(\NA)$
    \end{itemize}
    \end{theorem}
    It is also worth mentioning the following result of Carlson which use this invariant as a lower bound for additivity of strong measure zero sets (Carlson actually proved that $\bbb(\Baire,\in^{*},\mathcal{S}lm(b))\leq\add(\mathcal{SMZ})$ but the proof can be adapted in a straightforward manner).
    \begin{theorem}(Carlson; \cite{Carlson}, also \cite{BJ})
    For any increasing $b\in\Baire$ we have:
        \begin{itemize}
            \item[-] $\min\{\bbb(\PIF,\in^{*},\mathcal{S}lm(F,b)):F\in\Baire\}\leq \add(\mathcal{SMZ})$
        \end{itemize}
    \end{theorem}
    The dual invariant of $\bbb(\PIF,\in^{*},\mathcal{S}lm(F,b))$ also plays an important role in the study of the reals. The forcing property of keeping $\ddd(\PIF,\in^{*},\mathcal{S}lm(F,b))$'s small, is known as the “Laver property”. A forcing notion $\mathbb{P}$ has Laver property if it does not “increase” the cardinal invariants $\ddd(\PIF,\in^{*},\mathcal{S}lm(F,b))$, i.e. $\mathbb{P}$ has Laver property if for every $g\in \PIF\cap\mathbb{V}^{\mathbb{P}}$, where $F\in \Baire\cap\mathbb{V}$, there exists $S\in\mathbb{V} \cap\mathcal{S}lm(F,b))$ such that $g(n)\in S(n)$ for almost all $n\in\omega$. In \cite{KadaLaverNumber} Kada investigated the invariants $\ddd(\PIF,\in^{*},\mathcal{S}lm(F,b))$. The well-known fact that 'Sacks property' of a forcing notion is equivalent to the 'Laver property' together with being $\Baire$-bounding by the forcing can be formulated as follows.
    \begin{proposition}(Kada; \cite{KadaLaverNumber})\label{Kada}
    For any increasing $b\in\Baire$ we have:
    \begin{itemize}
        \item[--] $\cof(\Null)=\max\{\ddd,\sup\{\ddd(\PIF,\in^{*},\mathcal{S}lm(F,b)): F\in\Baire\}\}$
    \end{itemize}
    \end{proposition}

\subsection{Invariants related to additivity of meager sets}
Let $x,y\in\Baire$ (or $\PIF$) and let $I=(I_{n})_{n\in\omega}$ be an interval partition. We will write $x\circeq(y,I)$ if for almost all $n\in\omega$ there is $i\in I_{n}$ such that $x(i)=y(i)$. The following characterization of additivity and cofinality of meager sets is known.
\begin{theorem}(Bartoszyński; \cite{BartoszCombAspects}, see also Bartoszyński and Judah\cite{BJ})
    \begin{itemize}
        \item[--] $\add(\Meager)=\bbb(\Baire,\circeq,\Baire\times\mathbb{IP})$,
        \item[--] $\cof(\Meager)=\ddd(\Baire,\circeq,\Baire\times\mathbb{IP})$.
    \end{itemize}
\end{theorem}
As before, taking the bounded version of the above invariant gives us the transitive additivity of meager sets, and thus, the uniformity of meager additive sets
\begin{theorem}(Pawlikowski; \cite{PawlikowskiTransitiveInvariants})
    \begin{itemize}
        \item[--] $\add_{t}(\Meager)=\min\{\bbb(\PIF,\circeq,\PIF\times\mathbb{IP}): F\in\Baire\}$
    \end{itemize}
\end{theorem}
In the next section we will be interested in the $\sigma$-ideals $\mathcal{H}_{F}$ on $\PIF$ (see definition \ref{DefinicjaHF} below) for which $\non(\mathcal{H}_{F})=\bbb(\PIF,\circeq,\PIF\times\mathbb{IP})$ and $\cov(\mathcal{H}_{F})=\ddd(\PIF,\circeq,\PIF\times\mathbb{IP})$. These two invariants were investigated originally by Cardona, Mejía, Rivera-Madrid in \cite{MejiaCardonaMA} and later by Cardona in \cite{CardonaMA}. In the framework of relational systems, Cardona, Mejía, Rivera-Madrid investigated uniformity and covering of $\HH_{F}$. In their article, the invariants $\non(\HH_{F})$ and $\cov(\HH_{F})$ were denoted by $\bbb^{eq}_{F}$ and $\ddd^{eq}_{F}$ respectively.
They showed the following connections to the uniformity and covering of meager-additive sets.
\begin{theorem}(Cardona, Mejía, Rivera-Madrid; \cite{MejiaCardonaMA})\label{NonCovMAvscovnonHHF}
    The following hold:
    \begin{itemize}
        \item[--] $\non(\MA)=\min\{\bbb(\PIF,\circeq,\PIF\times\mathbb{IP}):F\in\Baire\}$,
        \item[--] $\cov(\MA)\geq\sup\{\ddd(\PIF,\circeq,\PIF\times\mathbb{IP}):F\in\Baire\}$.
    \end{itemize}
\end{theorem}

\subsection{Invariants related to infinitely equal and eventually different numbers}\label{SectionWithInftyInSlalom}
Next, we define relational systems related to infinitely-equal and eventually-different reals. For $x,y\in\Baire$ (or $\PIF$) let $x=^{\infty}y$ if $x(n)=y(n)$ for infinitely many $n\in\omega$. Dually, $x\neq^{*}y$ if $x(n)\neq y(n)$ for almost all $n\in\omega$. Note that $\ddd(\Baire,\neq^{*},\Baire)=\bbb(\Baire,=^{\infty},\Baire)$ and also $\ddd(\Baire,\in^{\infty},\mathcal{S}lm(b))=\bbb(\Baire,\neq^{*},\Baire)$. For an $x\in\Baire$ ($x\in\PIF$) and $S\in\mathcal{S}lm(b)$ ($S\in\mathcal{S}lm(F,b)$) we will write that $x\in^{\infty}S$ if $x(n)\in S(n)$ for infinitely many $n\in\omega$.\\
The following is a well-known characterization of covering and uniformity of meager sets.
\begin{theorem}(Bartoszyński; \cite{BartoszCombAspects}, Miller; \cite{MillerAddEvsB}, see also \cite{BJ})
For any increasing $b\in\Baire$ we have:
    \begin{itemize}
        \item[--] $\cov(\Meager)=\ddd(\Baire,\neq^{*},\Baire)=\ddd(\Baire,\in^{\infty},\mathcal{S}lm(b))$,
        \item[--] $\non(\Meager)=\bbb(\Baire,\neq^{*},\Baire)=\bbb(\Baire,\in^{\infty},\mathcal{S}lm(b))$.
    \end{itemize}
\end{theorem}
    Regarding bounded infinitely-equal/eventually-different numbers it is easy to see that $\ddd(\Baire,\neq^{*},\Baire)\leq\ddd(\PIF,\neq^{*},\PIF)$ and $\bbb(\PIF,\neq^{*},\PIF)\leq\bbb(\Baire,\neq^{*},\Baire)$ for any function $F\in\Baire$. Also, it is not difficult to notice that if $F_{0}\leq F_{1}$, then we have the inequalities $\ddd(\Pi F_{1},\neq^{*},\Pi F_{1})\leq\ddd(\Pi F_{0},\neq^{*},\Pi F_{0})$ and $\bbb(\Pi F_{0},\neq^{*},\Pi F_{0})\leq\bbb(\Pi F_{1},\neq^{*},\Pi F_{1})$. Also, if the range of the function $F$ is finite, the invariant $\bbb(\PIF,\neq^{*},\PIF)$ if finite and $\ddd(\PIF,\neq^{*},\PIF)=\continuum$.\\
    Miller used these invariants to characterize the uniformity of strong measure zero sets. Independently, Pawlikowski obtained the same characterization of the transitive covering of meager sets. Thus, we have the following result
    \begin{theorem}(Miller; \cite{MillerAddEvsB};  Pawlikowski; \cite{PawlikowskiTransitiveInvariants}; see also \cite{BJ})
        \begin{itemize}
            \item[--] $\non(\mathcal{SMZ})=\min\{\ddd(\PIF,\neq^{*},\PIF): F\in\Baire\}=\cov_{t}(\Meager)$
        \end{itemize}
    \end{theorem}
    These invariants can be also used to characterize additivity of meager sets.
    \begin{theorem}(Miller \cite{MillerAddEvsB})\label{MillerAddMeager}
    \begin{itemize}
        \item[--] $\add(\Meager)=\min\{\bbb, \min\{\ddd(\PIF,\neq^{*},\PIF): F\in\Baire\}\}$
    \end{itemize}
    \end{theorem}

\subsection{Invariants related to the covering of measure}
    An important role in the study of covering and uniformity of measure zero sets is played the two slalom numbers $\bbb(\PIF,\in^{\infty},\mathcal{S}lm(F,b))$ and $\ddd(\PIF,\in^{\infty},\mathcal{S}lm(F,b))$. It is not difficult to see that $F_{0}\leq^{*} F_{1}$ and $b_{1}\leq^{*} b_{0}$ implies that $\bbb(\Pi F_{1},\in^{\infty},\mathcal{S}lm(F_{1},b_{1}))\leq \bbb(\Pi F_{0},\in^{\infty},\mathcal{S}lm(F_{0},b_{0}))$ and $\ddd(\Pi F_{1},\in^{\infty},\mathcal{S}lm(F_{1},b_{1}))\geq \ddd(\Pi F_{0},\in^{\infty},\mathcal{S}lm(F_{0},b_{0}))$. One may also see that the inequalities $\bbb(\PIF,\in^{\infty},\mathcal{S}lm(F,b))\geq\bbb(\PIF,=^{\infty},\PIF)$ and $\ddd(\PIF,=^{\infty},\PIF)\geq \ddd(\PIF,\in^{\infty},\mathcal{S}lm(F,b))$ hold for every $F\in\Baire$ (see \cite{CardonaMejiaOnYoryoka} or \cite{SlalomNumbersOnReals}).\\
    
    Bartoszyński, in an attempt for combinatorial characterization of covering and uniformity of measure, proved the following connections to these slalom numbers.
    \begin{theorem}(Bartoszyński; \cite{BartoszynskiCoveringNull} or Cardona and Mejía; \cite{SlalomNumbersOnReals})\label{covNvsSmallsetInv}
    \begin{itemize}
        \item[--] $\cov(\Null)\leq\min\{\ddd(\PIF,\in^{\infty},\mathcal{S}lm(F,b)):\Sigma\frac{b(i)}{F(i)}<\infty\}$,
        \item[--] $\non(\Null)\geq\sup\{\bbb(\PIF,\in^{\infty},\mathcal{S}lm(F,b)):\Sigma\frac{b(i)}{F(i)}<\infty\}$
    \end{itemize}
    also
    \begin{itemize}
        \item[--] $\min\{\bbb,\min\{\ddd(\PIF,\in^{\infty},\mathcal{S}lm(F,b)):\Sigma\frac{b(i)}{F(i)}<\infty\}\}\leq\cov(\Null)$,
        \item[--] $\max\{\ddd,\sup\{\bbb(\PIF,\in^{\infty},\mathcal{S}lm(F,b)):\Sigma\frac{b(i)}{F(i)}<\infty\}\}\geq\non(\Null)$
    \end{itemize}
    in particular
    \begin{itemize}
        \item[--] if $\cov(\Null)<\bbb$, then $\cov(\Null)=\ddd(\PIF,\in^{\infty},\mathcal{S}lm(F,b))$ for some $b,F\in\Baire$,
        \item[--] if $\ddd<\non(\Null)$, then $\non(\Null)=\bbb(\PIF,\in^{\infty},\mathcal{S}lm(F,b))$ for some $b,F\in\Baire$.
    \end{itemize}
    where $\Sigma\frac{b(i)}{F(i)}<\infty$.
    \end{theorem}
    Regarding the connections to the transitive classes, Bartoszyński and Judah showed the following estimation of the uniformity of strongly-meager sets $\Nstar$ and thus of the transitive covering of measure zero sets.
    \begin{theorem}(Bartoszyński and Judah; \cite{BartJudahBorelImages})\label{SMestimatesnon}
    The following inequalities hold:
    \begin{itemize}
        \item[--] $\non(\mathcal{SM})\geq\min\{\ddd(\PIF,\in^{\infty},\mathcal{S}lm(F,b)):\Sigma\frac{b(i)}{F(i)}<\infty\}$,
        \item[--] $\non(\mathcal{SM})\leq\min\{\ddd(\PIF,\in^{*},\mathcal{S}lm(F,b)):\Sigma\frac{b(i)}{F(i)}<\infty\}$
    \end{itemize}
    \end{theorem}

In the following diagram we summarizes cardinal invariants from Cichoń's diagram with slalom numbers that are relevant to the uniformity of additive classes. We will use the following shortcuts in the diagram below:
\begin{itemize}
    \item[--] $\min\bbb(F,\in^{*})=\min\{\bbb(\PIF,\in^{*},\mathcal{S}lm(F,b)):F\in\Baire\}$,
    \item[--] $\sup\ddd(F,\in^{*})=\sup\{\ddd(\PIF,\in^{*},\mathcal{S}lm(F,b)): F\in\Baire\}$,
    \item[--] $\min\bbb(F,\circeq)=\min\{\bbb(\PIF,\circeq,\PIF\times\mathbb{IP}): F\in\Baire\}$,
    \item[--] $\min\ddd(F,\neq^{*})=\min\{\ddd(\PIF,\neq^{*},\PIF): F\in\Baire\}$,
    \item[--] $\sup\bbb(F,\neq^{*})=\sup\{\bbb(\PIF,\neq^{*},\PIF): F\in\Baire\}$,
    \item[--] $\min\bbb(F,\in^{\infty})=\min\{\bbb(\PIF,\in^{\infty},\mathcal{S}lm(F,b)):\Sigma\frac{b(i)}{F(i)}<\infty\}$.
\end{itemize}
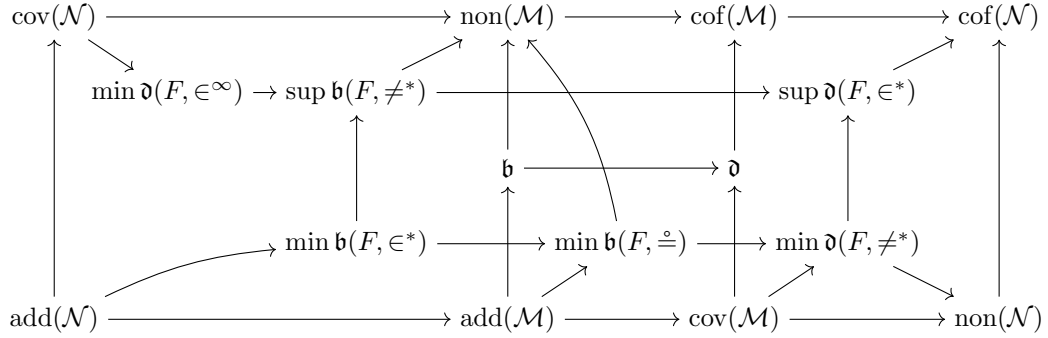
\begin{figure}[h!]\label{DiagramSlalomNumbers}
\centering
\tikzset{
  net node/.style = {draw, circle, minimum size=6mm},
  net edge/.style = {->},
  net cut/.style = {shorten >=-10mm, shorten <=-10mm, rounded corners=10mm, color=red},
  net cross/.style = {sloped, allow upside down, pos=.3},
}
\begin{tikzpicture}
        \newcommand{\edge}[5][]{\draw[net edge, #1] (#3) -- coordinate[net cross, name=#2] node[pos=.7, auto]{#5} (#4);}

        \node[name=addN] at ( 0, 0) {$\add(\Null)$};
        \node[name=addM] at ( 6, 0) {$\add(\Meager)$};
        \node[name=covM] at ( 9, 0) {$\cov(\Meager)$};
        \node[name=nonN] at ( 12.5, 0) {$\non(\Null)$};
        
        \node[name=b] at ( 6, 2) {$\bbb$};
        \node[name=d] at ( 9, 2) {$\ddd$};
        
        \node[name=cofN] at ( 12.5, 4) {$\cof(\Null)$};
        \node[name=cofM] at ( 9, 4) {$\cof(\Meager)$};
        \node[name=covN] at ( 0, 4) {$\cov(\Null)$};
        \node[name=nonM] at ( 6, 4) {$\non(\Meager)$};
        
        \node[name=NA] at ( 4, 1) {$\min\bbb(F,\in^{*})$};
        \node[name=nonMA]  at ( 7.5, 1) {$\min\bbb(F,\circeq)$};
        \node[name=nonSMZ] at ( 10.5, 1) {$\min\ddd(F,\neq^{*})$};
        \node[name=nonNstar]  at ( 1.5, 3) {$\min\ddd(F,\in^{\infty})$};
        \node[name=nonEM]  at ( 4, 3) {$\sup\bbb(F,\neq^{*})$};
        \node[name=Laver]  at ( 10.5, 3) {$\sup\ddd(F,\in^{*})$};

        


        \draw[->]       (addN) to (addM);
        \draw[->]       (addN) to (covN);
        \draw[->]       (addM) to (covM);
        \draw[->]       (addM) to (b);
        \draw[->]       (covM) to (nonN);
        \draw[->]       (b) to (d);
        \draw[->]       (b) to (nonM);
        \draw[->]       (nonM) to (cofM);
        \draw[->]       (d) to (cofM);
        \draw[->]       (cofM) to (cofN);
        \draw[->]       (covN) to (nonM);
        \draw[->]       (nonN) to (cofN);
        \draw[->]       (covM) to (d);

        \draw[->]       (addN) to[bend left=10] (NA);
        \draw[->]       (Laver) to (cofN);

        \draw[->]       (covN) to (nonNstar);
        \draw[->]       (nonEM) to (Laver);
        \draw[->]       (NA) to (nonEM);
        \draw[->]       (nonMA) to[bend right=15] (nonM);
        \draw[->]       (nonMA) to (nonSMZ);
        \draw[->]       (nonEM) to (nonM);
        \draw[->]       (nonNstar) to (nonEM);
        \draw[->]       (nonSMZ) to (Laver);
        \draw[->]       (NA) to (nonMA);
        \draw[->]       (nonSMZ) to (nonN);
        \draw[->]       (addM) to (nonMA);
        \draw[->]       (covM) to (nonSMZ);

\end{tikzpicture}\caption{Cicho\'n's diagram with the slalom numbers. The arrows point towards the ZFC-provable larger cadrinal invariant.}
\end{figure}

\subsection{Cardinal invariants of ideals $\HH_{F}$}
We now turn to the study of the ideals of $\PIF$ related to the cardinal invariants $\bbb(\PIF,\circeq,\PIF\times\mathbb{IP})$ and $\ddd(\PIF,\circeq,\PIF\times\mathbb{IP})$. We will be mostly interested in the cardinal invariants of these ideals, in ZFC-provable inequalities of these and in the independence the covering and uniformities of $\HH_{F}$'s and other classical characteristics like $\sss$, $\rrr$, $\bbb$, $\ddd$ and $\cov(\E)$, $\non(\E)$. By Theorem \ref{NonCovMAvscovnonHHF}, we already know $\non(\MA)=\min\{\bbb(\PIF,\circeq,\PIF\times\mathbb{IP}):F\in\Baire\}$ and $\cov(\MA)\geq\sup\{\ddd(\PIF,\circeq,\PIF\times\mathbb{IP}):F\in\Baire\}$.
The relational system $\langle\PIF,\circeq,\PIF\times\mathbb{IP}\rangle$ motivates the  definition of the following ideal on $\PIF$.
\begin{definition}\label{DefinicjaHF}
        Let $F\in\Baire$ be an increasing function. Denote by $\HH_{F}$ the $\sigma$-ideal on $\PIF$ generated by the sets of the form
        \begin{center}
            $H_{I,y}=\{x\in\PIF:\foralmostall_{n\in\omega}\exists k\in I_{n}$ $x(k)=y(n)\}$
        \end{center}
        where $I=(I_{n})_{n\in\omega}$ is an interval partition and $y\in\PIF$.
\end{definition}
Clearly, for any $F\in\Baire$ the $\sigma$-ideal $\mathcal{H}_{F}$ is generated by closed sets and $\mathcal{H}_{F}$ is properly contained in the $\sigma$-ideal of meager subsets of $\PIF$. It is easy to see that if $F_{0}\leq F_{1}$, then $\non(\HH_{F_{1}})\leq \non(\HH_{F_{0}})$ and $\cov(\HH_{F_{0}})\leq \cov(\HH_{F_{1}})$ hold. The following inequalities take place
\begin{proposition}(Mej\'ia, Cardona, Rivera-Madrid; \cite{MejiaCardonaMA} see also Cardona; \cite{CardonaMA})\label{CardonaBasicIneqNonCovHHF}
    The following inequalities hold for any $F\in\Baire$:
    \begin{enumerate}
        \item[--] $\cov(\Meager)\leq\cov(\mathcal{H}_{F})\leq\cof(\Meager)$
        \item[--] $\add(\Meager)\leq\non(\mathcal{H}_{F})\leq\non(\Meager)$
    \end{enumerate}
    also
    \begin{enumerate}
        \item[--] $\non(\HH_{F})\leq\bbb(\PIF,=^{\infty},\PIF)$
        \item[--] $\ddd(\PIF,=^{\infty},\PIF)\leq \cov(\HH_{F})$
    \end{enumerate}

\end{proposition}
We first focus on the additivities and cofinalities of ideals $\HH_{F}$. We want to show that the minimum of $\add(\HH_{F})$ for $F\in\Baire$ is equal to $\add(\Meager)$ and the maximum of $\cof(\HH_{F})$ for $F\in\Baire$ is equal to $\cof(\Meager)$. To do this, we will need the following result.
    \begin{proposition}\label{addHFbelowB}
    $\add(\mathcal{H}_{F})\leq\bbb$ and $\ddd\leq\cof(\mathcal{H}_{F})$ hold.
    \end{proposition}
    \begin{proof}
        First we show that $\add(\mathcal{H}_{F})\leq\bbb$. Let $\{I^{\alpha}:\alpha<\bbb\}$ be an unbounded set of interval partitions. For $\alpha<\bbb$ and $i\in 2$ define the following set
        \begin{center}
            $X^{i}_{\alpha}=\{x\in\PIF:\forall n\in\omega$ $x(\max I^{\alpha}_{2n+i})=0\}$.
        \end{center}
        Clearly, every such $X^{i}_{\alpha}$ is in $\mathcal{H}_{F}$. We claim that $\bigcup\{X^{i}_{\alpha}:\alpha<\bbb$ and $i\in 2\}\notin\mathcal{H}_{F}$. Assume towards contradiction that for some $y\in\PIF$ and an interval partition $J=(J_{n})_{n\in\omega}$ we have that
        \begin{center}
            $\bigcup\{X^{i}_{\alpha}:\alpha<\bbb$ and $i\in 2\}\subseteq H_{y,J}=\{x\in\PIF: \foralmostall n\in\omega$ $\exists k\in J_{n}$ $x(k)=y(k)\}$
        \end{center}
        As the family $\{I^{\alpha}:\alpha<\bbb\}$ is unbounded, there is $\alpha<\bbb$ such that for infinitely many $n$'s in $\omega$ the interval $J_{n}$ does not contain any $I^{\alpha}_{k}$, $k\in\omega$ and thus is itself contained in two consecutive intervals $I^{\alpha}_{k_{n}}\cup I^{\alpha}_{k_{n+1}}$. Then for some $i\in 2$, for infinitely many $n\in\omega$, $J_{n}\subseteq I^{\alpha}_{k_{n}}\cup I^{\alpha}_{k_{n+1}}$ and the parity of every $k_{n}$ is $i\in2$. Let then $x\in X^{i}_{\alpha}$ be such that $x|_{I^{\alpha}_{k_{n}}\cup I^{\alpha}_{k_{n+1}}}\neq y|_{I^{\alpha}_{k_{n}}\cup I^{\alpha}_{k_{n+1}}}$ for these intervals. But then $x\notin H_{y,J}$, which contradicts the previous assumption.\\
        To see that $\ddd\leq\cof(\mathcal{H}_{F})$ we dualize the argument above. Let $\kappa<\ddd$ and assume that $\{H_{\alpha}:\alpha<\kappa\}\subseteq\mathcal{H}_{F}$. We will show that these sets do not form a basis of $\mathcal{H}_{F}$. Without loss of generality assume that each $H_{\alpha}$ is of form $\{x\in\PIF:\foralmostall n\in\omega$ $\exists i\in I^{\alpha}_{n}$ $x(i)=y_{\alpha}(i)\}$ for some $y_{\alpha}\in\PIF$ and interval partition $I^{\alpha}=(I^{\alpha}_{n})_{n\in\omega}$. As $\kappa<\ddd$ there is an interval partition $(J_{n})_{n\in\omega}$ that dominates all $(I^{\alpha}_{n})_{n\in\omega}$'s. Then, similar argument as in the proof of the first inequality shows that the set
        \begin{center}
            $H=\{x\in\PIF:\forall n\in\omega$ $x(\min J_{n})=0\}$
        \end{center}
        is in $\mathcal{H}_{F}$ and cannot be covered by any $H_{\alpha}$ for $\alpha<\kappa$. This proves the second inequality.
    \end{proof}
    At this point we know that the inequalities $\add(\mathcal{H}_{F})\leq\min\{\bbb,\ddd(\PIF,=^{\infty},\PIF)\}$ and $\min\{\ddd,\bbb(\PIF,=^{\infty},\PIF)\}\leq\cof(\mathcal{H}_{F})$ hold for every increasing $F\in\Baire$. By Theorem \ref{MillerAddMeager}, we get the following characterization of additivity and cofinality of meager sets.
    \begin{corollary}
    We have:
        \begin{itemize}
        \item[--] $\add(\Meager)=\min\{\add(\HH_{F}):F\in\Baire\}$,
        \item[--] $\cof(\Meager)=\max\{\cof(\HH_{F}):F\in\Baire\}$.
    \end{itemize}
    \end{corollary}
    I do not know if additivity of $\HH_{F}$ is not simply equal to the additivity of meager sets in ZFC (for some or any $F\in\Baire$). The same goes for cofinalities of $\HH_{F}$'s.\\

    Next, we focus on the independence of the numbers $\non(\mathcal{H}_{F})$ and $\cov(\mathcal{H}_{F})$ from the numbers $\bbb$ and $\ddd$. We have the following in the case when $F\in\Baire$ has finite range.
        \begin{proposition}
        If $F\in \Baire$ is bounded, then:
        \begin{itemize}
            \item[--] $\cov(\mathcal{H}_{F})\leq\ddd$,
            \item[--] $\bbb\leq\non(\mathcal{H}_{F})$.
        \end{itemize}
    \end{proposition}
    \begin{proof}
        Assume that $F$ is bounded by $C\in\omega$. To show that $\cov(\mathcal{H}_{F})\leq\ddd$ let $I_{\alpha}=(I_{n}^{\alpha})_{n\in\omega}$'s where $\alpha<\ddd$ form a dominating family of interval partitions. We claim that the sets $H_{i,I^{\alpha}}$ for $i\leq C$ and $\alpha<\ddd$ cover the entire $\PIF$. To see this assume that $x\in\ \PIF$. Let $i\leq C$ and let $J=(J_{n})_{n\in\omega}\in\mathbb{IP}$ be such that for every $n\in\omega$ there is $j\in J_{n}$ with $x(j)=i$. But then, if $\alpha<\ddd$ is such that $J\leq^{*}I^{\alpha}$, then $x\in H_{i,I^{\alpha}}$.\\
        To show that $\bbb\leq\non(\mathcal{H}_{F})$ assume that $\kappa<\bbb$ and $\{x_{\alpha}:\alpha<\kappa\}\subseteq\PIF$. For every $\alpha<\kappa$ and $i\leq C$, if there are infinitely many $n\in\omega$ such that $x_{\alpha}(n)=i$, then define an interval partition $I^{i,\alpha}=(I^{i,\alpha}_{n})_{n\in\omega}$ such that for every $n\in\omega$ there is $j\in I^{i,\alpha}_{n}$ such that $x_{\alpha}(j)=i$. Let $J$ be an interval partition that dominated all $I^{i,\alpha}$'s, $i\leq C$, $\alpha<\kappa$. Then clearly $\{x_{\alpha}:\alpha<\kappa\}\subseteq\bigcup_{i\leq C}H_{i,J}$. 
    \end{proof}
    We will now focus on the independence of the invariants $\non(\HH_{F})$ and $\cov(\HH_{F})$ from $\bbb$ and $\ddd$ for increasing $F\in\Baire$. In \cite{MejiaCardonaMA} and \cite{CardonaMA} two different constructions of a model satisfying $\bbb<\non(\mathcal{H}_{F})$ were given. In \cite{CardonaMA} the author raised a question whether $\ddd<\non(\mathcal{H}_{F})$ is consistent. This inequality follows from the consistency of $\ddd<\min\{\bbb(\PIF,\in^{*},\mathcal{S}lm(F,b)): F\in\Baire\}$ proven by Brendle and Shelah in \cite{Evasionprediction}. Below we present a different proof of this inequality together with several other consistency results.
    \begin{proposition}
        For any increasing $F\in\Baire$ the following inequalities are consistent:
        \begin{itemize}
            \item[--] $\non(\mathcal{H}_{F})<\bbb$,
            \item[--] $\ddd<\non(\mathcal{H}_{F})$,
            \item[--] $\bbb<\cov(\mathcal{H}_{F})$,
            \item[--] $\ddd<\cov(\mathcal{H}_{F})$ for $F\in\Baire$ such that $\Sigma_{n}\frac{1}{F(n)}<\infty$,
            \item[--] $\cov(\mathcal{H}_{F})<\bbb$.
        \end{itemize}
    \end{proposition}
    \begin{proof}
        First, we show the consistency of $\non(\mathcal{H}_{F})<\bbb$. The inequality holds in Laver model in which $\bbb=\ddd=\omega_{2}$. To show that $\non(\mathcal{H}_{F})=\omega_{1}$ in Laver model we will use the following claim and the fact that Laver forcing (and its countable support iteration) has the Laver property.
        \begin{claim}\label{claimLaverProperty1}
            For an increasing $F\in\Baire$ Laver property keeps $\non(\mathcal{H}_{F})$ small.
        \end{claim}
        \begin{proof}
            Let $p\forces"\dot{y}\in\PIF$ and $\dot{I}=(\dot{I}_{n})_{n}\in\mathbb{IP}"$. Let $q\leq p$ and $S=(S_{n})_{n\in\omega}$ be such that $q\forces"\foralmostall n\in\omega$ $\dot{y}(n)\in S_{n}"$. Without loss of generality, we may assume that $|S_{n}|< F(n)$ for every $n\in\omega$. Then any $x\in\PIF$ such that $x(n)\in F(n)\setminus S_{n}$ satisfies $q\forces"x\notin H_{\dot{y}, \dot{I}}"$.
        \end{proof}
        As the Laver property keeps $\non(\HH_{F})$ small and is preserved by countable support iteration, the $\non(\HH_{F})=\omega_{1}$ holds in the Lavers model.\\
        
        Next, we show the consistency of $\ddd<\non(\HH_{F})$. Fix an interval partition $I=(I_{n})_{n\in\omega}$. We say that $A\subseteq \Pi_{i\in I_{n}}F(i)$ is good if $A=\{\tau\in \Pi_{i\in I_{n}}F(i):\exists i\in I_{n}$ $\tau(i)=\sigma(i)\}$ for some $\sigma\in \Pi_{i\in I_{n}}F(i)$. Let $Good_{n}$ stands for the collection of good subsets of $\Pi_{i\in I_{n}}F(i)$. Define the forcing $\mathbb{P}$ as follows: $T\in \mathbb{P}$ if
            \begin{enumerate}
                \item[1.] $T$ is a finitely splitting tree,
                \item[2.] for every $t\in T$ and $n\in dom(t)$ $t(n)\subseteq  \Pi_{i\in I_{n}}F(i)$ is good,
                \item[3.] $lim_{t\in T} norm_{n}(succ_{T}(t))=\infty$
            \end{enumerate}
            where the norm $norm_{n}:\mathcal{P}(Good_{n})\rightarrow \mathbb{Z}$ is recursively defined as follows: for $\mathcal{A}\subseteq Good_{n}$
            \begin{enumerate}
                \item[\textbf{a)}] $norm_{n}(\mathcal{A})\geq 0$ if and only if $\bigcup\mathcal{A}=\Pi_{i\in I_{n}}F(i)$,
                \item[\textbf{b)}] if $norm_{n}(\mathcal{A})=l+1$ and $\mathcal{A}=\mathcal{A}_{0}\cup\mathcal{A}_{1}$ then $norm_{n}(\mathcal{A}_{0})\geq l$ or $norm_{n}(\mathcal{A}_{1})\geq l$.
            \end{enumerate}
            In particular $\textbf{a)}$ and $\textbf{b)}$ implies that
            \begin{enumerate}
                \item[\textbf{c)}] if $norm_{n}(\mathcal{A})=l+1$ and $\sigma\in \Pi_{i\in I_{n}}F(i)$ then $norm_{n}(\{A\in\mathcal{A}:\sigma\in A\})\geq l$
            \end{enumerate}
            Clearly if $G\subseteq\mathbb{P}$ is generic filter then $\dot{R}_{gen}=\bigcup\bigcap G$ is a function with domain $\omega$ and for every $n\in\omega$ $\dot{R}_{gen}(n)\in Good_{n}$. In particular, there is $\dot{r}_{gen}\in \PIF$ such that for every $n\in\omega$ we have
            \begin{center}
                $\dot{R}_{gen}(n)=\{\tau\in \Pi_{i\in I_{n}}F(i):\exists i\in I_{n}$ $\tau(i)=(\dot{r}_{gen}|_{I_{n}})(i)\}$
            \end{center}
            Standard density argument together with the property $\textbf{c)}$ shows that for every $x\in \PIF$ from the ground model we have $x\in H_{I,r_{gen}}$. It follows that in $\mathbb{V}^{\omega_{2}}$ (the model obtained by countable support iteration of $\mathbb{P}$ of length $\omega_{2}$ over a model of $GCH$) we have $\non(\HH_{F})=\omega_{2}$.\\
            We will show that the forcing $\mathbb{P}$ is $\Baire$-bounding and thus $\ddd=\omega_{1}$ holds in $\mathbb{V}^{\omega_{2}}$.\\
            For $T\in\mathbb{P}$ and $n\in\omega$ define
            \begin{center}
                $lvl_{n}(T)=\min\{m\in\omega:\forall t\in T$ ($|t|\geq m \rightarrow norm(succ_{T}(t))\geq n$)$\}$
            \end{center}
            and let $T\leq_{n}S$ if $T\leq S$ and $\{t\in T:|t|\leq lvl_{n}(S)\}=\{t\in S:|t|\leq lvl_{n}(S)\}$. We will now show that the forcing notion $\mathbb{P}$ is $\Baire$-bounding. Similar argument shows that it has axiom A.\\
            We will need the following claim
            \begin{claim}
                $\forall T\Vdash"\dot{a}\in\omega"$ $\forall n\in\omega$ $\exists S\leq_{n} T$ $\exists X\subseteq\omega$ finite such that $S\Vdash"\dot{a}\in X"$.
            \end{claim}
            \begin{proof}
            Define the rank function $rank:T\rightarrow\omega$ such that for $t\in T$ we have: 
                \begin{center}
                    $rank(t)=0$ iff there is $n$-tree with stem $t$ which decides the value of $\dot{a}$,\\
                    $rank(t)=l+1$ iff $norm(\{A\in succ_{T}(t): rank(t^{\frown}A)=l\})\geq norm(succ_{T}(t))-1$
                \end{center}
                First, notice that the rank is well-defined for all $t\in T$. This is because if for some $t\in T$ the rank would be undefined $rank(t)=\infty$ then $norm(\{A\in succ_{T}(t):rank(t^{\frown}A)$ is defined$\})<norm(succ_{T}(t))-1$ and thus by the condition $\textbf{b)}$ we have $norm(\{A\in succ_{T}(t):rank(t^{\frown}A)$ is not defined$\})\geq norm(succ_{T}(t))-1$. Using this we can inductively construct a tree $S\in \mathbb{P}$, $S\leq T$ with stem $t$ such that for every $s\in S$ that extends the stem the $rank(s)$ is not defined. But then no extension of $S$ can decide the value of $\dot{a}$.\\
                For every $t\in T$ of length $lvl_{n}(T)$ we will construct the a tree $S_{t}\leq T|_{t}$ such that
                \begin{enumerate}
                    \item[--] $stem(S_{t})=t$,
                    \item[--] $\forall s\in S_{t}$ ( $t\subseteq s$ $\rightarrow$ $norm_{n}(succ_{S_{t}}(s))\geq n$)
                    \item[--] $S_{t}$ forces $\dot{a}$ into a finite set
                \end{enumerate}
                If $t$ is of rank $0$, then such tree exists by the definition. If $rank(t)>0$, then there is a finite tree $R\subseteq T|_{t}$ such that 
                \begin{itemize}
                    \item[--] $\forall s\in term(R)$ $rank(s)=0$,
                    \item[--] $\forall s\in R\setminus term(R)$ that extends $t$, $norm(succ_{R}(s))\geq norm(succ_{T}(s))-1$
                \end{itemize}
                Then let $S_{t}\leq T|_{t}$ be such that $R\subseteq S_{t}$ and for every $s\in term(R)$ $S_{t}|_{s}$ decides the value of $\dot{a}$. Then the tree $S:=\bigcup\{S_{t}:t\in T$ and $|t|=lvl_{n}(T)\}$ forces $\dot{a}$ into a finite set and also $S\leq_{n}T$.
            \end{proof}
            Assume now that $T\Vdash"\dot{f}\in\Baire"$. Using the claim above we can inductively construct a sequence $\{T_{n}:n\in\omega\}\subseteq \mathbb{P}$ and a sequence $\{X_{n}:n\in\omega\}$ of finite subsets of $\omega$ such that $T_{0}=T$ and for every $n\in\omega$ we have $T_{n+1}\leq_{n}T_{n}$ and $T_{n+1}\Vdash"\dot{f}(n)\in X_{n}"$. Clearly, if $S$ is fusion of the sequence $\{T_{n}:n\in\omega\}$, then $S$ forces that $\dot{f}$ is dominated by a ground model function.\\

        Next, we show the consistency of $\bbb<\cov(\mathcal{H}_{F})$. The inequality holds in the model obtained by c.s.i. of length $\omega_{2}$ with $\mathbb{P}_{\mathcal{H}_{F}}=\mathcal{B}or(\PIF)\setminus\mathcal{H}_{F}$ as each iterand. In this model we have $\cov(\mathcal{H}_{F})=\omega_{2}$ as we add appropriate generic reals along the iteration. On the other hand, by Theorem \ref{ZapletalClosedSetsForcingProperties} $\mathbb{P}_{\mathcal{H}_{F}}$ preserves Baire category (see also \cite{Cieslak}) so $\non(\Meager)=\omega_{1}$ in the iterated model.\\

        Next, we show the consistency of $\ddd<\cov(\mathcal{H}_{F})$ for $F\in\Baire$ fast enough that $\Sigma_{n}\frac{1}{F(n)}<\infty$. This inequality holds in random model. It is well-known that $\ddd=\omega_{1}<\cov(\Null)=\cov(\E)=\continuum=\omega_{2}$ holds in that model. On the other hand, by Proposition \ref{CardonaMAnons} we have that $\cov(\E)\leq\cov(\mathcal{H}_{F})$ for $F\in\Baire$ such that $\Sigma_{n}\frac{1}{F(n)}<\infty$.\\

        Next, we show the consistency of $\cov(\mathcal{H}_{F})<\bbb$. To show this, we again use Laver model in which $\bbb=\omega_{2}$. To complete the proof, we will show the following simple observation.
        \begin{claim}\label{claimLaverProperty2}
            For an increasing $F\in\Baire$ Laver property keeps $\cov(\mathcal{H}_{F})$ small.
        \end{claim}
        \begin{proof}
            Assume that  $p\forces"\dot{x}\in\PIF"$ and let $I=(I_{n})_{n\in\omega}$ be such interval partition that $|I_{n}|=n+1$ for all $n\in\omega$. Let $q\leq p$ and $S=(S_{n})_{n\in\omega}$ be such that
            \begin{itemize}
                \item[--] $\forall n\in\omega$ $\forall \sigma\in S_{n}$ $\sigma:I_{n}\rightarrow\omega$ is below $F$,
                \item[--] $q\forces"\foralmostall n\in\omega$ $\dot{x}|_{I_{n}}\in S_{n}"$.
            \end{itemize}
            Let then $y\in\PIF$ be such that for all $n\in\omega$ and every $\sigma\in S_{n}$ there is unique $i\in I_{n}$ such that $y(i)=\sigma(i)$. Such $y$ exists as $|S_{n}|=|I_{n}|$. Then, clearly $q\forces"\dot{x}\in H_{y, I}"$.
        \end{proof}
    \end{proof}
    I do not know the consistency of the inequality $\ddd<\cov(\mathcal{H}_{F})$ for $F\in\Baire$ such that $\Sigma_{n}\frac{1}{F(n)}=\infty$. 
    We now turn to the numbers $\non(\mathcal{H}_{F})$ and $\cov(\mathcal{H}_{F})$, and that these are independent from the numbers $\rrr$ and $\sss$. We have the following result.
    \begin{proposition}\label{HHFvsSplittingReaping}
        For any increasing $F\in\Baire$ the following inequalities are consistent:
        \begin{itemize}
            \item[--] $\non(\mathcal{H}_{F})<\sss$,
            \item[--] $\sss<\non(\mathcal{H}_{F})$,
            \item[--] $\non(\mathcal{H}_{F})<\rrr$,
            \item[--] $\rrr<\non(\mathcal{H}_{F})$,
            \item[--] $\cov(\mathcal{H}_{F})<\sss$,
            \item[--] $\sss<\cov(\mathcal{H}_{F})$,
            \item[--] $\cov(\mathcal{H}_{F})<\rrr$,
            \item[--] $\rrr<\cov(\mathcal{H}_{F})$ for sufficiently fast $F$.
        \end{itemize}
    \end{proposition}
    \begin{proof}
    First, we show the consistency of $\non(\mathcal{H}_{F})<\sss$. The inequality holds in iterated Mathias model. It is well-known that Mathias forcing add an unsplit real and thus forces $\sss=\omega_{2}=\continuum$ in the final model. To see that $\non(\mathcal{H}_{F})$ remains small we use Claim \ref{claimLaverProperty1} and the fact that Mathias forcing (and its countable support iterations) has the Laver property.\\

    Next, we show the consistency of $\sss<\non(\mathcal{H}_{F})$. The inequality follows from the consistency of $\sss<\add(\Meager)$ (which holds for example in Hechler model) and the fact that $\add(\Meager)\leq\non(\mathcal{H}_{F})$ (by Proposition \ref{CardonaBasicIneqNonCovHHF}).\\

    Next, we show the consistency of$\non(\mathcal{H}_{F})<\rrr$. The inequality again holds in Mathias model in which $\rrr=\omega_{2}$ as Mathias forcing adds splitting reals. Also follows from the consistency of $\non(\HH_{F})<\cov(\E)$ (see Proposition \ref{HHFvsClosedNull}) and the inequality $\cov(\E)\leq\rrr$.\\
        
    Next, we show the consistency of $\rrr<\non(\mathcal{H}_{F})$. The inequality follows from the consistency of $\non(\Meager)<\rrr$ (which holds, for example, in Silver model) and the inequality $\non(\mathcal{H}_{F})\leq\non(\Meager)$ (by Proposition \ref{CardonaBasicIneqNonCovHHF}).\\

    Next, we show the consistency of $\cov(\mathcal{H}_{F})<\sss$. The inequality holds in Mathias model as Mathias forcing adds unsplit real and has the Laver property (see Claim \ref{claimLaverProperty2}).\\

    Next, we show the consistency of $\sss<\cov(\mathcal{H}_{F})$. The inequality follows from the consistency of $\sss<\cov(\Meager)$ (which holds, for example, in Cohen model) and the inequality $\cov(\Meager)<\cov(\mathcal{H}_{F})$ (by Proposition \ref{CardonaBasicIneqNonCovHHF}). Also follows from the consistency of $\non(\E)<\cov(\mathcal{H}_{F})$ (see Proposition \ref{HHFvsClosedNull}) and the inequality $\sss\leq\non(\E)$.\\
        
    Next, we show the consistency of $\cov(\mathcal{H}_{F})<\rrr$. The inequality follows from the consistency of $\cof(\Meager)<\rrr$ (which holds for example in Silver model or Splitting tree model) and $\cov(\mathcal{H}_{F})\leq\cof(\Meager)$ (by Proposition \ref{CardonaBasicIneqNonCovHHF}) which holds in particular in Silver model.\\

    Next, we show the consistency of $\rrr<\cov(\mathcal{H}_{F})$ for sufficiently fast $F$. In \cite{Zapl1} Zapletal has shown that the forcing notion $\mathbb{PT}_{F,g}$ (see \cite{BJ} Definition 7.3.3. for the definition of the $\mathbb{PT}_{F,g}$) keeps $\rrr$ small if $F$ is sufficiently fast (we skip the technical condition of what it precisely means). On the other hand this forcing add an $F$-eventually different real and thus increases $\bbb(\PIF,\neq^{*},\PIF)$. It follows that in the iterated $\mathbb{PT}_{F,g}$-model we have the inequality $\rrr<\bbb(\PIF,\neq^{*},\PIF)$ holds. As we have $\bbb(\PIF,\neq^{*},\PIF)\leq\cov(\HH_{F})$ (by Proposition \ref{CardonaBasicIneqNonCovHHF}) the inequality $\rrr<\cov(\mathcal{H}_{F})$ follows. 
    \end{proof}
    The consistency of $\rrr<\cov(\mathcal{H}_{F})$ for arbitrary $F$ could follow from a positive answer to the question whether the idealized forcing notion $\mathbb{P}_{\HH_{F}}:=\mathcal{B}or(\PIF)\setminus\HH_{F}$ preserves P-points (see related Question \ref{DoesPHFpreservesPpoints}).\\
    We will now focus of the independence of the invariants $\non(\HH_{F})$ and $\cov(\HH_{F})$ from the invariants $\non(\E)$ and $\cov(\E)$. In \cite{MejiaCardonaMA} (see also \cite{CardonaMA}) the authors showed the following connections of $\HH_{F}$ to the closed null ideal $\E$.
    \begin{lemma}(Mej\'ia, Cardona, Rivera-Madrid; \cite{MejiaCardonaMA}, or Cardona;\cite{CardonaMA})\label{CardonaMAnons}
        For any $F\in\Baire$ we have the following:
        \begin{itemize}
            \item[--] if $\Sigma_{i}\frac{1}{F(i)}<\infty$, then $\non(\HH_{F})\leq\non(\E)$ and $\cov(\E)\leq\cov(\HH_{F})$,
            \item[--] if $\Sigma_{i}\frac{1}{F(i)}=\infty$, then $\cov(\Null)\leq\non(\HH_{F})$ and $\cov(\HH_{F})\leq\non(\Null)$.
        \end{itemize}
    \end{lemma}
    We have the following consistency results.
    \begin{proposition}\label{HHFvsClosedNull}
        The following equalities are consistent for increasing $F\in\Baire$:
        \begin{enumerate}
            \item[--] $\cov(\HH_{F})<\cov(\E)$ for $F\in\Baire$ such that $\Sigma_{i}\frac{1}{F(i)}=\infty$,
            \item[--] $\cov(\E)<\cov(\HH_{F})$ for sufficiently fast $F$,
            \item[--] $\cov(\HH_{F})<\non(\E)$,
            \item[--] $\non(\E)<\cov(\HH_{F})$,
            \item[--] $\non(\HH_{F})<\cov(\E)$,
            \item[--] $\cov(\E)<\non(\HH_{F})$,
            \item[--] $\non(\HH_{F})<\non(\E)$,
            \item[--] $\non(\E)<\non(\HH_{F})$ for $F\in\Baire$ such that $\Sigma_{i}\frac{1}{F(i)}=\infty$.
        \end{enumerate}
    \end{proposition}
    \begin{proof}
    First, we show the consistency of $\cov(\HH_{F})<\cov(\E)$ for $F\in\Baire$ such that $\Sigma_{i}\frac{1}{F(i)}=\infty$. The inequality follows from the consistency of $\non(\Null)<\cov(\Null)$ (which holds in the random model) and the inequalities $\cov(\HH_{F})\leq\non(\Null)$ (by Lemma \ref{CardonaMAnons}) and $\cov(\Null)\leq\cov(\E)$.\\
        
    Next, we show the consistency of $\cov(\E)<\cov(\HH_{F})$ for sufficiently fast $F$. This follows from the consistency of $\rrr<\cov(\HH_{F})$ (see Proposition \ref{HHFvsSplittingReaping}) and the inequality $\cov(\E)\leq\rrr$. Also follows from the consistency of $\max\{\cov(\Null),\ddd\}<\ddd(\PIF,=^{\infty},\PIF)$ and the inequalities $\cov(\E)\leq\max\{\cov(\Null),\ddd\}$ (by Theorem \ref{CardInvsofE}) and $\ddd(\PIF,=^{\infty},\PIF)\leq\cov(\HH_{F})$ (by Proposition \ref{CardonaBasicIneqNonCovHHF}).\\
        
    Next, we show the consistency of $\cov(\HH_{F})<\non(\E)$. The inequality follows from the consistency of $\cov(\HH_{F})<\sss$ (see Proposition \ref{HHFvsSplittingReaping}) and the inequality $\sss\leq\non(\E)$.\\
        
    Next, we show the consistency of $\non(\E)<\cov(\HH_{F})$. The inequality holds in the model obtained by c.s.i. of length $\omega_{2}$ with $\mathbb{P}_{\mathcal{H}_{F}}=\mathcal{B}or(\PIF)\setminus\mathcal{H}_{F}$ as each iterand. In this model, we have $\cov(\mathcal{H}_{F})=\omega_{2}$ as we add appropriate generic reals along the iteration. On the other hand, by Theorem \ref{ZapletalClosedSetsForcingProperties} $\mathbb{P}_{\mathcal{H}_{F}}$ preserves Baire category (see also \cite{Cieslak}) so $\non(\E)\leq\non(\Meager)=\omega_{1}$ in the iterated model.\\
        
    Next, we show the consistency of $\non(\HH_{F})<\cov(\E)$. The inequality holds in the model obtained by c.s.i. of length $\omega_{2}$ with $\mathbb{P}_{\E}=\mathcal{B}or(\Cantor)\setminus\E$ as each iterand. As in the previous inequality, from the fact that $\E$ is generated by closed sets, we get that in the model we have $\non(\HH_{F})\leq\non(\Meager)=\omega_{1}$.\\
    
    Next, we show the consistency of $\cov(\E)<\non(\HH_{F})$. The inequality follows from the consistency of $\rrr<\non(\HH_{F})$ (by Proposition \ref{HHFvsSplittingReaping}) and the inequality $\cov(\E)\leq\rrr$.\\
    
    Next, we show the consistency of $\non(\HH_{F})<\non(\E)$. The inequality follows from the consistency of $\non(\HH_{F})<\sss$ (by Proposition \ref{HHFvsSplittingReaping}) and the inequality $\sss\leq\non(\E)$.\\
        
    Next, we show the consistency of $\non(\E)<\non(\HH_{F})$ for $F\in\Baire$ such that $\Sigma_{i}\frac{1}{F(i)}=\infty$. This follows from the consistency of $\non(\Null)<\cov(\Null)$ (which holds in the random model) and the inequalities $\cov(\Null)\leq\non(\HH_{F})$ (by Lemma \ref{CardonaMAnons}) and $\non(\E)\leq\non(\Null)$.
    \end{proof}

    \begin{remark}
        The inequalities 
        \begin{center}
            $\non(\HH_{F})\leq\ddd(\PIF,=^{\infty},\PIF)$ and $\bbb(\PIF,=^{\infty},\PIF)\leq\cov(\HH_{F})$
        \end{center}
         are not provable in ZFC. On the one hand $\non(\HH_{F})\geq\add(\Meager)$ holds in ZFC and on the other $\ddd(\PIF,=^{\infty},\PIF)<\add(\Meager)$ holds in Hechler model (see \cite{BrendleLoweEventDiff}). Dually, the inequality $\cov(\HH_{F})\leq\cof(\Meager)<\bbb(\PIF,=^{\infty},\PIF)$ holds in the dual Hechler model.
    \end{remark}

\subsection{Forcing with $\mathcal{B}or\setminus\mathcal{H}_{F}$}
In this section we investigate the ideals $\mathcal{H}_{F}$ from forcing perspective. We will need the following framework.
\begin{definition}(Zapletal; \cite{Zapl1})
    Let $\J$ be a $\sigma$-ideal on a Polish space. Define then the idealized forcing notion $\mathbb{P}_{\J}=(\Borel\setminus\J,\subseteq)$ consisting of $\J$-positive Borel sets, such that $B\leq C$ means $B\subseteq C$. 
\end{definition}
In his book 'Idealized forcing' \cite{Zapl1}, Zapletal developed an extensive study of forcing notions of this form. Many classical forcing notions like Cohen or random forcing, Sacks forcing, Miller or Laver forcing can by seen in this framework with appropriate $\sigma$-ideals $\J$. In particular, Zapletal showed that if $\J$ is generated by closed sets (meaning that every set from $\J$ can be covered by an $F_{\sigma}$-set from $\J$), then the related idealized forcing has many desirable properties.
\begin{theorem}(Zapletal; \cite{Zapl1})\label{ZapletalClosedSetsForcingProperties}
    If $\J$ is generated by closed sets, then forcing $\mathbb{P}_{\J}$ is proper, has continuous reading of names and preserves Baire category.
\end{theorem}
Clearly the $\sigma$-ideal $\mathcal{H}_{F}$ is generated by closed sets and thus posesses the properties stated in the above theorem. However, Zapletal also showed that for a $\sigma$-ideal $\J$, the idealized forcing notion $\mathbb{P}_{\J}$ is $\Baire$-bounding if and only if compact set are dense in (any $\J$-positive set contains compact $\J$-positive set). We will show that this in not the case for the forcings $\mathbb{P}_{\mathcal{H}_{F}}$ for certain conditions. This means that c.s.i. of uncountable length of the forcing $\mathbb{P}_{\mathcal{H}_{F}}$ increases $\ddd$. We have the following result for eventually increasing function $F\in\Baire$.
    \begin{proposition}\label{NoTreeForHF}
    Compact subsets of $\PIF$ are not dense in the forcing $\mathbb{P}_{\HH_{F}}$.
    \end{proposition}
    \begin{proof}
        Let $G=\{y\in\PIF:\exists^{\infty}_{n}y(n)=0\}$. We will show that $G$ is $\mathcal{H}_{F}$-positive but any closed subset of $G$ is in $\mathcal{H}_{F}$. To see the former fix a set $H_{x,f}$. Without loss of generality we may assume that each interval $[f(n),f(n+1))$ consists of at least two elements. To see that $G\setminus H_{x,f}\neq\emptyset$ let $y\in\PIF$ be such that for every $n\in\omega$, $y(i_{n})=x(i_{n})$ for some $i_{n}\in[f(n),f(n+1))$ and $y(i_{n}+1)=0$. Then $y\in G\setminus H_{x,f}$.\\
        To see the latter let $D\subseteq G$ be closed and let $D=[T]$ for some pruned tree $T\subseteq Seq(F)$. As $D$ is compact, there are $A_{n}\subseteq T$ such that for every $n\in\omega$:
    \begin{enumerate}
        \item $A_{n}$ is a finite maximal antichain in $(T,\subseteq)$,
        \item $\max\{|\tau|:\tau\in A_{n}\}+1<\min\{|\tau|:\tau\in A_{n+1}\}$,
        \item for every $\tau\in A_{n}$ we have $\tau(|\tau|-1)=0$.
    \end{enumerate}
    Let now $f\in\Baire$ be such that $f(n)=\max\{|\tau|:\tau\in A_{n}\}$. We claim that $D\subseteq H_{x,f}$ where $x\in\PIF$ is the constant sequence with the value $0$. Clearly, if $x\in[T]$ then there is an increasing sequence $(k_{n})_{n\in\omega}$ such that $x|_{k_{n}}\in A_{n}$ so $x(k_{n}-1)=0$ and $k_{n}\in[f(n),f(n+1))$, which means that $x\in H_{x,f}$.
\end{proof}
\begin{corollary}
    In the iterated $\mathbb{P}_{\HH_{F}}$-model (countable support iteration of $\mathbb{P}_{\HH_{F}}$ of length $\omega_{2}$ over a model of CH) we have $\ddd=\omega_{2}$.
\end{corollary}
We know that $\mathbb{P}_{\mathcal{H}_{F}}$ adds unbounded reals below certain conditions. However, I do not know if this is true for any condition. It may turn out that below some conditions the forcing $\mathbb{P}_{\mathcal{H}_{F}}$ is $\Baire$-bounding. It would be interesting to establish more standard properties of the forcing notions $\mathbb{P}_{\mathcal{H}_{F}}$.
    \begin{question}\label{DoesPHFpreservesPpoints}
    Investigate the following properties of $\mathbb{P}_{\mathcal{H}_{F}}$:
        \begin{itemize}
            \item[--] Is there a condition $B\in\mathbb{P}_{\mathcal{H}_{F}}$ such that $\mathbb{P}_{\mathcal{H}_{F}}|_{B}$ is $\Baire$-bounding? What about the condition $B=\{x\in\PIF:\forall i\in\omega$ $\foralmostall n\in\omega$ $x(n)\neq i\}$?
            \item[--] Does $\mathbb{P}_{\mathcal{H}_{F}}$ preserve measure for a function $F\in\Baire$ such that $\Sigma_{i}\frac{1}{F(i)}<\infty$?
            \item[--] Does $\mathbb{P}_{\mathcal{H}_{F}}$ add splitting reals?
            \item[--] Does $\mathbb{P}_{\mathcal{H}_{F}}$ preserves P-points?
            \item[--] Is $\mathbb{P}_{\mathcal{H}_{F}}$ homogeneous? 
            \item[--] Does $\mathbb{P}_{\mathcal{H}_{F}}$ add Cohen reals? 
        \end{itemize}
    \end{question}

    \section{Uniformity of {$\E^{\star}$}}
    In this section, we will investigate the uniformity numbers of  the classes $\EM$ and $\Estar$. We already know the inclusions between the classes $\NA$, $\MA$, $\Mstar$, $\Nstar$ as well as $\EM$ and $\Estar$. The following diagram summarizes all ZFC provable inequalities between the uniformity numbers of these classes.
\begin{figure}[h!]\label{DiagramUniformitiesOfAdditiveClasses}
\centering
\tikzset{
  net node/.style = {draw, circle, minimum size=6mm},
  net edge/.style = {->},
  net cut/.style = {shorten >=-10mm, shorten <=-10mm, rounded corners=10mm, color=red},
  net cross/.style = {sloped, allow upside down, pos=.3},
}
\begin{tikzpicture}
        \newcommand{\edge}[5][]{\draw[net edge, #1] (#3) -- coordinate[net cross, name=#2] node[pos=.7, auto]{#5} (#4);}

        \node[name=addN] at ( 0, 0) {$\add(\Null)$};
        \node[name=addM] at ( 7, 0) {$\add(\Meager)$};
        \node[name=covM] at ( 9, 0) {$\cov(\Meager)$};
        \node[name=nonN] at ( 13, 0) {$\non(\Null)$};
        
        \node[name=b] at ( 7, 2) {$\bbb$};
        \node[name=d] at ( 9, 2) {$\ddd$};
        
        \node[name=cofN] at ( 13, 4) {$\cof(\Null)$};
        \node[name=cofM] at ( 9, 4) {$\cof(\Meager)$};
        \node[name=covN] at ( 0, 4) {$\cov(\Null)$};
        \node[name=nonM] at ( 7, 4) {$\non(\Meager)$};
        
        \node[name=NA] at ( 2, 1) {$\non(\NA)$};
        \node[name=nonMA]  at ( 5, 1) {$\non(\MA)$};
        \node[name=nonSMZ] at ( 11, 1) {$\non(\Mstar)$};
        \node[name=nonNstar]  at ( 2, 3) {$\non(\Nstar)$};
        \node[name=nonEM]  at ( 5, 3) {$\non(\EM)$};
        \node[name=Laver]  at ( 11, 3) {$\non(\Estar)$};

        \edge       {e1}  {addN} {addM} {}
        \edge       {e2}  {addN} {covN} {}
        \edge       {e3}  {addM} {covM} {}
        \edge       {e4}  {addM} {b} {}
        \edge       {e5}  {covM} {nonN} {}
        \edge       {e6}  {b} {d} {}
        \edge       {e7}  {b} {nonM} {}
        \edge       {e8}  {nonM} {cofM} {}
        \edge       {e9}  {d} {cofM} {}
        \edge       {e10}  {cofM} {cofN} {}
        \edge       {e11}  {covN} {nonM} {}
        \edge       {e12}  {nonN} {cofN} {}
        \edge       {e13}  {covM} {d} {}
        
        \edge       {e14}  {addN} {NA} {}
        \edge       {e15}  {NA} {nonMA} {}

        \edge       {e17}  {NA} {nonNstar} {}
        \edge       {e18}  {nonMA} {nonSMZ} {}
        \edge       {e19}  {nonEM} {Laver} {}
        \edge       {e20}  {nonSMZ} {Laver} {}
        \edge       {e21}  {covN} {nonNstar} {}
        \edge       {e22}  {nonEM} {nonM} {}
        \edge       {e23}  {covM} {nonSMZ} {}
        \edge       {e24}  {nonSMZ} {nonN} {}
        \edge       {e25}  {nonMA} {nonEM} {}
        \edge       {e26}  {nonNstar} {nonEM} {}
        \edge       {e27}  {Laver} {cofN} {}
        \edge       {e28}  {addM} {nonMA} {}

\end{tikzpicture}\caption{Cicho\'n's diagram with uniformities of additive classes.}
\end{figure}
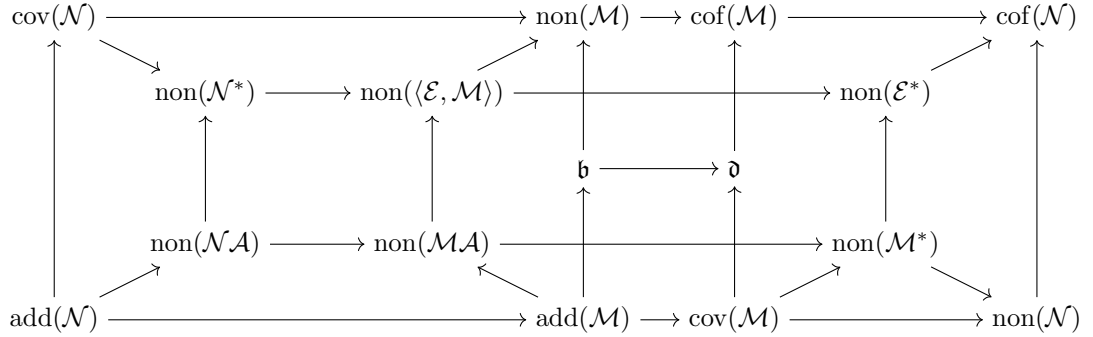

We will start with the following simple estimate of the uniformity of $\EM$. Regarding the cardinal invariant $\add(\E,\Meager)$, Brendle in \cite{BrendleBetween} showed that $\add(\E,\Meager)$ is in fact equal to $\min\{\bbb,\cov(\E)\}$ where the inequality $\add(\E,\Meager)\leq\bbb$ was shown by Miller (see \cite{MillerAddEvsB} or Theorem 2.6.10 in \cite{BJ}). We start with the following simple lower bound 
    \begin{proposition}\label{NonEMvsNonEstar}
    The following inequality holds:
    \begin{itemize}
        \item[--] $\min\{\bbb,\non(\Estar)\}\leq\non(\EM)\leq\non(\Estar)$
    \end{itemize}
    \end{proposition}
    \begin{proof}
    The inequality $\non(\EM)\leq\non(\Estar)$ follows from the inculsion $\EM\subseteq\Estar$. To show that $\min\{\bbb,\non(\Estar)\}\leq\non(\EM)$, assume $\kappa<\min\{\bbb,\non(\Estar)\}$. Let $X=\{x_{\alpha}:\alpha<\kappa\}\subseteq\Cantor$ and $E\in\E$. We want to show that $X+E$ is meager. Enumerate $\cantor$ as $\{\sigma_{n}:n\in\omega\}$. As $\kappa<\non(\Estar)$, for every $n\in\omega$ there is $x_{n}\in[\sigma_{n}]\setminus\bigcup_{\alpha<\kappa}(E+x_{\alpha})$. For each $\alpha<\kappa$ let $f_{\alpha}\in\Baire$ be such that for every $n\in\omega$ we have $[x_{n}|_{f_{\alpha}(n)}]\cap (E+x_{\alpha})=\emptyset$. As $\kappa<\bbb$, there is $g\in\Baire$ that dominates all $f_{\alpha}$'s. Clearly then
        \begin{center}
            $\bigcup_{\alpha<\kappa}(E+x_{\alpha})\subseteq\Cantor\setminus\bigcup_{n\in\omega}[x_{n}|_{g(n)}]$
        \end{center}
        and the latter is a meager set. This finishes the proof of the first inequality.
    \end{proof}
    We will now prove the main theorem of this section, i.e. we will show that $\non(\Estar)$, which is equal to $\cov_{t}(\E)$, is close to the cardinal invariant related to the Laver property. To show this, we will need the following lemma.
    \begin{lemma}\label{LemmaTranslatinalE}
        There is $E\in\E$ such that for every $Y\subseteq\Cantor$, if $E+Y\in\E$, then for some $y\in\Cantor$ we have that $Y\subseteq E+y$.
    \end{lemma}
    \begin{proof}
           In this proof we are going to use the following notation: if $R\subseteq\cantor$ it a finite tree, by $[R]$ we will denote the union of basic open sets given by the terminal nodes of $R$, i.e. $[R]=\bigcup\{[\sigma]:\sigma$ is a terminal node of $R\}$.\\
           
           Let $g\in\Baire$ be defined as $g(n)=n\cdot2^{n^{2}}$ and let $f\in\Baire$ as  $f(0)=0$ and $f(n+1)= f(n)+n+g(n)\cdot 2^{f(n)}$ for each $n\in\omega$. We start with the following claim
           \begin{claim}\label{claimWithConditions}
               For every $n\in\omega$ there is 
               \begin{itemize}
                   \item[--] a set $B_{n}\subseteq2^{[f(n),f(n+1))}$,
                   \item[--] and a collection $\{A_{s}\subseteq2^{[f(n),f(n+1))}:s\in2^{f(n)}\}$
               \end{itemize}
                such that for every $n\in\omega$ the following conditions hold:
               \begin{itemize}
               \setlength\itemsep{0.4em}
                   \item[(a)] $\frac{|A_{s}|}{2^{f(n+1)-f(n)}}= \frac{1}{2^{n}}(1-\frac{1}{2^{g(n)}})$,
                   \item[(b)] $\frac{|B_{n}|}{2^{f(n+1)-f(n)}}= \frac{1}{2^{n}}(1-(1-\frac{1}{2^{2^{f(n)}}})^{g(n)})$,
                   \item[(c)] for all $s, t\in2^{f(n)}$ and $\sigma,\tau\in 2^{[f(n),f(n+1)-n)}$, if $s\neq t$ then the sets\\ $A_{s}|_{[f(n),f(n+1)-n)}+\sigma$, and $A_{t}|_{[f(n),f(n+1)-n)}+\tau$ are probabilistically independent,
                   \item[(d)] for every $ Y\subseteq2^{[f(n),f(n+1))}$, if $|Y|\leq g(n)$, then $Y+(2^{[f(n),f(n+1))}\setminus B_{n})\neq 2^{[f(n),f(n+1))}$.
               \end{itemize}
           \end{claim}
           \begin{proof}
               Split the interval $[f(n),f(n+1))$ into $(g(n)+1)$-many intervals $J_{k}$, $k\leq g(n)$ where $|J_{g(n)}|=n$ and for $k<g(n)$, $|J_{k}|=2^{f(n)}$ i.e. $J_{k}=\{j^{k}_{s}:s\in 2^{f(n)}\}$. Define now 
               \begin{center}
                   $A_{s}=\{\sigma\in 2^{[f(n),f(n+1))}:\sigma|_{J_{g(n)}}\equiv0$ and $\exists k<g(n)$ $\sigma(j^{k}_{s})=0\}$,\\
                   $B_{n}=\{\sigma\in 2^{[f(n),f(n+1))}:\sigma|_{J_{g(n)}}\equiv0$ and $\exists k<g(n)$ $\sigma|_{J_{k}}\equiv 0\}$
               \end{center}
               Clearly $B_{n}\subseteq A_{s}$ for every $s\in 2^{f(n)}$. Points $(a)$ and $(b)$ are standard calculations and are similar. Therefore we deal only with the second one:
               \begin{center}
                    $\vspace{8px}\frac{|B_{n}|}{2^{f(n+1)-f(n)}}=\frac{|\{\sigma\in2^{[f(n),f(n+1)-n)}:\textbf{ }\exists k<g(n)\textbf{ }\sigma|_{J_{k}}\equiv0\}|}{2^{f(n+1)-f(n)}}=\newline\vspace{8px}=\frac{2^{f(n+1)-f(n)-n}-|\{\sigma\in2^{[f(n),f(n+1)-n)}:\textbf{ }\forall k<g(n)\textbf{ }\sigma|_{J_{n}}\neq0\}|}{2^{f(n+1)-f(n)}}=\frac{2^{f(n+1)-f(n)-n}-\Pi_{k<g(n)}(2^{2^{f(n)}-1})}{2^{f(n+1)-f(n)}}=\frac{1}{2^{n}}[1-(1-\frac{1}{2^{2^{f(n)}}})^{g(n)}]$
               \end{center}          
               We check the last condition. Let $Y=\{\sigma_{k}:k<g(n)\}$ and define
               \begin{center}
                    $\sigma^{*}=\sigma_{0}|_{J_{0}}$$^{\frown}...^{\frown}\sigma_{g(n)-1}|_{J_{g(n)-1}}$$^{\frown}0^{n}\in 2^{[f(n),f(n+1))}$
               \end{center}
               We claim that $\sigma^{*}\notin Y+(2^{[f(n),f(n+1))}\setminus B_{n})$. This is because, if $\sigma^{*}=\sigma_{k}+\tau$, where $k<g(n)$ and $\tau\notin B_{n}$, then for every $k<g(n)$ there is $i_{k}\in J_{k}$ such that $\tau(i_{k})=1$. This gives the contradicton, as $\sigma^{*}(i_{k})=\sigma_{k}(i_{k})+\tau(i_{k})=\sigma^{*}(i_{k})+1$. This finishes the proof of the claim.
           \end{proof}
           Define now the sets
           \begin{center}
               $E=\{x\in \Cantor:\foralmostall n\in\omega$ $x|_{[f(n),f(n+1))}\in A_{x|_{f(n)}}\}$,\\
               $F=\{x\in\Cantor:\foralmostall n\in\omega$ $x|_{[f(n),f(n+1))}\in B_{n}\}$.
           \end{center}
           As $B_{n}\subseteq A_{s}$, we have that $F\subseteq E$, and both $E$ and $F$ are $F_{\sigma}$'s of measure zero because of points (a) and (b). We claim that the set $E$ is the one we are looking for, while the set $F$ will have auxiliary role. Suppose that $Y\subseteq \Cantor$ is such that $E+Y\in \E$.
           \begin{claim}
               There is a slalom $S=(S_{n})_{n\in\omega}$ such that for every $n\in\omega$ we have:
               \begin{itemize}
                   \item[--] $S_{n}\subseteq2^{[f(n),f(n+1))}$,
                   \item[--] $|S_{n}|\leq g(n)$
                   \item[--] $Y\subseteq\{x\in\Cantor:\foralmostall n\in\omega$ $x|_{[f(n),f(n+1))}\in S_{n}\}$
               \end{itemize}
           \end{claim}
           \begin{proof}
               Let $R\subseteq\cantor$ be a tree such that $[R]=\{x\in\Cantor:\forall n\in\omega\textbf{ }x|_{[f(n),f(n+1))}\in A_{x|_{f(n)}}\}$. It follows that $[R]+Y\in\E$. Let $\{T_{m}:m\in\omega\}$ be such collection of trees on $\cantor$ that $[R]+Y\subseteq\bigcup_{n}[T_{m}]$ and every $[T_{m}]$ is of measure zero. By Baire category theorem we have that 
               \begin{center}
                   $\forall y\in Y$ $\exists\sigma\in R$ $\exists m\in\omega$ $[R|_{\sigma}]+y\subseteq[T_{m}]$
               \end{center}
            Thus, if $Y_{\sigma,m}=\{y\in Y:[R|_{\sigma}]+y\subseteq[T_{m}]\}$, then $Y=\bigcup_{\sigma,m}Y_{\sigma,m}$. To construct the required slalom $S=(S_{n})_{n\in\omega}$, for every $\sigma$ and $m$ we will first construct a slalom $S^{\sigma,m}=(S_{n}^{\sigma,m})_{n\in\omega}$ with $S_{n}^{\sigma,m}\subseteq2^{[f(n),f(n+1))}$ and $|S_{n}^{\sigma,m}|\leq 2^{n^{2}}$ such that
            \begin{center}
                $Y_{\sigma,m}\subseteq\{x\in\Cantor:\foralmostall n\in\omega $ $y|_{[f(n),f(n+1))}\in S_{n}^{\sigma,m}\}$
            \end{center}
            When these slaloms are constructed, we define $S=(S_{n})_{n\in\omega}$ as a diagonal union over $S^{\sigma,m}$'s as follows: enumerate all pairs $(\sigma,m)\in R\times \omega$ as $\{(\sigma_{k},m_{k}):k\in\omega\}$ and then for every $n\in\omega$ define $S_{n}=\bigcup_{k\leq n}S^{\sigma_{k},m_{k}}_{n}$. Then, we clearly have that $|S_{n}|\leq n\cdot 2^{n^{2}}=g(n)$ and also $Y\subseteq\{x\in\Cantor:\foralmostall n\in\omega$ $x|_{[f(n),f(n+1))}\in S_{n}\}$ This will finish the proof of the claim.\\
            To construct the slalom $S^{\sigma,m}$ fix $\sigma$ and $m$. For every $n\in\omega$ define
            \begin{center}
                $S_{n}^{\sigma,m}=\{y|_{[f(n),f(n+1))}:y\in Y_{\sigma,m}\}$
            \end{center}
            We only have to check that $|S_{n}^{\sigma,m}|\leq 2^{n^{2}}$. Let $\widetilde{S}_{n}^{\sigma,m}=\{y|_{f(n+1)}:y\in Y_{\sigma,m}\}$. As $|S_{n}^{\sigma,m}|\leq|\widetilde{S}_{n}^{\sigma,m}|$ we will only have to estimate the size of $\widetilde{S}_{n}^{\sigma,m}$. Let $N\in\omega$ be large enough so that $\frac{|T_{m}\cap 2^{n}|}{2^{n}}<\frac{1}{2}$ for every $n>N$. Then for such $n$'s we have that:
            \begin{center}
                $\bigcup_{y\in Y_{\sigma,m}}(R\cap 2^{f(n+1)})+y|_{f(n+1)}=\bigcup_{\tau\in \widetilde{S}_{n}^{\sigma,m}}(R\cap 2^{f(n+1)})+\tau\subseteq T_{m}\cap 2^{f(n+1)}$
            \end{center}
            Thus we have that:
            \begin{center}
                $\vspace{8px}\frac{1}{2}\geq\mu(\bigcup_{y\in Y_{\sigma,m}}R\cap 2^{f(n+1)}+y|_{f(n+1)})=1-\mu(\bigcap_{\tau\in\widetilde{S}_{n}^{\sigma,m}}\Cantor\setminus[R\cap2^{f(n+1)}+\tau])\geq\newline \vspace{8px}\geq 1-\Pi_{\tau\in\widetilde{S}_{n}^{\sigma,m}}\mu(\Cantor\setminus[R\cap2^{f(n)}+\tau])=1-\Pi_{\tau\in\widetilde{S}_{n}^{\sigma,m}}[1-\mu([R\cap2^{f(n+1)}]+\tau)]=\newline=1-\Pi_{\tau\in\widetilde{S}_{n}^{\sigma,m}}[1-\frac{|R\cap2^{f(n+1)}|}{2^{f(n+1)}}]=1-[1-\Pi_{i=0}^{n}\frac{|A_{s_{i}}|}{2^{f(i+1)-f(i)}}]^{|\widetilde{S}_{n}^{\sigma,m}|}=1-[1-\Pi_{i=0}^{n}\frac{1}{2_{i}}(1-\frac{1}{2^{g_{i}}})]^{|\widetilde{S}_{n}^{\sigma,m}|}$
            \end{center}
            and then
            \begin{center}
                $\vspace{8px}\frac{1}{2}\leq(1-\Pi_{i=0}^{n}\frac{1}{2_{i}}(1-\frac{1}{2^{g_{i}}}))^{|\widetilde{S}_{n}^{\sigma,m}|}\leq(1-\Pi_{i=0}^{n}\frac{1}{2_{i}}(1-\frac{1}{2}))^{|\widetilde{S}_{n}^{\sigma,m}|}=\newline =[1-\frac{1}{2^{\Sigma_{i=0}^{n}(i+1)}}]^{|\widetilde{S}_{n}^{\sigma,m}|}\leq(1-\frac{1}{2^{n^{2}}})^{|\widetilde{S}_{n}^{\sigma,m}|}$
            \end{center}
            It follows that
            \begin{center} $|\widetilde{S}_{n}^{\sigma,m}|\leq2^{n^{2}}$
            \end{center}

           \end{proof}
           Now we will make use of the set $F$. We have the following claim.
           \begin{claim}
               $Y+(\Cantor\setminus F)\neq \Cantor$
           \end{claim}
           \begin{proof}
               According to the last item of the claim \ref{claimWithConditions}, for every $n\in\omega$, there is $\sigma^{*}_{n}\in 2^{[f(n),f(n+1))}$ such that $\sigma^{*}_{n}\notin S_{n}+(2^{[f(n),f(n+1))}\setminus B_{n})$. Define then $z=\sigma^{*}_{n}$$^{\frown}\sigma^{*}_{n}$$^{\frown}...$ in $\Cantor$. We claim that $z\notin Y+(\Cantor\setminus F)$. Suppose towards a contradiction that $z=y+x$ where $y\in Y$ and $x\notin F$. Then there is $n\in\omega$ such that $y|_{[f(n),f(n+1))}\in S_{n}$ and $x|_{[f(n),f(n+1))}\notin B_{n}$. This contradicts the fact that $(y+x)|_{[f(n),f(n+1))}=z|_{[f(n),f(n+1))}=\sigma^{*}_{n}\notin S_{n}+(2^{[f(n),f(n+1))}\setminus B_{n})$.
           \end{proof}
           Once the last claim is proven, the rest of the lemma easily follows. Let $z\notin Y+(\Cantor\setminus F)$. Then clearly we have $z+Y\subseteq F\subseteq E$ which finishes the proof.
    \end{proof}
    Granted the lemma above, we are ready to prove the main theorem of this section. Namely, we estimate the uniformity of $\Estar$ using the slalom numbers related to the Laver property. The following Theorem strengthens the estimate for $\non(\mathcal{SM})$ given by the theorem \ref{SMestimatesnon}. 
    \begin{theorem}
        The following inequalities hold:
        \begin{itemize}
            \item[--] $\non(\Estar)\leq\sup\{\ddd(\PIF,\in^{*},\mathcal{S}lm(F,b)):F\in\Baire,$ and $b\leq F\}$,
            \item[--] $\non(\Estar)\geq\min\{\ddd(\PIF,\in^{*},\mathcal{S}lm(F,b)):F\in\Baire,$ and $b\leq F$, $\Sigma_{n}\frac{b(n)}{F(n)}<\infty\}$.
        \end{itemize}

    \end{theorem}
    \begin{proof}
         To show the first inequality assume that $F,b\in\Baire$ and that $\{S_{\alpha}:\alpha<\kappa\}$ is a collection of $(F,b)$-slaloms that witness for the $\ddd(\PIF,\in^{*},\mathcal{S}lm(F,b))$. Without loss of generality, we may enlarge $F$ and assume that $F(n)=2^{i_{n}}$ for some $i_{n}\in\omega$. Let now $(I_{n})_{n\in\omega}$ be such interval partition that $|I_{n}|=i_{n}$ for every $n\in\omega$. For $\alpha<\kappa$ define the set
        \begin{center}
            $F_{\alpha}=\{x\in\Cantor:\foralmostall n\in\omega$ $x|_{I_{n}}\in S_{\alpha}(n)\}$
        \end{center}
        By the choice of $\{S_{\alpha}:\alpha<\kappa\}$ we have that $\bigcup_{\alpha<\kappa}F_{\alpha}$ covers $\Cantor$. Let $E\in\E$ be the set from the lemma \ref{LemmaTranslatinalE}. We need the following.
        \begin{claim}
            For every $\alpha<\kappa$ we have $F_{\alpha}+E\in\E$.
        \end{claim}
        \begin{proof}
            Recall the construction of the set $E$ from the lemma \ref{LemmaTranslatinalE}. Notice first that in the construction, we may assume that every interval $(I_{k})_{k\in\omega}$ is a union of the intervals $[f(n),f(n+1))$. Then we have that 
            \begin{center}
                $F_{\alpha}+E\subseteq\{x\in\Cantor:\foralmostall n\in\omega$ $x|_{I_{n}}\in A_{x|_{f(n)}}+S_{\alpha}(n)|_{[f(n),f(n+1))}\}$
            \end{center}
            and the latter is an $F_{\sigma}$-set of measure zero. This finishes the proof of the claim.
        \end{proof}
        By the lemma \ref{LemmaTranslatinalE}, for every $\alpha<\kappa$ there is $y_{\alpha}\in \Cantor$ such that $F_{\alpha}\subseteq E+y_{\alpha}$. We claim that the set $Y=\{y_{\alpha}:\alpha<\kappa\}$ is not in $\Estar$ because $Y+E=\Cantor$. To see this, let $x\in\Cantor$ be arbitrary. Then for some $\alpha<\kappa$, we have that
        \begin{center}
            $x\in F_{\alpha}\subseteq E+y_{\alpha}\subseteq E+Y$
        \end{center}
        This finishes the proof of the first inequality.\\
        To prove the second inequality, assume that for some $X\subseteq\Cantor$ and some $E\in\E$ we have $\Cantor=\bigcup_{x\in X}E+x$. Using the characterization \ref{CharacterizationE} of the $\sigma$-ideal $\E$, we may assume that for some interval partition $(I_{n})_{n\in\omega}$ and a slalom $(S_{n})_{n\in\omega}$ such that $S_{n}\subseteq 2^{I_{n}}$, $\frac{|S_{n}|}{2^{|I_{n}|}}\leq\frac{1}{2^{n}}$, the set $E$ is of the form $\{x\in\Cantor:\foralmostall n\in\omega$ $x|_{I_{n}}\in S_{n}\}$. Let $F,b\in\Baire$ be such that $F(n)=2^{|I_{n}|}$ and $b(n)=\frac{F(n)}{2^{n}}$. In this way, every element of $\PIF$ naturally corresponds with a unique element of $\Cantor$.
        Now, if for $x\in X$, the slalom $(S^{x}_{n})_{n\in\omega}$ is defined as
        \begin{center}
            $S^{x}_{n}=\{i\in F(n): i$ corresponds with an element of $S_{n}+x|_{I_{n}}\}$
        \end{center}
        then clearly the collection $\{(S^{x}_{n})_{n\in\omega}:x\in X\}$ is a witness for $\ddd(\PIF,\in^{*},\mathcal{S}lm(F,b))$.
        This finishes the proof of the second inequality and the theorem.
        \end{proof}
        As a result of the above Theorem together with the fact that $\non(\Estar)=\cov_{t}(\E)$ and Theorem \ref{Kada}, we get that $\cov_{t}(\E)\leq\sup\{\ddd(\PIF,\in^{*},\mathcal{S}lm(F,b)):F\in\Baire,$ and $b\leq F\}$, which strengthens the previously known inequality $\cov_{t}(\E)\leq\cof(\Null)$ obtained by Elekes and Stepr\=ans, in \cite{ElekesSteprans}.\\
        In forcing language, we can formulate the previous Theorem as follows:
        \begin{corollary}
            If a forcing notion $\mathbb{P}$ has the Laver property, then $\Vdash_{\mathbb{P}}"\Cantor\cap\mathbb{V}\notin\Estar"$
        \end{corollary}
        At this point, we know the value of $\add_{t}(\E)$ and we have the above estimate for $\cov_{t}(\E)$. As $\EA=\MA$, we have that $\add_{t}(\E)=\non(\mathcal{EA})=\non(\MA)=\add_{t}(\Meager)$ and that $\cov_{t}(\E)=\non(\Estar)$ is equal to the Laver property number. We also know that $\add_{t}(\E,\Null)=\non(\EN)=\non(\Mstar)$ is the cardinal invariant related to the bounded infinitely equal reals. This motivates the following.
        \begin{question}\label{QuestionSlalomNumbersVsNons}
        Are the following equalities true in ZFC?
        \begin{itemize}
            \item[--] $\non(\EM)=\sup\{\ddd(\PIF,=^{\infty},\PIF):F\in\Baire\}$,
            \item[--] $\non(\Estar)=\sup\{\ddd(\PIF,\in^{*},Slm(F)):F\in\Baire\}$.
        \end{itemize} 
        \end{question}
         Regarding additivities of the additive classes, we alredy mentioned that $\add(\NA)=\non(\NA)=\min\{\bbb(\PIF,=^{\infty},\PIF):F\in\Baire\}$ and $\add(\NA)\leq\add(\SMZ)$. Regarding classes $\MA$ and $\EM$ we have the following:
        \begin{question}
        Are the following equalities true in ZFC?
        \begin{itemize}
            \item[--] $\add(\MA)=\add(\Meager)$ (Cardona, Mejía and Rivera-Madrid),
            \item[--] $\add(\EM)=\add(\Meager)$.
        \end{itemize} 
        \end{question}
        Regarding the additivity of $\Estar$, Bartoszyński and Shelah showed (\cite{BartShNstarIsNotIdeal}) that under the assumption of Continuum Hypothesis, $\mathcal{SM}=\Nstar$ does not form an ideal. In an earlier version of this article I gave as a problem whether the collection $\Estar$ forms an ideal? Tomasz Weiss informed be that $\Estar$ is in fact a $\sigma$-ideal (see \cite{WeissNew}).\\
        
         We finish this section with considerations regarding the following two cardinal invariants. 
        \begin{definition}
        For a $\sigma$-ideal $\I$ on $\Cantor$ define: 
            \begin{itemize}
                \item[--] $\add^{*}_{t}(\I)=\min\{|\mathcal{F}|:\mathcal{F}\subseteq\I$ and $\neg\exists X\in\I$ $\forall F\in\mathcal{F}$ $\exists x\in\Cantor$ $F+x\subseteq X\}$
                \item[--] $\mathrm{cof}^{*}_{t}(\I)=\min\{|\mathcal{F}|:\mathcal{F}\subseteq\I$ and $\forall Y\in\I$ $\exists Y'\in\mathcal{F}$ $\exists x\in\Cantor$ $Y\subseteq Y'+x\}$
            \end{itemize}
        \end{definition}
        Clearly $\add(\I)\leq\add^{*}_{t}(\I)$ and $\cof^{*}_{t}(\I)\leq\cof(\I)$ hold for any translational ideal $\I$. In case of measure zero and meager sets the following result was obtained by Pawlikowski.
        \begin{theorem}(Pawlikowski; \cite{PawlikowskiTransitiveInvariants})
            The following holds:
            \begin{itemize}
                \item[--] $\add^{*}_{t}(\Null)=\add(\Null)$,
                \item[--] $\cof^{*}_{t}(\Null)=\cof(\Null)$,
                \item[--] $\add^{*}_{t}(\Meager)=\bbb$,
                \item[--] $\cof^{*}_{t}(\Meager)=\ddd$.
            \end{itemize}
        \end{theorem}
        Regarding these cardinal invariants for the $\sigma$-ideal $\E$ we have:
        \begin{theorem}
        The following inequalities hold:
        \begin{itemize}
            \item[--] $\add(\Meager)\leq\add^{*}_{t}(\E)\leq\bbb$,
            \item[--] $\ddd\leq\cof^{*}_{t}(\E)\leq\cof(\Meager)$.
            \end{itemize}
        \end{theorem}
        \begin{proof}
            First, we notice that $\add(\Meager)\leq\add^{*}_{t}(\E)$ and $\cof^{*}_{t}(\E)\leq\cof(\Meager)$ easily follows from the observation above and the fact that $\add(\E)=\add(\Meager)$ and $\cof(\E)=\cof(\Meager)$ (see Theorem \ref{CardInvsofE}).\\
            To show that $\add^{*}_{t}(\E)\leq\bbb$ let $\{f_{\alpha}:\alpha<\bbb\}\subseteq\Baire$ be an unbounded family of increasing functions. For every $\alpha<\bbb$ let $T_{\alpha}\subseteq\cantor$ be the tree that splits only on levels $n\in\omega\setminus rng(f_{\alpha})$ and picks $0$ elsewhere. Clearly, $E_{\alpha}=[T_{\alpha}]\in E$ and for every $n\in\omega$ we have 
            \begin{center}
                $\frac{|T_{\alpha}\cap 2^{f_{\alpha}(n)}|}{2^{f_{\alpha}(n)}}=\frac{1}{2^{n}}$
            \end{center}
            Assume that $E=\bigcup_{n}F_{n}$ is such that each $F_{n}$ is closed set of measure zero and that for every $\alpha<\bbb$ there is $x_{\alpha}\in \Cantor$ such that $E_{\alpha}+x_{\alpha}\subseteq E$. By Baire category theorem for every $\alpha$ there are $\sigma\in \cantor$ and $n_{\alpha}\in\omega$ such that $[T_{\alpha}+x_{\alpha}|_{\sigma_{\alpha}}]\subseteq F_{n_{\alpha}}$. Without loss of generality, we may assume that all $\sigma_{\alpha}$'s and $n_{\alpha}$'s are the same and are equal $\sigma$ and $n$. Let $T\subseteq\cantor$ be such tree that $F_{n}=[T]$ and let $g\in\Baire$ be such that for every $n\in\omega$ for every $m>g(n)$ we have
            \begin{center}
                $\frac{|T\cap 2^{m}|}{2^{m}}<\frac{1}{2^{n+|\sigma|}}$
            \end{center}    
            Let $\alpha<\bbb$ be such that $g(n)<f_{\alpha}(n)$ for infinitely many $n$'s. But then for every such $n>|\sigma|$ we get 
            \begin{center}
                $\frac{1}{2^{n+|\sigma|}}
                =\frac{|(T_{\alpha}+x_{\alpha})|_{\sigma}|}{2^{f_{\alpha}(n)}}
                \leq\frac{|T\cap 2^{f_{\alpha}(n)|}}{2^{f_{\alpha}(n)}}
                <\frac{1}{2^{f_{\alpha}(n)}}$
            \end{center}
            which is the final contradiction.\\
            Next we show that $\ddd\leq\cof^{*}_{t}(\E)$. Let $\{E_{\alpha}:\alpha<\kappa\}\subseteq\E$ where $\kappa<\ddd$. We will construct a set $E\in\E$ such that no translation of any $E_{\alpha}$ covers $E$. By Proposition \ref{CharacterizationE}, without loss of generality, we may assume that for every $\alpha<\kappa$ there is an interval partition $(I^{\alpha}_{n})_{n\in\omega}$ and a slalom $S^{\alpha}=(S^{\alpha}(n))_{n\in\omega}$ such that 
            \begin{center}
                $E_{\alpha}\subseteq\{x\in\Cantor:\foralmostall n\in\omega\textbf{ }x|_{I^{\alpha}_{n}}\in S^{\alpha}(n)\}$
            \end{center}
            As $\kappa<\ddd$, there is an interval partition $(J_{k})_{k<\omega}$ not dominated by any $(I^{\alpha}_{n})_{n\in\omega}$. For $i=0,1$ let $T_{i}\subseteq\cantor$ be such tree that $\sigma\in split(T_{i})$ if and only if $|\sigma|\notin\{\min J_{2k+i}:k<\omega\}$. Define $E=[T_{0}]\cup[T_{1}]$. Clearly $E\in\E$. Suppose that $E\subseteq E_{\alpha}+y=\{x\in\Cantor:x|_{I^{\alpha}_{n}}\in S^{\alpha}(n)+y|_{I^{\alpha}_{n}}\}$. Then there are infinitely many intervals $I^{\alpha}_{n}$ covered by $J_{k}\cup J_{k+1}$. Let $i\in 2$ be the parity of infinitely many of these $k$'s. Then it is easy to construct $z\in[T_{i}]$ such that $z|_{I^{\alpha}_{n}}\notin S^{\alpha}(n)+y|_{I^{\alpha}_{n}}$ for infinitely many $n\in\omega$. Clearly, then $z\in E\setminus (E_{\alpha}+y)$.
        \end{proof}
        I do not know the precise values of these invariants.
        \begin{question}
            Is $\add^{*}_{t}(\E)=\add(\Meager)$ and $\cof^{*}_{t}(\E)=\cof(\Meager)$?
        \end{question}
        It would be interesting to investigate the covering numbers and cofinalities of the additive classes $\EM$ and $\Estar$. Regarding these two invariants of the classes $\MA$, $\mathcal{SMZ}$ and $\mathcal{SM}$, several interesting results were proven in \cite{coveringstrongmeasurecanbeaboveeverythingelse}, \cite{MejiaDirectedSums}, \cite{MejiaCardonaMA}, \cite{CardonaMejiaMoreOnCofCovSMZ} and \cite{BrendleCardonaMejia}.\\

    \section{Every {$\mathcal{E}$}-Luzin set is in {$\Estar$}}\label{SpecialSmallSetsSection}
    \begin{definition}
    Let $\I$ be a $\sigma$-ideal on $\Cantor$. An uncountable set $X\subseteq\Cantor$ is $\I$-Luzin set if the intersection $X\cap F$ is countable for any $F\in\I$.
\end{definition}
    In the sense of the above definition, the classical Luzin set is just an $\Meager$-Luzin set, and classical Sierpiński set is just an $\Null$-Luzin set. The reader interested in $\J$-Luzin sets may consult the Miller's survey article \cite{SpecialSubsets}. \\
    It is not difficult to see that $\cof(\I)=\omega_{1}$ implies the existence of $\I$-Luzin set. It is well-known, that every Luzin set in of strong measure zero (is in $\Mstar$). 
    
    The dual question, if every Sierpiński set is strongly meager, was asked by Galvin. This question was answered by Pawlikowski.
    \begin{theorem}(Pawlikowski; \cite{PawlikowskiSierpinski})\label{SierpIsStrMeager}
        Every Sierpiński set is strongly meager.
    \end{theorem}
    To show that every $\E$-Luzin set is in $\Estar$, we will need the following property. 
    \begin{definition}
        A set $X\subseteq\Cantor$ has the Menger property if for every continuous function $f:X\rightarrow\Baire$, the image $f[X]$ is not dominating. 
    \end{definition}
    This property (in many different forms) has been investigated by many researchers in the context of small sets or selection principles. It is usually defined by means of open covers, however, here we will only use continuous functions.
    \begin{lemma}\label{SierpińskiHASMenger}
        Every $\E$-Luzin set has the Menger property.
    \end{lemma}
    \begin{proof}
        Let $f:\Cantor\rightarrow\Baire$ be a continuous function and let $L\subseteq\Cantor$ be an $\E$-Luzin set. For every $n\in\omega$ pick $g(n)\in\omega$ be such that the clopen set 
        \begin{center}
            $F_{n}=:\bigcup_{m>g(n)}f^{-1}[\{x\in\Baire:x(n)=m\}]$
        \end{center}
        is of measure $<\frac{1}{2^{n}}$. Then the set $E=:\{x\in\Cantor:\foralmostall n\in\omega$ $x\in F_{n}\}$ is $F_{\sigma}$-set of measure zero. As $L\cap E$ is countable, we have that 
        \begin{center}
            $L\setminus E\subseteq\{x\in\Cantor:\existsinfty n\in\omega$ $f(x)(n)\leq g(n)\}$
        \end{center}
        It follows that $f[L\setminus E]$ does not dominate $g\in\Baire$. We can then improve $g$ so we dominate this countably many exceptions. This finishes the proof.
    \end{proof}
    I do not know what happens is the function in the Lemma above is Borel instead of continuous (such property is sometimes referred to as Borel-Menger property).
    \begin{question}
        Is it true that for any $\E$-Luzin set $X\subseteq\Cantor$ and any Borel function $f:X\rightarrow\Baire$ the image $f[X]$ does not form a dominating family?
    \end{question}
    We will now prove the main theorem of this section. The idea of the proof is an adaptation of the proof of Pawlikowski's Theorem \ref{SierpIsStrMeager}.
    \begin{theorem}
        Every $\E$-Luzin is in $\Estar$.
    \end{theorem}
    \begin{proof}
        Let $L\subseteq\Cantor$ be an $\E$-Luzin set and let $E\in\E$. We are looking for a point $z$ outside of the set $L+E$. Assume that $\{T_{n}:n\in\omega\}$ is a sequence of trees on $\cantor$ such that $E\subseteq\bigcup_{n\in\omega}[T_{n}]$ and every $[T_{n}]$ is of measure zero. Let $Q=\{q_{n}:n\in\omega\}\in[\omega]^{\omega}$ be such that for every $n\in\omega$ for every $m>q_{n}$ we have $\frac{|T_{n}\cap2^{m}|}{2^{m}}<\frac{1}{2^{n}}$. For $l\in\Cantor$ and $n\in\omega$ we define
        \begin{center}
            $S^{l}_{n}=(T_{n}\cap 2^{q_{n}})+l|_{q_{n}}$
        \end{center}
        Note that $L+E\subseteq\bigcup_{l\in L}E_{l}$ where $E_{l}=\{y\in\Cantor:\foralmostall_{n\in\omega}$ $ y|_{q_{n}}\in S^{l}_{n}\}\in\E$. For $n,k\in\omega$ and $l\in L$ define 
        \begin{equation*}
            \begin{aligned}
            F^{l}_{n}(k)=\{\langle\sigma_{0},...,\sigma_{k}\rangle\in(2^{q_{n}})^{k+1}:&\exists 0\leq \tilde{q}_{0}\leq...\leq\tilde{q}_{k-1}\leq q_{n} \text{ with all } \tilde{q}_{i}\in Q\\
             & \text{and }\sigma_{0}|_{\tilde{q}_{0}}\cup \sigma_{1}|_{[\tilde{q}_{0},\tilde{q}_{1})}\cup...\cup \sigma_{k}|_{[\tilde{q}_{k-1},q_{n})}\in S^{l}_{n}\}
            \end{aligned}
        \end{equation*}
        In such case will say that such the sequence $\langle\sigma_{0},...,\sigma_{k}\rangle\in(2^{q_{n}})^{k+1}$ has a diagonal in $S^{l}_{n}$. Define now the following sets
        \begin{center}
            $F^{l}(k)=\{\langle x_{0},...,x_{k}\rangle\in(\Cantor)^{k+1}: \foralmostall n\in\omega$ $\langle x_{0}|_{q_{n}},x_{1}|_{q_{n}},...,x_{k}|_{q_{n}}\rangle\in F^{l}_{n}(k)\}$,
        \end{center}
        \begin{center}
            $F^{l}=\{\langle x_{0},x_{1},....\rangle\in(\Cantor)^{\omega}:\exists k\in\omega$ $\langle x_{0},...,x_{k}\rangle\in F^{l}(k)\}$
        \end{center}
        Let now $F\subseteq\Cantor\times(\Cantor)^{\omega}$ be such that the vertical section of $F$ above $l\in\Cantor$ is equal $F^{l}$ i.e. $F=\{(l,\langle x_{k}:k\in\omega\rangle)\in\Cantor\times(\Cantor)^{\omega}:\langle x_{k}:k\in\omega\rangle\in F^{l}\}$. We will need the following claim.
        \begin{claim}
            $F$ is an $F_{\sigma}$-set with all vertical sections of measure zero.
        \end{claim}
        \begin{proof}
            Clearly $F$ a an $F_{\sigma}$-set. To see that its sections are of measure zero let us first estimate the sizes of sets $F^{l}_{n}(k)$. Let $l\in\Cantor$. First notice that there are less than $(1+n)^{k}$ choices for the sequence $0\leq \tilde{q}_{0}\leq...\leq\tilde{q}_{k-1}\leq q_{n} \text{ with all } \tilde{q}_{i}\in Q$. There are also $|S^{l}_{n}|$ choices for the diagonal sequence $\sigma_{0}|_{\tilde{q}_{0}}\cup \sigma_{1}|_{[\tilde{q}_{0},\tilde{q}_{1})}\cup...\cup \sigma_{k}|_{[\tilde{q}_{k-1},q_{n})}$. It follows that
            \begin{center}
                $|F^{l}_{n}(k)|\leq(1+n)^{k}\cdot |S^{l}_{n}|\cdot 2^{(q_{n}-q_{0})}\cdot2^{(q_{n}-q_{1}+q_{0})}\cdot...\cdot2^{(q_{n}-q_{n}+q_{k-1})}=(1+n)^{k}\cdot |S^{l}_{n}|\cdot (2^{q_{n}})^{k-1}$
            \end{center}
            Thus we have 
            \begin{center}
                $\vspace{10px}\mu(\{\langle x_{0},x_{1},...,x_{k}\rangle\in(\Cantor)^{k+1}:\forall n\in\omega$ $\langle x_{0}|_{q_{n}},x_{1}|_{q_{n}},...,x_{k}|_{q_{n}}\rangle\in F^{x}_{n}\})\leq\newline\vspace{10px}\leq\mu(\{\langle x_{0},x_{1},...,x_{k}\rangle\in(\Cantor)^{k+1}:\forall n\in\omega$ $\langle x_{0}|_{q_{n}},x_{1}|_{q_{n}},...,x_{k}|_{q_{n}}\rangle\in F^{x}_{n}\})\leq\newline\leq\frac{|F^{x}_{n}(k)|}{(2^{q_{n}})^{k+1}}=\frac{(1+n)^{k}\cdot |F^{x}_{n}|\cdot (2^{q_{n}})^{k-1}}{(2^{q_{n}})^{k+1}}\leq\frac{(1+n)^{k}\cdot |F^{x}_{n}|}{2^{q_{n}}}\leq\frac{(1+n)^{k}}{2^{n}}$
            \end{center}
            and the last item goes to zero with $n$ going to infinity.
        \end{proof}
        We will now inductively construct sequences
        \begin{itemize}
            \item[--] $\langle z_{k}:k\in\omega\rangle\in(\Cantor)^{\omega}$,
            \item[--] $\{L_{k}:k\in\omega\}$, a partition of $L$ such that $|L_{k}|=\omega$ for $k\geq1$, and
            \item[--] $\{D_{k}\subseteq\Cantor:k\in\omega\}$ of sets of measure zero.
        \end{itemize}
        Before we proceed with the induction we will need the following notation. For any $k\in\omega$, $\langle x_{0},...,x_{k}\rangle\in(\Cantor)^{k+1}$ and $l\in\Cantor$ define the sets
            \begin{center}
                $F^{l}_{\langle x_{0},...,x_{k}\rangle}=\{\langle x_{k+1},x_{k+2},...\rangle\in(\Cantor)^{\omega\setminus (k+1)}:\langle x_{0},...,x_{k},x_{k+1},...\rangle\in F^{l}\}$,\\
                $C^{l}_{\langle x_{0},...,x_{k}\rangle}=\{x_{k+1}\in\Cantor:\mu(F^{l}_{\langle x_{0},...,x_{k},x_{k+1}\rangle})>0\}$
            \end{center}
            and
            \begin{center}
                $F_{\langle x_{0},...,x_{k}\rangle}=\{(l,\langle x_{k+1},...\rangle)\in\Cantor\times(\Cantor)^{\omega\setminus(k+1)}:\langle x_{k+1},x_{k+2},...\rangle\in F^{l}_{\langle x_{0},...,x_{k}\rangle}\}$,\\
                $C_{\langle x_{0},...,x_{k}\rangle}=\{(l,x_{k+1})\in\Cantor\times\Cantor:x_{k+1}\in C^{l}_{\langle x_{0},...,x_{k}\rangle}\}$
            \end{center}
            Observe that the set $C_{\langle x_{0},...,x_{k}\rangle}$ is an $F_{\sigma}$-set. This is because for any $F_{\sigma}$ subset of the plane the set of those vertical sections that are of measure zero, is an $F_{\sigma}$-set.\\

        Start the induction by applying Fubini's theorem to the set $C_{\langle\rangle}$ to find $z_{0}\in\Cantor$ such that the set $D_{0}=\{l\in\Cantor:z_{0}\in C_{\langle\rangle}\}$ is null (and $F_{\sigma}$ by the observation above). Then the set $L_{1}=L\cap D_{0}$ is countable as $L$ is $\E$-Luzin. Assume now that $z_{k}$, $L_{k+1}$ and $D_{k}$ have been constructed. Apply Fubini's theorem to the set $C_{\langle z_{0},...,z_{k}\rangle}$ (which is an $F_{\sigma}$-set in $\Cantor\times\Cantor$ with horizontal sections in $\E$) to find a point 
        \begin{center}
            $z_{k+1}\in\Cantor\setminus(\bigcup_{l\in L_{1}}C^{l}_{\langle z_{1},...,z_{k}\rangle}\cup\bigcup_{l\in L_{2}}C^{l}_{\langle z_{2},...,z_{k}\rangle}\cup...\cup\bigcup_{l\in L_{k}}C^{l}_{\langle z_{k}\rangle}\cup\bigcup_{l\in L_{k+1}}C^{l}_{\langle\rangle})$
        \end{center}
        such that the set $D_{k+1}$ defined as $\{l\in\Cantor:z_{k+1}\in C^{l}_{\langle z_{0},z_{1},...,z_{k}\rangle}\}$ is of measure zero (and $F_{\sigma}$ by the preceding observation). Let then $L_{k+2}=(L\setminus \bigcup_{i\leq k+1}L_{i})\cap D_{k+1}$, which is a countable set as $L$ is $\E$-Luzin. This finishes the induction step. When the whole induction is over let $L_{0}=L\setminus \bigcup_{k\geq 1}L_{k}$.\\
        Note that the induction was performed in a way that 
        \begin{center}
            $\forall i\in\omega$ $(l\in L_{i}\rightarrow\forall k\geq i$ $z_{k+1}\notin C^{l}_{\langle z_{i},...,z_{k}\rangle})$
        \end{center}
        \begin{claim}
            If $l\in L_{i}$ and $k\geq i$, then $\exists^{\infty}q\in Q$ $\langle z_{i}|_{q},...,z_{k+1}|_{q}\rangle\notin F^{l}_{q}(k+1-i)$
        \end{claim}
        \begin{proof}[proof of the claim]
            Suppose otherwise: that for some $i\leq k$ and $l\in L_{i}$ for almost all $q\in Q$ we have that $\langle z_{i}|_{q},...,z_{k+1}|_{q}\rangle\in F^{l}_{q}(k+1-i)$. This means that $\langle z_{i}|_{q},...,z_{k+1}|_{q}\rangle\in F^{l}(k+1-i)$. On the other hand we have that $F^{l}_{\langle z_{i},...,z_{k+1}\rangle}=(\Cantor)^{\omega\setminus(k+1-i)}$ which in particular means that $\mu(F^{l}_{\langle z_{i},...,z_{k}\rangle^{\frown}z_{k+1}})>0$. It follows that $z_{k+1}\in C^{l}_{\langle z_{i},...,z_{k}\rangle}$ which contradicts the preceding condition and thus proves the claim.
        \end{proof}
        Fix $l\in L_{i}$ and $i\leq k$. Let $Q^{l}_{k}\in[Q]^{\omega}$ be such that for all $q\in Q^{l}_{k}$ we have $\langle z_{i}|_{q},...,z_{k+1}|_{q}\rangle\notin F^{l}_{q}(k+1-i)$. Let $\{q^{l}_{k}:k\in\omega\}\in\Baire$ be such sequence that $q^{l}_{k}$ is minimal with $|Q^{l}_{k}\cap q^{l}_{k}|=k+1-i$ for each $k\in\omega$.\\
        Define then $f:L\rightarrow \Baire$ such that $f(l)(k)=q^{l}_{k}$. As the assignment $l\mapsto \{S^{l}_{k}:k\in\omega\}$ is continuous, the map $f$ is continuous as well. By the Lemma \ref{SierpińskiHASMenger} $f[L]$ is not a dominating family in $\Baire$ so there is a $\{q_{k}:k\in\omega\}\in\Baire$ such that for every $l\in L$ there are infinitely many $k\in\omega$ with $q^{l}_{k}\leq q_{k}$.\\
        Define now 
        \begin{center}
            $z=z_{0}|_{q_{0}}\cup z_{1}|_{[q_{0},q_{1})}\cup z_{2}|_{[q_{1},q_{2})}\cup...$
        \end{center}
        We will finish the proof by showing that $z\notin L+E$ so let $l\in L_{i}$ for some $i\in\omega$. It will suffice to show that the point $z'=z_{i}|_{q_{i}}\cup z_{i+1}|_{[q_{i},q_{i+1})}\cup...$ (which is a finite modification of $z$) is such that $z'|_{q}\notin S^{l}_{q}$ for infinitely many $q\in Q$. Let $A\in[\omega]^{\omega}$ be such that $q^{l}_{q}\leq q_{k}$ for every $k\in A$. For every $k\in A$ there is the biggest $q\leq q_{k}$, $q\in Q^{l}_{k}$. In particular for such $q\in Q^{l}_{k}\cap q_{k}$ we have $\langle z_{i}|_{q},...,z_{k+1}|_{q}\rangle\notin F^{l}_{q}(k+1-i)$ which means that $\langle z_{i}|_{q},...,z_{k+1}|_{q}\rangle$ does not have diagonal in $S^{l}_{q}$. In particular $z'|_{q}\notin S^{l}_{q}$ for every $k\in A$. This completes the proof of the theorem.
    \end{proof}
    We now know that every $\I$-Luzin set is in $\I^{*}$ for the ideals $\Meager$, $\Null$ and $\E$. It is natural to ask if this is the case for other known $\sigma$-ideals of the reals. This motivates the following question.
    \begin{question}
        Is every $\I$-Luzin in $\I^{*}$ for other $\sigma$-ideals $\I$ like $\sigma$-porosity ideal or the infinitely equal and eventually different ideals investigated in \cite{KhomskiiLaguzzi}. What about the $\sigma$-ideals of $\sigma$-compact (or not strongly dominating) subsets of $\mathbb{Z}^{\omega}$? Is there a natural example of $\sigma$-ideal $\I$, so that some $\I$-Luzin set is not $\I^{*}$?
    \end{question}

\begin{acknowledgements}
    I would like to thank Diego Mejía and Tomasz Weiss for reading an early version of the manuscript and many suggestions that improved the quality of the article.
\end{acknowledgements}

\bibliographystyle{alphadin} 
\bibliography{bib.bib}

@article {GalvMycielSolov,
    AUTHOR = {Galvin, Fred and Mycielski, Jan and Solovay, Robert M.},
     TITLE = {Strong measure zero and infinite games},
   JOURNAL = {Arch. Math. Logic},
  FJOURNAL = {Archive for Mathematical Logic},
    VOLUME = {56},
      YEAR = {2017},
    NUMBER = {7-8},
     PAGES = {725--732},
      ISSN = {0933-5846,1432-0665},
   MRCLASS = {03E15 (54G15 91A44)},
  MRNUMBER = {3696064},
MRREVIEWER = {Kandasamy\ Muthuvel},
       DOI = {10.1007/s00153-017-0541-z},
       URL = {https://doi.org/10.1007/s00153-017-0541-z},
}

@article {BartoszynskiBD,
    AUTHOR = {Bartoszy{\' n}ski, Tomek},
     TITLE = {Remarks on small sets of reals},
   JOURNAL = {Proc. Amer. Math. Soc.},
  FJOURNAL = {Proceedings of the American Mathematical Society},
    VOLUME = {131},
      YEAR = {2003},
    NUMBER = {2},
     PAGES = {625--630},
      ISSN = {0002-9939,1088-6826},
   MRCLASS = {03E17 (03E15)},
  MRNUMBER = {1933355},
MRREVIEWER = {Peter\ Elia\v s},
       DOI = {10.1090/S0002-9939-02-06567-X},
       URL = {https://doi.org/10.1090/S0002-9939-02-06567-X},
}

@article {PawlikowskiEN,
    AUTHOR = {Pawlikowski, Janusz},
     TITLE = {A characterization of strong measure zero sets},
   JOURNAL = {Israel J. Math.},
  FJOURNAL = {Israel Journal of Mathematics},
    VOLUME = {93},
      YEAR = {1996},
     PAGES = {171--183},
      ISSN = {0021-2172,1565-8511},
   MRCLASS = {28A05},
  MRNUMBER = {1380640},
MRREVIEWER = {Gabriel\ Debs},
       DOI = {10.1007/BF02761100},
       URL = {https://doi.org/10.1007/BF02761100},
}

@book {BJ,
    AUTHOR = {Bartoszy{\' n}ski, Tomek and Judah, Haim},
     TITLE = {Set theory},
      NOTE = {On the structure of the real line},
 PUBLISHER = {A K Peters, Ltd., Wellesley, MA},
      YEAR = {1995},
     PAGES = {xii+546},
      ISBN = {1-56881-044-X},
   MRCLASS = {03-02 (03Exx)},
  MRNUMBER = {1350295},
MRREVIEWER = {Eva\ Coplakova},
}

@article {ZindulkaMAEA,
    AUTHOR = {Zindulka, Ond{\v r}ej},
     TITLE = {Strong measure zero and meager-additive sets through the prism
              of fractal measures},
   JOURNAL = {Comment. Math. Univ. Carolin.},
  FJOURNAL = {Commentationes Mathematicae Universitatis Carolinae},
    VOLUME = {60},
      YEAR = {2019},
    NUMBER = {1},
     PAGES = {131--155},
      ISSN = {0010-2628,1213-7243},
   MRCLASS = {03E15 (03E05 03E20 28A78)},
  MRNUMBER = {3946667},
MRREVIEWER = {Arnold\ W.\ Miller},
       DOI = {10.14712/1213-7243.2015.277},
       URL = {https://doi.org/10.14712/1213-7243.2015.277},
}

@article {PawlikowskiSierpinski,
    AUTHOR = {Pawlikowski, Janusz},
     TITLE = {Every {S}ierpi{\'n}ski set is strongly meager},
   JOURNAL = {Arch. Math. Logic},
  FJOURNAL = {Archive for Mathematical Logic},
    VOLUME = {35},
      YEAR = {1996},
    NUMBER = {5-6},
     PAGES = {281--285},
      ISSN = {0933-5846,1432-0665},
   MRCLASS = {03E15 (28A05)},
  MRNUMBER = {1420258},
MRREVIEWER = {Hiroshi\ Fujita},
       DOI = {10.1007/s001530050045},
       URL = {https://doi.org/10.1007/s001530050045},
}

@book {Zapl1,
    AUTHOR = {Zapletal, Jind{\v r}ich},
     TITLE = {Forcing idealized},
    SERIES = {Cambridge Tracts in Mathematics},
    VOLUME = {174},
 PUBLISHER = {Cambridge University Press, Cambridge},
      YEAR = {2008},
     PAGES = {vi+314},
      ISBN = {978-0-521-87426-7},
   MRCLASS = {03-02 (03E15 03E35 03E40 28A05 54A35)},
  MRNUMBER = {2391923},
MRREVIEWER = {Miroslav\ Repick\'y},
       DOI = {10.1017/CBO9780511542732},
       URL = {https://doi.org/10.1017/CBO9780511542732},
}

@article {BartJudahBorelImages,
    AUTHOR = {Bartoszy{\' n}ski, Tomek and Judah, Haim},
     TITLE = {Borel images of sets of reals},
   JOURNAL = {Real Anal. Exchange},
  FJOURNAL = {Real Analysis Exchange},
    VOLUME = {20},
      YEAR = {1994/95},
    NUMBER = {2},
     PAGES = {536--558},
      ISSN = {0147-1937,1930-1219},
   MRCLASS = {04A20 (04A15 26A21)},
  MRNUMBER = {1348078},
MRREVIEWER = {Jakub\ Jasi\'nski},
}

@article {JudShWoodRandomVsBorelConj,
    AUTHOR = {Judah, Haim and Shelah, Saharon and Woodin, William Hugh},
     TITLE = {The {B}orel conjecture},
   JOURNAL = {Ann. Pure Appl. Logic},
  FJOURNAL = {Annals of Pure and Applied Logic},
    VOLUME = {50},
      YEAR = {1990},
    NUMBER = {3},
     PAGES = {255--269},
      ISSN = {0168-0072,1873-2461},
   MRCLASS = {03E35 (03E15)},
  MRNUMBER = {1086456},
MRREVIEWER = {Jakub\ Jasi\'nski},
       DOI = {10.1016/0168-0072(90)90058-A},
       URL = {https://doi.org/10.1016/0168-0072(90)90058-A},
}

@article {BorelWithDualBorel,
    AUTHOR = {Goldstern, Martin and Kellner, Jakob and Shelah, Saharon and
              Wohofsky, Wolfgang},
     TITLE = {Borel conjecture and dual {B}orel conjecture},
   JOURNAL = {Trans. Amer. Math. Soc.},
  FJOURNAL = {Transactions of the American Mathematical Society},
    VOLUME = {366},
      YEAR = {2014},
    NUMBER = {1},
     PAGES = {245--307},
      ISSN = {0002-9947,1088-6850},
   MRCLASS = {03E35 (03E17 28E15)},
  MRNUMBER = {3118397},
MRREVIEWER = {Szymon\ \.Zeberski},
       DOI = {10.1090/S0002-9947-2013-05783-2},
       URL = {https://doi.org/10.1090/S0002-9947-2013-05783-2},
}

@article {MillerAddEvsB,
    AUTHOR = {Miller, Arnold W.},
     TITLE = {Some properties of measure and category},
   JOURNAL = {Trans. Amer. Math. Soc.},
  FJOURNAL = {Transactions of the American Mathematical Society},
    VOLUME = {266},
      YEAR = {1981},
    NUMBER = {1},
     PAGES = {93--114},
      ISSN = {0002-9947,1088-6850},
   MRCLASS = {03E35 (03E15 28C15 54A35 54H05)},
  MRNUMBER = {613787},
MRREVIEWER = {John\ K.\ Truss},
       DOI = {10.2307/1998389},
       URL = {https://doi.org/10.2307/1998389},
}

@article {Carlson,
    AUTHOR = {Carlson, Timothy J.},
     TITLE = {Strong measure zero and strongly meager sets},
   JOURNAL = {Proc. Amer. Math. Soc.},
  FJOURNAL = {Proceedings of the American Mathematical Society},
    VOLUME = {118},
      YEAR = {1993},
    NUMBER = {2},
     PAGES = {577--586},
      ISSN = {0002-9939,1088-6826},
   MRCLASS = {03E15 (03E35 28A05)},
  MRNUMBER = {1139474},
MRREVIEWER = {Thomas\ J.\ Jech},
       DOI = {10.2307/2160341},
       URL = {https://doi.org/10.2307/2160341},
}

@article {PawlikowskiDualBorel,
    AUTHOR = {Pawlikowski, Janusz},
     TITLE = {Finite support iteration and strong measure zero sets},
   JOURNAL = {J. Symbolic Logic},
  FJOURNAL = {The Journal of Symbolic Logic},
    VOLUME = {55},
      YEAR = {1990},
    NUMBER = {2},
     PAGES = {674--677},
      ISSN = {0022-4812,1943-5886},
   MRCLASS = {03E35 (03E15 03E50 28A05)},
  MRNUMBER = {1056381},
MRREVIEWER = {Andrzej\ Szyma\'nski},
       DOI = {10.2307/2274657},
       URL = {https://doi.org/10.2307/2274657},
}

@article {LaverBorelConj,
    AUTHOR = {Laver, Richard},
     TITLE = {On the consistency of {B}orel's conjecture},
   JOURNAL = {Acta Math.},
  FJOURNAL = {Acta Mathematica},
    VOLUME = {137},
      YEAR = {1976},
    NUMBER = {3-4},
     PAGES = {151--169},
      ISSN = {0001-5962,1871-2509},
   MRCLASS = {04A15 (28A05)},
  MRNUMBER = {422027},
MRREVIEWER = {John\ P.\ Burgess},
       DOI = {10.1007/BF02392416},
       URL = {https://doi.org/10.1007/BF02392416},
}

@article {PawlikowskiTransitiveInvariants,
    AUTHOR = {Pawlikowski, Janusz},
     TITLE = {Powers of transitive bases of measure and category},
   JOURNAL = {Proc. Amer. Math. Soc.},
  FJOURNAL = {Proceedings of the American Mathematical Society},
    VOLUME = {93},
      YEAR = {1985},
    NUMBER = {4},
     PAGES = {719--729},
      ISSN = {0002-9939,1088-6826},
   MRCLASS = {03E15 (03E35 04A15 28A05 54A25)},
  MRNUMBER = {776210},
MRREVIEWER = {James\ Baumgartner},
       DOI = {10.2307/2045552},
       URL = {https://doi.org/10.2307/2045552},
}

@article {BartShClosedMeasureZeroSets,
    AUTHOR = {Bartoszy{\' n}ski, Tomek and Shelah, Saharon},
     TITLE = {Closed measure zero sets},
   JOURNAL = {Ann. Pure Appl. Logic},
  FJOURNAL = {Annals of Pure and Applied Logic},
    VOLUME = {58},
      YEAR = {1992},
    NUMBER = {2},
     PAGES = {93--110},
      ISSN = {0168-0072,1873-2461},
   MRCLASS = {03E05 (03E35 04A20)},
  MRNUMBER = {1186905},
MRREVIEWER = {L.\ Bukovsk\'y},
       DOI = {10.1016/0168-0072(92)90001-G},
       URL = {https://doi.org/10.1016/0168-0072(92)90001-G},
}

@incollection {SpecialSubsets,
    AUTHOR = {Miller, Arnold W.},
     TITLE = {Special subsets of the real line},
 BOOKTITLE = {Handbook of set-theoretic topology},
     PAGES = {201--233},
 PUBLISHER = {North-Holland, Amsterdam},
      YEAR = {1984},
      ISBN = {0-444-86580-2},
   MRCLASS = {54H05 (04A15 28A05 54C50)},
  MRNUMBER = {776624},
MRREVIEWER = {H.\ Sarbadhikari},
}

@article {KillingLuzinSierp,
    AUTHOR = {Judah, Haim and Shelah, Saharon},
     TITLE = {Killing {L}uzin and {S}ierpi{\'n}ski sets},
   JOURNAL = {Proc. Amer. Math. Soc.},
  FJOURNAL = {Proceedings of the American Mathematical Society},
    VOLUME = {120},
      YEAR = {1994},
    NUMBER = {3},
     PAGES = {917--920},
      ISSN = {0002-9939,1088-6826},
   MRCLASS = {03E15 (03E35)},
  MRNUMBER = {1164145},
MRREVIEWER = {Eva\ Coplakova},
       DOI = {10.2307/2160487},
       URL = {https://doi.org/10.2307/2160487},
}

@article {NowikWeiss,
    AUTHOR = {Nowik, Andrzej and Weiss, Tomasz},
     TITLE = {On the {R}amseyan properties of some special subsets of
              {$2^\omega$} and their algebraic sums},
   JOURNAL = {J. Symbolic Logic},
  FJOURNAL = {The Journal of Symbolic Logic},
    VOLUME = {67},
      YEAR = {2002},
    NUMBER = {2},
     PAGES = {547--556},
      ISSN = {0022-4812,1943-5886},
   MRCLASS = {03E15 (28E15)},
  MRNUMBER = {1905154},
MRREVIEWER = {Miroslav\ Repick\'y},
       DOI = {10.2178/jsl/1190150097},
       URL = {https://doi.org/10.2178/jsl/1190150097},
}

@article {KysiakNowikWeiss,
    AUTHOR = {Kysiak, Marcin and Nowik, Andrzej and Weiss, Tomasz},
     TITLE = {Special subsets of the reals and tree forcing notions},
   JOURNAL = {Proc. Amer. Math. Soc.},
  FJOURNAL = {Proceedings of the American Mathematical Society},
    VOLUME = {135},
      YEAR = {2007},
    NUMBER = {9},
     PAGES = {2975--2982},
      ISSN = {0002-9939,1088-6826},
   MRCLASS = {03E40 (03E05 03E35 03E50 28E15 54G15)},
  MRNUMBER = {2317976},
MRREVIEWER = {J\"org\ D.\ Brendle},
       DOI = {10.1090/S0002-9939-07-08808-9},
       URL = {https://doi.org/10.1090/S0002-9939-07-08808-9},
}

@article {ZindulkaMeagerAdditiveinTopGroups,
    AUTHOR = {Zindulka, Ond{\v r}ej},
     TITLE = {Meager-additive sets in topological groups},
   JOURNAL = {J. Symb. Log.},
  FJOURNAL = {The Journal of Symbolic Logic},
    VOLUME = {87},
      YEAR = {2022},
    NUMBER = {3},
     PAGES = {1046--1064},
      ISSN = {0022-4812,1943-5886},
   MRCLASS = {22B05 (03E17 22A10)},
  MRNUMBER = {4472524},
MRREVIEWER = {Su\ Gao},
       DOI = {10.1017/jsl.2021.79},
       URL = {https://doi.org/10.1017/jsl.2021.79},
}

@article {KadaLaverNumber,
    AUTHOR = {Kada, Masaru},
     TITLE = {More on {C}icho{\'n}'s diagram and infinite games},
   JOURNAL = {J. Symbolic Logic},
  FJOURNAL = {The Journal of Symbolic Logic},
    VOLUME = {65},
      YEAR = {2000},
    NUMBER = {4},
     PAGES = {1713--1724},
      ISSN = {0022-4812,1943-5886},
   MRCLASS = {03E05 (91A13)},
  MRNUMBER = {1812176},
MRREVIEWER = {Otmar\ Spinas},
       DOI = {10.2307/2695071},
       URL = {https://doi.org/10.2307/2695071},
}

@article {BartoszynskiCoveringNull,
    AUTHOR = {Bartoszy{\'n}ski, Tomek},
     TITLE = {On covering of real line by null sets},
   JOURNAL = {Pacific J. Math.},
  FJOURNAL = {Pacific Journal of Mathematics},
    VOLUME = {131},
      YEAR = {1988},
    NUMBER = {1},
     PAGES = {1--12},
      ISSN = {0030-8730,1945-5844},
   MRCLASS = {03E05 (03E15 04A15 26A03)},
  MRNUMBER = {917862},
MRREVIEWER = {Kurt\ Wolfsdorf},
       URL = {http://projecteuclid.org/euclid.pjm/1102690066},
}

@article {SlalomNumbersSurvey,
    AUTHOR = {Cardona, Miguel A. and Gavalov{\'a}, Viera and Mej{\'i}a, Diego A.
              and Repick{\'y}, Miroslav and {\v S}upina, Jaroslav},
     TITLE = {Slalom numbers},
   JOURNAL = {Dissertationes Math.},
  FJOURNAL = {Dissertationes Mathematicae},
    VOLUME = {607},
      YEAR = {2026},
     PAGES = {67},
      ISSN = {0012-3862,1730-6310},
   MRCLASS = {03E17 (03E35 54D20 54G15)},
  MRNUMBER = {5027593},
       DOI = {10.4064/dm240826-29-5},
       URL = {https://doi.org/10.4064/dm240826-29-5},
}

@article {BrendleLoweEventDiff,
    AUTHOR = {Brendle, J{\"o}rg and L{\"o}we, Benedikt},
     TITLE = {Eventually different functions and inaccessible cardinals},
   JOURNAL = {J. Math. Soc. Japan},
  FJOURNAL = {Journal of the Mathematical Society of Japan},
    VOLUME = {63},
      YEAR = {2011},
    NUMBER = {1},
     PAGES = {137--151},
      ISSN = {0025-5645,1881-1167},
   MRCLASS = {03E15 (03E40 03E55)},
  MRNUMBER = {2752434},
MRREVIEWER = {Miroslav Repick\'y},
       URL = {http://projecteuclid.org/euclid.jmsj/1296138346},
}

@article {BreWoh,
    AUTHOR = {Brendle, J{\"o}rg and Wohofsky, Wolfgang},
     TITLE = {Borel conjecture for the {M}arczewski ideal},
   JOURNAL = {Proc. Amer. Math. Soc.},
  FJOURNAL = {Proceedings of the American Mathematical Society},
    VOLUME = {152},
      YEAR = {2024},
    NUMBER = {12},
     PAGES = {5395--5410},
      ISSN = {0002-9939,1088-6826},
   MRCLASS = {03E05 (03E15 03E17 03E35 03E50 22A05 54H11)},
  MRNUMBER = {4855892},
MRREVIEWER = {Andrzej Ros\l anowski},
       DOI = {10.1090/proc/16981},
       URL = {https://doi.org/10.1090/proc/16981},
}

@article {Tala,
    AUTHOR = {Talagrand, Michel},
     TITLE = {Compacts de fonctions mesurables et filtres non mesurables},
   JOURNAL = {Studia Math.},
  FJOURNAL = {Polska Akademia Nauk. Instytut Matematyczny. Studia
              Mathematica},
    VOLUME = {67},
      YEAR = {1980},
    NUMBER = {1},
     PAGES = {13--43},
      ISSN = {0039-3223,1730-6337},
   MRCLASS = {28A20 (46E27 54C50)},
  MRNUMBER = {579439},
MRREVIEWER = {J.\ D.\ Knowles},
       DOI = {10.4064/sm-67-1-13-43},
       URL = {https://doi.org/10.4064/sm-67-1-13-43},
}

@incollection {Blass,
    AUTHOR = {Blass, Andreas},
     TITLE = {Combinatorial cardinal characteristics of the continuum},
 BOOKTITLE = {Handbook of set theory. {V}ols. 1, 2, 3},
     PAGES = {395--489},
 PUBLISHER = {Springer, Dordrecht},
      YEAR = {2010},
      ISBN = {978-1-4020-4843-2},
   MRCLASS = {03E17 (03E35)},
  MRNUMBER = {2768685},
MRREVIEWER = {Vera\ Fischer},
       DOI = {10.1007/978-1-4020-5764-9\_7},
       URL = {https://doi.org/10.1007/978-1-4020-5764-9_7},
}

@article {KhomskiiLaguzzi,
    AUTHOR = {Khomskii, Yurii and Laguzzi, Giorgio},
     TITLE = {Full-splitting {M}iller trees and infinitely often equal
              reals},
   JOURNAL = {Ann. Pure Appl. Logic},
  FJOURNAL = {Annals of Pure and Applied Logic},
    VOLUME = {168},
      YEAR = {2017},
    NUMBER = {8},
     PAGES = {1491--1506},
      ISSN = {0168-0072,1873-2461},
   MRCLASS = {03E15 (03E17 03E35 03E75)},
  MRNUMBER = {3650350},
MRREVIEWER = {Shuguo\ Zhang},
       DOI = {10.1016/j.apal.2017.02.001},
       URL = {https://doi.org/10.1016/j.apal.2017.02.001},
}

@article {Calderon,
    AUTHOR = {Calder{\'o}n, Daniel},
     TITLE = {Borel's conjecture and meager-additive sets},
   JOURNAL = {Proc. Amer. Math. Soc.},
  FJOURNAL = {Proceedings of the American Mathematical Society},
    VOLUME = {149},
      YEAR = {2021},
    NUMBER = {11},
     PAGES = {4943--4954},
      ISSN = {0002-9939,1088-6826},
   MRCLASS = {03E35 (03E15 03E75 54F65)},
  MRNUMBER = {4310117},
MRREVIEWER = {Riccardo\ Camerlo},
       DOI = {10.1090/proc/15536},
       URL = {https://doi.org/10.1090/proc/15536},
}

@article {coveringstrongmeasurecanbeaboveeverythingelse,
    AUTHOR = {Cardona, Miguel A. and Mej{\'i}a, Diego A. and Rivera-Madrid,
              Ismael E.},
     TITLE = {The covering number of the strong measure zero ideal can be
              above almost everything else},
   JOURNAL = {Arch. Math. Logic},
  FJOURNAL = {Archive for Mathematical Logic},
    VOLUME = {61},
      YEAR = {2022},
    NUMBER = {5-6},
     PAGES = {599--610},
      ISSN = {0933-5846,1432-0665},
   MRCLASS = {03E17 (03E35 03E40)},
  MRNUMBER = {4452119},
MRREVIEWER = {Andrzej\ Ros\l anowski},
       DOI = {10.1007/s00153-021-00808-0},
       URL = {https://doi.org/10.1007/s00153-021-00808-0},
}

@article {Evasionprediction,
    AUTHOR = {Brendle, J{\"o}rg},
     TITLE = {Evasion and prediction---the {S}pecker phenomenon and {G}ross
              spaces},
   JOURNAL = {Forum Math.},
  FJOURNAL = {Forum Mathematicum},
    VOLUME = {7},
      YEAR = {1995},
    NUMBER = {5},
     PAGES = {513--541},
      ISSN = {0933-7741,1435-5337},
   MRCLASS = {03E05 (03E35 15A63 20A10)},
  MRNUMBER = {1346879},
MRREVIEWER = {Marion\ Scheepers},
       DOI = {10.1515/form.1995.7.513},
       URL = {https://doi.org/10.1515/form.1995.7.513},
}

@article {ShelahNAMA,
    AUTHOR = {Shelah, Saharon},
     TITLE = {Every null additive set of reals is meager additive},
   JOURNAL = {Israel Journal of Mathematics},
  FJOURNAL = {Israel Journal of Mathematics},
    VOLUME = {89},
      YEAR = {1995},
    NUMBER = {},
     PAGES = {357–376},
      ISSN = {},
   MRCLASS = {},
  MRNUMBER = {},
MRREVIEWER = {},
       DOI = {},
       URL = {},
}

@article {MejiaCardonaMA,
    AUTHOR = {Mej{\'i}a, Diego A. and Cardona, Miguel A. and Rivera-Madrid,
              Ismael E.},
     TITLE = {Uniformity numbers of the null-additive and meager-additive ideals},
   JOURNAL = {J. Symbolic Logic},
  FJOURNAL = {The Journal of Symbolic Logic},
    VOLUME = {},
      YEAR = {2023},
    NUMBER = {},
     PAGES = {1–37},
      ISSN = {},
   MRCLASS = {},
  MRNUMBER = {},
MRREVIEWER = {},
       DOI = {10.1017/jsl.2025.10105},
       URL = {},
}

@article {CardonaMA,
    AUTHOR = {Cardona, Miguel A.},
     TITLE = {Cardinal invariants associated with the combinatorics of the uniformity number of the ideal of meager-additive sets},
   JOURNAL = {},
  FJOURNAL = {},
    VOLUME = {},
      YEAR = {2025},
    NUMBER = {},
     PAGES = {},
      ISSN = {},
   MRCLASS = {},
  MRNUMBER = {},
MRREVIEWER = {},
       DOI = {},
       URL = {},
}

@article {BartShNstarIsNotIdeal,
    AUTHOR = {Bartoszy{\'n}ski, Tomek and Shelah, Saharon},
     TITLE = {Strongly Meager Sest Do Not Form An Ideal},
   JOURNAL = {Journal of Mathematical Logic},
  FJOURNAL = {Journal of Mathematical Logic},
    VOLUME = {01},
      YEAR = {2001},
    NUMBER = {01},
     PAGES = {1-34},
      ISSN = {},
   MRCLASS = {},
  MRNUMBER = {},
MRREVIEWER = {},
       DOI = {},
       URL = {},
}

@article {Cieslak,
    AUTHOR = {Cie{\'s}lak, Aleksander},
     TITLE = {Combinatorics of {$\sigma$}-ideals generated by closed sets},
   JOURNAL = {in preparation},
  FJOURNAL = {},
    VOLUME = {},
      YEAR = {2026},
    NUMBER = {},
     PAGES = {},
      ISSN = {},
   MRCLASS = {},
  MRNUMBER = {},
MRREVIEWER = {},
       DOI = {},
       URL = {},
}

@article {WohoPHD,
    AUTHOR = {Wohofsky, Wolfgang},
     TITLE = {Special sets of real numbers and variants of the Borel Conjecture},
   JOURNAL = {},
  FJOURNAL = {},
    VOLUME = {},
      YEAR = {2013},
    NUMBER = {},
     PAGES = {},
      ISSN = {},
   MRCLASS = {},
  MRNUMBER = {},
MRREVIEWER = {},
       DOI = {},
       URL ={https://www.wohofsky.eu/math/PhD_thesis_Wolfgang_Wohofsky.pdf},
}

@article {MejiaDirectedSums,
    AUTHOR = {Cardona, Miguel A. and Mej{\'i}a, Diego A. and Ismael, Rivera-Madrid E.},
     TITLE = {Directed schemes of ideals and cardinal characteristics I: the meager additive ideal},
   JOURNAL = {},
  FJOURNAL = {},
    VOLUME = {},
      YEAR = {2025},
    NUMBER = {},
     PAGES = {},
      ISSN = {},
   MRCLASS = {},
  MRNUMBER = {},
MRREVIEWER = {},
       DOI = {},
       URL = {},
}

@article {SlalomNumbersOnReals,
    AUTHOR = {Cardona, Miguel A. and Mej{\'i}a, Diego A.},
     TITLE = {Localization and anti-localization cardinals},
   JOURNAL = {},
  FJOURNAL = {},
    VOLUME = {},
      YEAR = {2023},
    NUMBER = {},
     PAGES = {},
      ISSN = {},
   MRCLASS = {},
  MRNUMBER = {},
MRREVIEWER = {},
       DOI = {https://doi.org/10.48550/arXiv.2305.03248},
       URL = {},
}

@article {BrendleBetween,
    AUTHOR = {Brendle, J{\"o}rg},
     TITLE = {Between {$P$}-points and nowhere dense ultrafilters},
   JOURNAL = {Israel J. Math.},
  FJOURNAL = {Israel Journal of Mathematics},
    VOLUME = {113},
      YEAR = {1999},
     PAGES = {205--230},
      ISSN = {0021-2172,1565-8511},
   MRCLASS = {03E17 (03E05 03E35)},
  MRNUMBER = {1729447},
MRREVIEWER = {Miroslav\ Repick\'y},
       DOI = {10.1007/BF02780177},
       URL = {https://doi.org/10.1007/BF02780177},
}

@article {BrendleCardonaMejia,
    AUTHOR = {Brendle, J{\"o}rg and Cardona, Miguel A. and Mej{\'i}a, Diego A.},
     TITLE = {Separating cardinal characteristics of the strong measure zero ideal},
   JOURNAL = {Journal of Mathematical Logic},
  FJOURNAL = {Journal of Mathematical Logic},
    VOLUME = {},
      YEAR = {2023},
     PAGES = {},
      ISSN = {},
   MRCLASS = {03E17 03E35 03E40},
  MRNUMBER = {},
MRREVIEWER = {},
       DOI = {10.1142/s0219061325500126},
       URL = {},
}

@article {CardonaMejiaMoreOnCofCovSMZ,
    AUTHOR = {Cardona, Miguel A. and Mej{\'i}a, Diego A.},
     TITLE = {More about the cofinality and the covering of the ideal of
              strong measure zero sets},
   JOURNAL = {Ann. Pure Appl. Logic},
  FJOURNAL = {Annals of Pure and Applied Logic},
    VOLUME = {176},
      YEAR = {2025},
    NUMBER = {4},
     PAGES = {Paper No. 103537, 31},
      ISSN = {0168-0072,1873-2461},
   MRCLASS = {03E17 (03E10 03E35 03E40)},
  MRNUMBER = {4834901},
MRREVIEWER = {Andrzej\ Ros\l anowski},
       DOI = {10.1016/j.apal.2024.103537},
       URL = {https://doi.org/10.1016/j.apal.2024.103537},
}

@article {Weiss2013,
    AUTHOR = {Weiss, Tomasz},
     TITLE = {A note on the intersection ideal {$\mathcal{M}\cap\mathcal{N}$}},
   JOURNAL = {Comment. Math. Univ. Carolin.},
  FJOURNAL = {Commentationes Mathematicae Universitatis Carolinae},
    VOLUME = {54},
      YEAR = {2013},
    NUMBER = {3},
     PAGES = {437--445},
      ISSN = {0010-2628,1213-7243},
   MRCLASS = {03E05 (03E17)},
  MRNUMBER = {3090421},
MRREVIEWER = {Shuguo\ Zhang},
}

@article {Weiss2018,
    AUTHOR = {Weiss, Tomasz},
     TITLE = {More remarks on the intersection ideal {$\mathcal{M}\cap\mathcal{N}$}},
   JOURNAL = {Comment. Math. Univ. Carolin.},
  FJOURNAL = {Commentationes Mathematicae Universitatis Carolinae},
    VOLUME = {59},
      YEAR = {2018},
    NUMBER = {3},
     PAGES = {311--316},
      ISSN = {0010-2628,1213-7243},
   MRCLASS = {03E05 (03E17)},
  MRNUMBER = {3861554},
MRREVIEWER = {Arnold\ W.\ Miller},
       DOI = {10.14712/1213-7243.2015.252},
       URL = {https://doi.org/10.14712/1213-7243.2015.252},
}

@article {WeissNew,
    AUTHOR = {Weiss, Tomasz},
     TITLE = {On the class on {$(\mathcal
     {E},\mathcal{M})*$} sets and their relatives},
   JOURNAL = {in preparation},
  FJOURNAL = {},
    VOLUME = {},
      YEAR = {2026},
    NUMBER = {},
     PAGES = {},
      ISSN = {},
   MRCLASS = {},
  MRNUMBER = {},
MRREVIEWER = {},
       DOI = {},
       URL = {},
}

@article {WeissSzewczak,
    AUTHOR = {Szewczak, Piotr and Weiss, Tomasz},
     TITLE = {Null sets and combinatorial covering properties},
   JOURNAL = {J. Symb. Log.},
  FJOURNAL = {The Journal of Symbolic Logic},
    VOLUME = {87},
      YEAR = {2022},
    NUMBER = {3},
     PAGES = {1231--1242},
      ISSN = {0022-4812,1943-5886},
   MRCLASS = {03E35 (03E75 54D20)},
  MRNUMBER = {4472532},
MRREVIEWER = {Klaas\ Pieter\ Hart},
       DOI = {10.1017/jsl.2021.51},
       URL = {https://doi.org/10.1017/jsl.2021.51},
}

@article {CardonaMejiaOnYoryoka,
    AUTHOR = {Cardona, Miguel A. and Mej{\'i}a, Diego A.},
     TITLE = {On cardinal characteristics of {Y}orioka ideals},
   JOURNAL = {MLQ Math. Log. Q.},
  FJOURNAL = {MLQ. Mathematical Logic Quarterly},
    VOLUME = {65},
      YEAR = {2019},
    NUMBER = {2},
     PAGES = {170--199},
      ISSN = {0942-5616,1521-3870},
   MRCLASS = {03E17 (03E35)},
  MRNUMBER = {4019620},
MRREVIEWER = {Andrzej\ Ros\l anowski},
       DOI = {10.1002/malq.201800034},
       URL = {https://doi.org/10.1002/malq.201800034},
}

@article {ElekesSteprans,
    AUTHOR = {Elekes, M{\'a}rton and Stepr{\=a}ns, Juris},
     TITLE = {Less than {$2^\omega$} many translates of a compact nullset
              may cover the real line},
   JOURNAL = {Fund. Math.},
  FJOURNAL = {Fundamenta Mathematicae},
    VOLUME = {181},
      YEAR = {2004},
    NUMBER = {1},
     PAGES = {89--96},
      ISSN = {0016-2736,1730-6329},
   MRCLASS = {28E15 (03E17 03E35)},
  MRNUMBER = {2071696},
MRREVIEWER = {Peter\ Elia\v s},
       DOI = {10.4064/fm181-1-4},
       URL = {https://doi.org/10.4064/fm181-1-4},
}

@article {LyubomyrSzewczak,
    AUTHOR = {Haberl, Valentin and Szewczak, Piotr and Zdomskyy, Lyubomyr},
     TITLE = {Universally meager sets in the Miller model and similar ones},
   JOURNAL = {},
  FJOURNAL = {},
    VOLUME = {},
      YEAR = {2025},
    NUMBER = {},
     PAGES = {},
      ISSN = {},
   MRCLASS = {},
  MRNUMBER = {},
MRREVIEWER = {},
       DOI = {10.48550/arXiv.2512.15490},
       URL = {https://arxiv.org/abs/2512.15490},
}

@article {ConcentratedGammaSetsInMillersModel,
    AUTHOR = {Haberl, Valentin and Szewczak, Piotr and Zdomskyy, Lyubomyr},
     TITLE = {Concentrated sets and {$\gamma$}-sets in the {M}iller model},
   JOURNAL = {Topology Appl.},
  FJOURNAL = {Topology and its Applications},
    VOLUME = {379},
      YEAR = {2026},
     PAGES = {Paper No. 109503, 10},
      ISSN = {0166-8641,1879-3207},
   MRCLASS = {03E35 (03E75 54D20)},
  MRNUMBER = {5015025},
       DOI = {10.1016/j.topol.2025.109503},
       URL = {https://doi.org/10.1016/j.topol.2025.109503},
}

@article {ScalesProductsScheepersDiagram,
    AUTHOR = {Pawlikowski, Michał and Szewczak, Piotr and Zdomskyy, Lyubomyr},
     TITLE = {Scales, products and the second row of the Scheepers diagram},
   JOURNAL = {},
  FJOURNAL = {},
    VOLUME = {},
      YEAR = {2025},
    NUMBER = {},
     PAGES = {},
      ISSN = {},
   MRCLASS = {},
  MRNUMBER = {},
MRREVIEWER = {},
       DOI = {10.48550/arXiv.2503.18615},
       URL = {https://arxiv.org/abs/2503.18615},
}

@incollection {SelectionPrinciplesLaverMillerSacksModels,
    AUTHOR = {Zdomskyy, Lyubomyr},
     TITLE = {Selection principles in the {L}aver, {M}iller, and {S}acks
              models},
 BOOKTITLE = {Centenary of the {B}orel conjecture},
    SERIES = {Contemp. Math.},
    VOLUME = {755},
     PAGES = {229--242},
 PUBLISHER = {Amer. Math. Soc., [Providence], RI},
      YEAR = {[2020] \copyright 2020},
      ISBN = {978-1-4704-5099-1},
   MRCLASS = {03E35 (03E05 54C50 54D20)},
  MRNUMBER = {4146586},
MRREVIEWER = {Klaas\ Pieter\ Hart},
       DOI = {10.1090/conm/755/15176},
       URL = {https://doi.org/10.1090/conm/755/15176},
}

@article {RepovsZdomskyy,
    AUTHOR = {Repov{\v s}, Du{\v s}an and Zdomskyy, Lyubomyr},
     TITLE = {{$M$}-separable spaces of functions are productive in the
              {M}iller model},
   JOURNAL = {Ann. Pure Appl. Logic},
  FJOURNAL = {Annals of Pure and Applied Logic},
    VOLUME = {171},
      YEAR = {2020},
    NUMBER = {7},
     PAGES = {102806, 8},
      ISSN = {0168-0072,1873-2461},
   MRCLASS = {03E35 (03E05 54C50 54D20)},
  MRNUMBER = {4099834},
MRREVIEWER = {Andrzej\ Ros\l anowski},
       DOI = {10.1016/j.apal.2020.102806},
       URL = {https://doi.org/10.1016/j.apal.2020.102806},
}

@article {PreservationGammaSpaces,
    AUTHOR = {Repov{\v s}, Du{\v s}an and Zdomskyy, Lyubomyr},
     TITLE = {Preservation of {$\gamma$}-spaces and covering properties of
              products},
   JOURNAL = {Proc. Amer. Math. Soc.},
  FJOURNAL = {Proceedings of the American Mathematical Society},
    VOLUME = {147},
      YEAR = {2019},
    NUMBER = {11},
     PAGES = {4979--4985},
      ISSN = {0002-9939,1088-6826},
   MRCLASS = {03E35 (03E05 54C50 54D20)},
  MRNUMBER = {4011529},
MRREVIEWER = {Sheldon\ W.\ Davis},
       DOI = {10.1090/proc/14593},
       URL = {https://doi.org/10.1090/proc/14593},
}

@article {ProductsOfMengerInMillerModel,
    AUTHOR = {Zdomskyy, Lyubomyr},
     TITLE = {Products of {M}enger spaces in the {M}iller model},
   JOURNAL = {Adv. Math.},
  FJOURNAL = {Advances in Mathematics},
    VOLUME = {335},
      YEAR = {2018},
     PAGES = {170--179},
      ISSN = {0001-8708,1090-2082},
   MRCLASS = {03E35 (03E05 54C50 54D20)},
  MRNUMBER = {3836661},
MRREVIEWER = {Samuel\ Gomes\ da Silva},
       DOI = {10.1016/j.aim.2018.06.016},
       URL = {https://doi.org/10.1016/j.aim.2018.06.016},
}

@article {GalvinMiller,
    AUTHOR = {Galvin, Fred and Miller, Arnold W.},
     TITLE = {$\gamma$-sets and other singular sets of real numbers},
   JOURNAL = {Topology Appl.},
  FJOURNAL = {Topology and its Applications},
    VOLUME = {17},
      YEAR = {1984},
     PAGES = {145-155},
      ISSN = {},
   MRCLASS = {},
  MRNUMBER = {},
MRREVIEWER = {},
       DOI = {10.1016/0166-8641(84)90038-5},
       URL = {https://doi.org/10.1016/0166-8641(84)90038-5},
}

@article {BartoszCombAspects,
    AUTHOR = {Bartoszy{\'n}ski, Tomek},
     TITLE = {Combinatorial aspects of measure and category},
   JOURNAL = {Fund. Math.},
  FJOURNAL = {Polska Akademia Nauk. Fundamenta Mathematicae},
    VOLUME = {127},
      YEAR = {1987},
    NUMBER = {3},
     PAGES = {225--239},
      ISSN = {0016-2736,1730-6329},
   MRCLASS = {04A15 (03E15 28A05 54H05)},
  MRNUMBER = {917147},
MRREVIEWER = {J.\ S.\ Lipi\'nski},
       DOI = {10.4064/fm-127-3-225-239},
       URL = {https://doi.org/10.4064/fm-127-3-225-239},
}

\end{document}